\documentclass[11pt]{article}

\usepackage{amsmath,amssymb,amsfonts,xy,enumitem,amsthm,changepage,lscape,rotating,floatpag}
\xyoption{all}
\usepackage{pstricks}

\floatpagestyle{empty}

\newenvironment{mywidth}{\begin{adjustwidth}{.5cm}{}}{\end{adjustwidth}}
\newenvironment{mywidth2}{\begin{adjustwidth}{.7cm}{}}{\end{adjustwidth}}

\numberwithin{equation}{section}
\newtheorem*{thm*}{Theorem}
\newtheorem{mythm}{Theorem}
\newtheorem{thm}{Theorem}[section]
\newtheorem{prp}[thm]{Proposition}
\newtheorem{lmm}[thm]{Lemma}   
\newtheorem{crl}[thm]{Corollary} 
 
\theoremstyle{definition}

\newtheorem{rmk}[thm]{Remark}

\def\eref#1{(\ref{#1})}
\def\BE#1{\begin{equation}\label{#1}}
\def\EE{\end{equation}}

\def\lr#1{\langle{#1}\rangle}
\def\ti#1{\tilde{#1}}
\def\wt#1{\widetilde{#1}}

\def\sf#1{\textsf{#1}}
\def\ch#1{\check{#1}}

\def\lra{\longrightarrow}
\def\llra{\longleftrightarrow}

\def\cB{\mathcal B}
\def\C{\mathbb C}
\def\cC{\mathcal C}
\def\fc{\mathfrak c}
\def\fd{\mathfrak d}
\def\cD{\mathcal D}
\def\cF{\mathcal F}

\def\cL{\mathcal L}

\def\cI{\mathcal I}
\def\fI{\mathfrak i}
\def\fJ{\mathfrak j}
\def\cK{\mathcal K}
\def\cL{\mathcal L}

\def\cP{\mathcal P}
\def\R{\mathbb R}
\def\cR{\mathcal R}
\def\fs{\mathfrak s}
\def\cT{\mathcal T}
\def\ft{\mathfrak t}

\def\Z{\mathbb Z}

\def\al{\alpha}
\def\be{\beta}

\def\de{\delta}
\def\ep{\epsilon}
\def\io{\iota}
\def\ka{\kappa}
\def\la{\lambda}

\def\si{\sigma}

\def\La{\Lambda}
\def\Om{\Omega}

\def\Th{\Theta}

\def\OmN{\Omega_N}

\def\i{\infty}

\def\w{\wedge}

\def\tnB{\textnormal{B}}
\def\tnC{\textnormal{C}}
\def\tnL{\textnormal{L}}
\def\tnM{\textnormal{M}}
\def\tnR{\textnormal{R}}
\def\tnT{\textnormal{T}}

\def\dom{\textnormal{dom}}
\def\Hom{\textnormal{Hom}}
\def\id{\textnormal{id}}
\def\ind{\textnormal{ind}}
\def\Im{\textnormal{Im}}
\def\Re{\textnormal{Re}}
\def\top{\textnormal{top}}

\begin{document}

\author{Aleksey Zinger\thanks{Partially supported by 
NSF grants DMS-0846978 and DMS-1901979
and IAS Fund for Math}}
\title{The Determinant Line Bundle for Fredholm Operators:
Construction, Properties, and Classification}
\date{\small{{\it Updated} \today}}
\maketitle

\begin{abstract}
\noindent
We provide a thorough construction of a system of compatible determinant line bundles 
over spaces of Fredholm operators, fully verify that this system satisfies 
a number of important properties, and include explicit formulas for all relevant
isomorphisms between these line bundles.
We also completely describe all possible systems of compatible determinant line bundles
and compare the conventions and approaches used elsewhere in the literature.
\end{abstract}

\tableofcontents

\section{Introduction}
\label{intro_sec}

\noindent
A Fredholm operator between Banach vector spaces $X$ and $Y$ is a bounded
homomorphism \hbox{$D\!:X\!\lra\!Y$} such that 
$$\Im\,D\equiv\big\{Dx\!:\,x\!\in\!X\big\}$$ 
is closed in~$Y$ and the dimensions~of its kernel and cokernel,
$$\ka(D)\equiv\big\{x\!\in\!X\!:\,Dx\!=\!0\big\}\qquad\hbox{and}\qquad
\fc(D)\equiv Y/(\Im\,D),$$
are finite.\footnote{The first condition is implied by the other two, but is
traditionally stated explicitly.}
The space $\cF(X,Y)$ of Fredholm operators is an open subspace of 
the space $\cB(X,Y)$ of bounded linear operators $D\!:X\!\lra\!Y$ in the normed topology;
see \cite[Theorem~A.1.5(ii)]{MS}.
Quillen's construction, outlined in \cite[Section~2]{Quillen}, associates to each Fredholm operator~$D$ 
a $\Z_2$-graded one-dimensional vector space $\la(D)=\det\,D$, 
called \textsf{the determinant line of~$D$},
and topologizes, in a systematic way, the set
$${\det}_{X,Y}\equiv \bigsqcup_{D\in\cF(X,Y)}\hspace{-.2in}\la(D)$$
as a line bundle over $\cF(X,Y)$ for each pair $(X,Y)$ of Banach vector spaces.
There are in fact infinitely many compatible systems of such topologies,
all of which we describe in Section~\ref{classify_subs};
they are isomorphic pairwise.
This is contrary to suggestions in many papers that there is a unique way
of topologizing determinant line bundles in a systematic way and can be viewed
as capturing the essence of the {\it unique up to a canonical isomorphism}
statement in \cite[Theorem~1]{KM}.
We describe some intrinsic and not so-intrinsic ways of narrowing down
the choices and  of choosing a specific system 
at the end of Subsection~\ref{detLBprop_subs} and 
at the end of Remark~\ref{Seidel_rmk}.\\

\noindent
The determinant line bundle plays a prominent role in a number of geometric situations, 
but unfortunately there appears to be no thorough description of its 
construction and properties in the literature.
The key issue in its construction is the existence of a collection of 
(set-theoretic) trivializations for ${\det}_{X,Y}$,
such as $\wt\cI_{D,T}$ in~\eref{surjtriv_e} and $\hat\cI_{\Th;D}$ in~\eref{cIThD_e2}, 
that overlap continuously.
The justification for the existence of such a collection in~\cite{Quillen} 
consists of an allusion to some unspecified collection of compatible isomorphisms 
relating the determinant line bundles in the short exact~triples
\BE{QuillenET_e}\begin{split}
\xymatrix{0\ar[r]&0\ar[r]\ar[d]&\R^{k+m}\ar[r]\ar[d]&\R^{k+m}\ar[r]\ar[d]&0\\
0\ar[r]&\R^{c+m}\ar[r]&\R^{c+m}\ar[r]&0\ar[r]&0  }
\end{split}\EE
of homomorphisms, where the middle arrow is the projection onto the last~$m$ coordinates. 
Explicit formulas for such a collection of isomorphisms appear in 
\cite[Section~(f)]{BF}, \cite[Section~3.2.1]{EES}, 
\cite[Section~20.2]{KrMr}, \cite[Appendix~A.2]{MS}, \cite[Section~2]{Salamon}, 
and \cite[Section~(11a)]{Seidel},
while \cite[Appendix~D.2]{Huang} and \cite[Chapter~I]{KM} describe it more abstractly.
The proof of \cite[Theorem~A.2.2]{MS} uses them to describe trivializations for determinant
line bundles for Fredholm operators without checking that they overlap continuously,
which in fact is not the case, as discovered by~\cite{MW}; 
see Section~\ref{othset_subs} for more details.
Key properties of such collections of isomorphisms necessary for the construction 
of the determinant line
bundle are specified in~\cite{BF}, \cite{KrMr}, \cite{Seidel}, and 
the construction itself is then briefly outlined.
The discussion of the relevant linear algebra considerations is more extensive
in~\cite{Huang}, but it contains an important deficiency, which is described
in Remark~\ref{Huang_rmk}, and does not complete the construction.
However, the general approach of \cite[Appendix~B]{Huang} is well-suited for
an explicit construction of the determinant line bundle and the analysis of its properties.
Explicit formulas for the above collection are used directly to topologize determinant 
line bundles over spaces of Fredholm operators  and for Kuranishi structures
in~\cite{Salamon} and~\cite{MW}, respectively.
The latter are closely related to the two-term case of the bounded complexes
of vector bundles for which a determinant line bundle is constructed in~\cite{KM}.
As explained in detail in Section~\ref{KM_subs}, 
using \cite[Theorem~I]{KM}, which predates~\cite{Quillen}, is perhaps 
the most efficient way for constructing the determinant line bundle and 
verifying its properties and would eliminate the need
for most of our Section~\ref{lin_alg}, but at the cost of explicit formulas
for important isomorphisms (which may well be useful in specific applications)
and of being self-contained.
None of the above works explicitly considers most of the non-trivial properties of
the determinant line bundle for Fredholm operators listed in Subsection~\ref{detLBprop_subs}.\\

\noindent
This paper provides a comprehensive construction of a system of determinant line
bundles and a complete verification of many important properties it satisfies.
Section~\ref{summ_sec} sets up the necessary notation and precisely 
describes the properties we later show this system satisfies.
Section~\ref{outline_subs} outlines the determinant line bundle construction 
carried out in this paper and three alternative approaches,
while Section~\ref{KM_subs} provides more details for the approach
based on the results obtained in~\cite{KM}.
Section~\ref{othset_subs} compares several conventions for the determinant line bundle 
that have appeared in the literature.
Section~\ref{classify_subs} establishes Theorem~\ref{classify_thm},
which describes all determinant line bundle systems satisfying the properties in 
Subsection~\ref{detLBprop_subs} and shows that such systems correspond to collections of isomorphisms 
\BE{Aicisom_e}A_{i,c}\!:\La^c(\R^c)\lra\R, \qquad i\!\in\!\Z,~c\!\in\!\Z^+,~c\ge-i\,,\EE
that are orientation-preserving if $i,c\!\in\!2\Z$. 
In contrast to the viewpoint of the previous paragraph,
there are  {\it no} compatibility conditions on the isomorphisms in these collections.
By Theorem~\ref{classify_thm}, the compatible systems of topologies on determinant
line bundles correspond to the compatible systems of isomorphisms 
for the exact triples~\eref{QuillenET_e} and to
the compatible collections 
of isomorphisms for exact triples of Fredholm operators.
Section~\ref{lin_alg}, which is motivated by \cite[Section~1]{KM} and \cite[Appendix~D.2]{Huang},
deals with the relevant linear algebra.
In particular, Subsection~\ref{ET_subs} provides explicit formulas for 
a collection of exact triple isomorphisms~$\Psi_{\ft}$ as in~\eref{ETisom_e}
and dualization isomorphisms~$\wt\cD_D$ as in~\eref{wtcD_e} satisfying
all properties of Subsection~\ref{detLBprop_subs};
see~\eref{cUDDdfn_e} and~\eref{cDdfn_e}, respectively.
Section~\ref{mainthmpf_sec} concludes this paper with topological arguments;
this section is motivated by the approach in \cite[Appendix~A.2]{MS}.
Many of the individual steps that we describe in this paper are not new.
However, even the full statement of Theorem~\ref{main_thm} on page~\pageref{main_thm}
does not seem to  appear elsewhere.\\

\noindent
The author would like to thank M.~Abouzaid, P.~Georgieva, H.~Hofer,
Y.-Z.~Huang, D.~McDuff, D.~Salamon, and K.~Wehrheim for related discussions,
the referee for pointing out additional relevant literature, 
and the IAS School of Mathematics for hospitality.

\subsection{Post-publication updates}
\label{updates_subs}

\noindent
The present version of this paper contains one modification and 
two additions to the mathematical content of the published version, which are described below.
The statements of the theorems, 
\ref{main_thm}~on~page~\pageref{main_thm} and~\ref{classify_thm} on~page~\pageref{classify_thm},
and the discussions of the connections between various properties have been
updated in accordance with these changes.
Subsections~\ref{Reg_subs}, \ref{cont_subs}, and~\ref{DualProp_subs} 
have been changed from the collections of specific isomorphisms~$\Psi_{\ft}$ in~\eref{cUDDdfn_e}
and~$\wt\cD_D$ in~\eref{cDdfn_e} to collections of such isomorphisms satisfying certain properties.
Furthermore, Section~\ref{summ_sec} has been split into two subsections, and
more details have been added to some arguments as well.
The enumeration of theorems, propositions, etc.~has not changed from the published version, 
but many equations numbers in Section~\ref{summ_sec} and 
some in Section~\ref{mainthmpf_sec} have changed.\\

\noindent
{\bf\emph{Naturality II property broadened to quasi-isomorphisms.}}
The Naturality~II property on page~\pageref{NaturII_prop} has been broadened
from isomorphisms of exact triples of Fredholm operators in the published version
to quasi-isomorphisms.
In the present version of the paper, we refer to the previous version of this property
as the {\it isomorphism} Naturality~II property.
Relatedly, the Naturality~III and Normalization~III, III$^{\star}$ properties
in the published version of the paper are now called Normalization~III, IV, IV$^{\star}$,
respectively.\\

\noindent
While no issues have been discovered with any formal statements in the published version,
the informal summary of connections between the various properties was off.
In particular, the isomorphism Naturality~II, Normalization~II,III, and Compositions~I,II 
properties do not imply the Exact Squares property.
For example, the collection of exact triple isomorphisms~$\Psi_{\ft}$ in~\eref{ETisom_e} 
given by~\eref{cUDDdfn_e} satisfies the (full) Naturality~II, Normalization~II,III, 
and Compositions~I,II properties.
For each Fredholm operator~$D$, let  
$$A_D=\begin{cases}-1,&\hbox{if}~\fc(D)\neq0,~\dim\dom(D)=\i;\\
1,&\hbox{otherwise}.\end{cases}$$
For each exact triple~$\ft$ of Fredholm operators as in~\eref{cTdiag_e},
let
$$\Psi_{\ft}'=\begin{cases}
\frac{A_{D'}A_{D''}}{A_D}\Psi_{\ft},&\hbox{if}~
\dim\dom(D'),\dim\dom(D'')=\i;\\
\Psi_{\ft},&\hbox{otherwise}.
\end{cases}$$
The new collection of exact triple isomorphisms~$\Psi_{\ft}'$ satisfies 
the isomorphism Naturality~II, Normalization~II,III, and Compositions~I,II properties.
However, this collection does not satisfy the full Naturality~II property,
for quasi-isomorphisms between exact triples with infinite- and finite-dimensional 
Fredholm operators.
Since the isomorphism Naturality~II, Normalization~II,III, and Exact Squares properties
imply the Naturality~II property by Subsection~\ref{classify_subs},
the above collection~$\{\Psi_{\ft}'\}_{\ft}$ does not satisfy the Exact Squares property either.\\

\noindent
{\bf\emph{Complex orientations.}}
For a $\C$-linear Fredholm operator~$D$ between Banach vector spaces~$X$ and~$Y$
with complex structures, the determinant line~$\la(D)$ has a canonical orientation.
Such orientations are sometimes used to orient the determinant lines for
other operators by transferring them along paths of Fredholm operators.
The Complex Orientations, Complex Exact Triples, and 
Dual Complex Orientations properties on pages~\pageref{COrient_prop}, \pageref{CET_prop}, 
and~\pageref{DualCOrient_prop} have been added to require the topologies on 
the determinant line bundles, isomorphisms for exact triples of Fredholm operators,
and dualization isomorphisms to be compatible with these orientations.
As indicated by the proof of Theorem~\ref{classify_thm} in Subsection~\ref{classify_subs},
these properties do not cut down on the admissible systems of determinant line bundles
significantly.\\

\noindent
{\bf\emph{Wall-crossing for orientations.}}
If $D$ is an isomorphism between  Banach vector spaces~$X$ and~$Y$,
the determinant line~$\la(D)$ again has a canonical orientation. 
By the Normalization~I property on page~\pageref{NormalI_prop},
these orientations vary continuously over the space of isomorphisms.
Subsection~\ref{CrossNums_subs} has been added to give a criterion
determining whether the extension of the canonical orientation for an isomorphism~$D$
over a generic path in~$\cF(X,Y)$ ending an another isomorphism~$D'$
restricts to the canonical orientation of~$\la(D')$.
The answer turns out to be independent of the choice of an admissible system of
determinant line bundles.
Along with the added complex orientation properties,
this implies that the signs defined in certain geometric settings,
such as in Gromov-Witten theory, by transferring orientations from
$\C$-linear operators along paths do not depend on the choice of 
either $\C$-linear operators or 
an admissible system of determinant line bundles.

\section{The determinant line bundle}
\label{summ_sec}

\subsection{Notation and terminology}
\label{detLBnota_subs}

\noindent
All vector spaces we consider are over~$\R$.
We denote by $\fd(V)$ the dimension of a vector space~$V$ and~by
$$\la(V)\equiv\La^{\top}V\equiv \La^{\fd(V)}V \qquad\hbox{and}\qquad 
\la^*(V)\equiv \big(\la(V)\big)^*  $$
the top exterior power of~$V$ and its dual, whenever $\fd(V)\!<\!\i$.
We view $\la(V)$ and $\la^*(V)$ as graded lines of degrees
$$\deg\la(V),\deg\la^*(V)=\fd(V)+2\Z\in\Z_2.$$
For any two $\Z_2$-graded lines $L_1$ and $L_2$, we define
\begin{gather}
\deg L_1\!\otimes\!L_2=\deg L_1+\deg L_2,\notag \\
\label{Risom_e}R\!:L_1\!\otimes\!L_2\lra L_2\!\otimes\!L_1,\qquad 
R(v_1\!\otimes\!v_2)=(-1)^{(\deg L_1)(\deg L_2)}v_2\!\otimes\!v_1.
\end{gather}
If $\cL_1,\cL_2\!\lra\!\cF$ are $\Z_2$-graded line bundles (each fiber has a grading
varying continuously over~$\cF$), the fiberwise isomorphisms~$R$ give rise
to an isomorphism
$$R\!:\cL_1\!\otimes\!\cL_2\lra\cL_2\!\otimes\!\cL_1$$
of $\Z_2$-graded line bundles over $\cF$.
If $L$ is a line and $v\!\in\!L\!-\!0$, we define $v^*\!\in\!L^*$ by $v^*(v)\!=\!1$.\\

\noindent
For a finite-dimensional vector space $V$, we define
\BE{cPisom_e}\cP\!:\la(V^*)\lra\la^*(V), \quad
\big\{\cP(\al_1\!\w\!\ldots\w\!\al_n)\big\}(v_1\!\w\!\ldots\w\!v_n)
=(-1)^{\binom{n}{2}}\det\big(\al_i(v_j))_{i,j=1,\ldots,n}\EE
and denote the inverse of~$\cP$ also by~$\cP$.
The advantages of the isomorphism~\eref{cPisom_e} over the isomorphism induced by
the first pairing in~\eref{pairdfn_e} are that the former 
respects complex orientations and fits better
with short exact sequences; see \eref{Cdualpair_e2} and
the last statement of Lemma~\ref{sesdual_lmm}.\\

\noindent
For a Fredholm operator $D\!:X\!\lra\!Y$, we define
\BE{DLdfn_e} \la(D)= \la(\ka(D))\otimes\la^*(\fc(D)) \EE
with the grading 
$$\deg\la(D)\equiv \ind\,D+2\Z\equiv\fd(\ka(D))-\fd(\fc(D))+2\Z\in\Z_2.$$
This is the same definition as in \cite[Section~20.2]{KrMr} and \cite[Section~7.4]{MW};
we discuss alternative versions of~\eref{DLdfn_e} in 
Subsections~\ref{KM_subs} and~\ref{othset_subs}.\\

\noindent
{\bf\emph{Morphisms.}}
A \textsf{homomorphism between Fredholm operators $D\!:X\!\lra\!Y$ and
$D'\!:X'\!\lra\!Y'$} is a pair of homomorphisms $\phi\!:X\!\lra\!X'$ 
and $\psi\!:Y\!\lra\!Y'$ so that $D'\!\circ\!\phi\!=\!\psi\!\circ\!D$.
A \textsf{quasi-isomorphism between Fredholm operators $D$ and~$D'$} 
is a homomorphism $(\phi,\psi)\!:D\!\lra\!D'$ that induces isomorphisms 
$$\phi_{\ka}\!:\ka(D)\lra\ka(D') \qquad\hbox{and}\qquad
\psi_{\fc}\!:\fc(D)\lra\fc(D').$$
In such a case, we denote by 
\begin{gather}
\la(\phi_{\ka})\!:\la(\ka(D))\lra\la(\ka(D')), \qquad
\la(\psi_{\fc}^{-1})\!:\la(\fc(D'))\lra\la(\fc(D)),\notag\\
\label{cIphipsi_e}
\wt\cI_{\phi,\psi;D}\!:\la(D)\lra \la(D'), \quad
x\otimes \al\lra  \big(\la(\phi_{\ka})x\big) \otimes\big(\al\!\circ\!\la(\psi_{\fc}^{-1})\big)
\end{gather}
the induced isomorphisms of the associated lines.
An \textsf{isomorphism between Fredholm operators $D$ and~$D'$} 
is a homomorphism $(\phi,\psi)\!:D\!\lra\!D'$
so that $\phi$ and $\psi$ are isomorphisms.\\

\noindent
Isomorphisms $\phi\!:X\!\lra\!X'$ and $\psi\!:Y\!\lra\!Y'$ 
between Banach vector spaces induce a homeomorphism
$$\cI_{\phi,\psi}\!:\cF(X,Y)\lra\cF(X',Y'), \qquad 
\cI_{\phi,\psi}(D)=\psi\circ D\circ\phi^{-1}.$$
In particular, $(\phi,\psi)\!:D\!\lra\!\cI_{\phi,\psi}(D)$ is an isomorphism of 
Fredholm operators for each $D\!\in\!\cF(X,Y)$.
Putting the isomorphisms~$\wt\cI_{\phi,\psi;D}$ in~\eref{cIphipsi_e}
together, we obtain a bundle map
\BE{wtphipsi_e}\wt\cI_{\phi,\psi}\!: {\det}_{X,Y}\lra \cI_{\phi,\psi}^*{\det}_{X',Y'}\EE
covering the identity on~$\cF(X,Y)$.\\

\noindent
{\bf\emph{Exact Triples.}}
An \sf{exact triple of Fredholm operators}, 
$$0\lra D'\lra D\lra D''\lra0,$$
is a commutative diagram
\BE{cTdiag_e}\begin{split}
\xymatrix{0\ar[r]& X'\ar[d]^{D'}\ar[r]^{\fI_X}& X\ar[d]^D\ar[r]^{\fJ_X}& 
X''\ar[d]^{D''}\ar[r]&0\\
0\ar[r]& Y'\ar[r]^{\fI_Y}& Y\ar[r]^{\fJ_Y}& Y''\ar[r]&0}
\end{split}\EE
of homomorphisms between Banach vector spaces with exact rows
and with Fredholm operators as columns.
A \sf{quasi-isomorphism}
\BE{ETisom_e5}\begin{split}
\xymatrix{0\ar[r]& D_{\tnT}'\ar[r]\ar[d]|{(\phi',\psi')} & D_{\tnT}\ar[r]\ar[d]|{(\phi,\psi)}
& D_{\tnT}''\ar[r]\ar[d]|{(\phi'',\psi'')}&0\\
0\ar[r]& D_{\tnB}'\ar[r]& D_{\tnB}\ar[r] & D_{\tnB}''\ar[r]&0}
\end{split}\EE
between exact triples 
$$0\lra D_{\tnT}'\lra D_{\tnT}\lra D_{\tnT}''\lra0 \quad\hbox{and}\quad
0\lra D_{\tnB}'\lra D_{\tnB}\lra D_{\tnB}''\lra0$$
of Fredholm operators are homomorphisms
$$\xymatrix{0\ar[r]& X_{\tnT}'\ar[r]\ar[d]^{\phi'} & X_{\tnT}\ar[r]\ar[d]^{\phi}
& X_{\tnT}''\ar[r]\ar[d]^{\phi''}&0&
0\ar[r]& Y_{\tnT}'\ar[r]\ar[d]^{\psi'} & Y_{\tnT}\ar[r]\ar[d]^{\psi}
& Y_{\tnT}''\ar[r]\ar[d]^{\psi''}&0\\
0\ar[r]& X_{\tnB}'\ar[r]& X_{\tnB}\ar[r] & X_{\tnB}''\ar[r]&0&
0\ar[r]& Y_{\tnB}'\ar[r]& Y_{\tnB}\ar[r] & Y_{\tnB}''\ar[r]&0}$$
of short exact sequences of Banach vector spaces so that 
$(\phi',\psi')$, $(\phi,\psi)$, and $(\phi'',\psi'')$ are quasi-isomorphisms
between the Fredholm operators~$D_{\tnT}'$ and~$D_{\tnB}'$,
$D_{\tnT}$ and~$D_{\tnB}$, and $D_{\tnT}''$ and~$D_{\tnB}''$, respectively.
An \sf{isomorphism} between exact triples of Fredholm operators as above
is a quasi-isomorphism between these exact triples as in~\eref{ETisom_e5}
so that the homomorphisms 
$\phi',\psi',\phi,\psi,\phi'',\psi''$ are isomorphisms.\\

\noindent
For Banach vector spaces $X,Y,X',Y',X'',Y''$, let
\BE{cTdfn_e}\begin{split}
&\cT(X,Y;X',Y';X'',Y'')\\
&\quad\subset \cF(X,Y)\times\cF(X',Y')\times\cF(X'',Y'')\times
\cB(X',X)\times\cB(X,X'')\times \cB(Y',Y)\times\cB(Y,Y'')
\end{split}\EE
be the subspace of tuples $(D,D',D'',\fI_X,\fJ_X,\fI_Y,\fJ_Y)$ so that 
\eref{cTdiag_e} is an exact triple of Fredholm operators.
Denote~by 
$$\pi_{\tnC},\pi_{\tnL},\pi_{\tnR}\!: \cT(X,Y;X',Y';X'',Y'')\lra
\cF(X,Y),\cF(X',Y'),\cF(X'',Y'')$$
the restrictions of the projection~maps. 
For $\star\!=\!',''$, denote by 
$$\cT^{\star}(X,Y;X',Y';X'',Y'')\subset \cT(X,Y;X',Y';X'',Y'')$$
the subspace of diagrams~\eref{cTdiag_e} so that $D^{\star}$ is an isomorphism.\\

\noindent
If $\ft\!\in\!\cT'(X,Y;X',Y';X'',Y'')$ is as in~\eref{cTdiag_e}, 
$(\fJ_X,\fJ_Y)$ is a quasi-isomorphism between the Fredholm operators $D$ and~$D''$.
With the notation as in~\eref{cIphipsi_e}, define
\BE{cIft1_e}\cI_{\ft}'\!:\la(D')\!\otimes\!\la(D'')\lra\la(D), \qquad
\cI_{\ft}'\big((1\!\otimes\!1^*)\!\otimes\!\si''\big)
=\wt\cI_{\fJ_X,\fJ_Y;D}^{-1}(\si'').\EE
If $\ft\!\in\!\cT''(X,Y;X',Y';X'',Y'')$ is as in~\eref{cTdiag_e}, 
$(\fI_X,\fI_Y)$ is a quasi-isomorphism between the Fredholm operators~$D'$ and~$D$.
Define
\BE{cIft2_e}\cI_{\ft}''\!:\la(D')\!\otimes\!\la(D'')\lra\la(D), \qquad
\cI_{\ft}''\big(\si'\!\otimes\!(1\!\otimes\!1^*)\big)
=\wt\cI_{\fI_X,\fI_Y;D'}(\si').\EE

\vspace{.1in}

\noindent
{\bf\emph{Direct Sums.}}
For Banach vector spaces $X',Y',X'',Y''$, the direct sum operation
$$\oplus\!: \cF(X',Y')\!\times\!\cF(X'',Y'') \lra 
\cF(X'\!\oplus\!X'',Y'\!\oplus\!Y''), 
\qquad (D',D'')\lra D'\!\oplus\!D'',$$
is a continuous map.
For any $D\!\in\!\cF(X,Y)$ and a Banach vector space~$Z$,
the projections 
$$(\phi,\psi)\!:(Z\!\oplus\!X,Z\!\oplus\!Y)\lra (X,Y) \qquad\hbox{and}\qquad
(\phi,\psi)\!:(X\!\oplus\!Z,Y\!\oplus\!Z)\lra (X,Y)$$
are quasi-isomorphism between the Fredholm operators 
$\id_Z\!\oplus\!D$, $D$, and $D\!\oplus\!\id_Z$.
Via~\eref{cIphipsi_e}, they thus determine identifications
\begin{alignat}{1}\label{addid_e}
\la(\id_Z\!\oplus\!D)=&\la(D)=\la(D\!\oplus\!\id_Z),\\
(0,x_1)\!\w\!\ldots\!\w\!(0,x_k)\otimes \big((0,y_1)\!\w\!\ldots\!\w\!(0,y_{\ell})\big)^*
&\llra x_1\!\w\!\ldots\!\w\!x_k\otimes \big(y_1\!\w\!\ldots\!\w\!y_{\ell}\big)^*\notag\\
&\llra (x_1,0)\!\w\!\ldots\!\w\!(x_k,0)\otimes 
\big((y_1,0)\!\w\!\ldots\!\w\!(y_{\ell},0)\big)^*.\notag
\end{alignat}

\vspace{.15in}

\noindent
We denote by
\begin{gather*}
R_{X',X''}\!: X'\!\oplus\!X''\lra X''\!\oplus\!X'
\qquad\hbox{and}\\
\cR_{\cF}\!:  
\cF(X',Y')\!\times\!\cF(X'',Y'') \lra  \cF(X'',Y'')\!\times\!\cF(X',Y') 
\end{gather*}
the maps interchanging the two factors.
Let
\begin{equation*}\begin{split}
\oplus'\!:
\cF(X',Y')\!\times\!\cF(X'',Y'') &\lra\cF(X''\!\oplus\!X',Y''\!\oplus\!Y')\qquad\hbox{and}\\
\bigoplus\!: \cF(X',Y')\!\times\!\cF(X'',Y'')\!\times\!\cF(X''',Y''')&\lra
\cF(X'\!\oplus\!X''\!\oplus\!X''',Y'\!\oplus\!Y''\!\oplus\!Y''')
\end{split}\end{equation*}
be the compositions 
\begin{alignat}{1}\label{commut_e0} 
\oplus\circ \cR_{\cF} &=\cI_{R_{X',X''},R_{Y',Y''}}\circ\oplus \qquad\hbox{and}\\
\label{commut_e0b}
\oplus\circ \oplus\!\times\!\id_{\cF(X''',Y''')} &=
  \oplus\circ\id_{\cF(X',Y')}\!\times\!\oplus\,,
\end{alignat}
respectively.\\

\noindent
We associate the direct sum $D'\!\oplus\!D''$ of Fredholm operators 
$D'\!:X'\!\lra\!Y'$ and $D''\!:X''\!\lra\!Y''$ with the commutative diagram
\BE{sumEX_e}\begin{split}
\xymatrix{0\ar[r]& X'\ar[d]^{D'}
\ar[r]^>>>>>{\fI_X}& X'\!\oplus\!X''\ar[d]|{D'\oplus D''}\ar[r]^<<<<<{\fJ_X}& 
X''\ar[d]^{D''}\ar[r]&0
&*\txt{$\fI_X(x')=(x',0)$ \\ $\fJ_X(x',x'')=x''$}\\
0\ar[r]& Y'\ar[r]^>>>>>{\fI_Y}& Y'\!\oplus\!Y''\ar[r]^<<<<<{\fJ_Y}& Y''\ar[r]&0
&*\txt{$\fI_Y(y')=(y',0)$\\ $\fJ_Y(y',y'')=y''$\,.}}
\end{split}\EE
This yields an embedding
\begin{gather*}
\io_{\oplus}\!: \cF(X',Y')\!\times\!\cF(X'',Y'') \lra 
\cT(X'\!\oplus\!X'',Y'\!\oplus\!Y'';X',Y';X'',Y'')\qquad\hbox{s.t.}\\ 
\pi_{\tnC}\!\circ\!\io_{\oplus}=\oplus, \quad
\pi_{\tnL}\!\circ\!\io_{\oplus}=\pi_1,\quad
\pi_{\tnR}\!\circ\!\io_{\oplus}=\pi_2,
\end{gather*} 
where 
$$\pi_1,\pi_2\!:\cF(X',Y')\!\times\!\cF(X'',Y'') \lra \cF(X',Y'),\cF(X'',Y'')$$
are the projection maps.\\

\noindent
{\bf\emph{Compositions.}}
For Banach vector spaces $X_1,X_2,X_3$, the composition map
$$\cC_{X_2}\!: \cF(X_1,X_2)\times\cF(X_2,X_3) \lra \cF(X_1,X_3),
\qquad (D_1,D_2)\lra D_2\circ D_1,$$
is continuous as well.
If $X_4$ is another Banach vector space, let 
$$\cC_{X_2,X_3}\!:\cF(X_1,X_2)\times \cF(X_2,X_3)\times \cF(X_3,X_4)\lra \cF(X_1,X_4)$$
denote the compositions
\BE{assoc_e0}\cC_{X_3}\circ \big\{\cC_{X_2}\!\times\!\id_{\cF(X_3,X_4)}\big\}
=\cC_{X_2}\circ \big\{\id_{\cF(X_1,X_2)}\!\times\!\cC_{X_3}\big\}.\EE

\vspace{.15in}

\noindent
With the notation as in~\eref{cTdfn_e}, we denote~by
$$\cC_{\cT}\!: \cT(X_1,X_2;X_1',X_2';X_1'',X_2'')\times
\cT(X_2,X_3;X_2',X_3';X_2'',X_3'') \lra \cT(X_1,X_3;X_1',X_3';X_1'',X_3'')$$
the continuous map sending commutative diagrams
\BE{cCcTdiag_e}\begin{split}
\xymatrix{0\ar[r]& X_1'\ar[d]^{D_1'}\ar[r]^{\fI_1}& X_1\ar[d]^{D_1}\ar[r]^{\fJ_1}& 
X_1''\ar[d]^{D_1''}\ar[r]&0&
0\ar[r]& X_2'\ar[d]^{D_2'}\ar[r]^{\fI_2}& X_2\ar[d]^{D_2}\ar[r]^{\fJ_2}& 
X_2''\ar[d]^{D_2''}\ar[r]&0\\
0\ar[r]& X_2'\ar[r]^{\fI_2}& X_2\ar[r]^{\fJ_2}& X_2''\ar[r]&0&
0\ar[r]& X_3'\ar[r]^{\fI_3}& X_3\ar[r]^{\fJ_3}& X_3''\ar[r]&0}
\end{split}\EE
to the commutative diagram
\BE{cCcTdiag_e2}\begin{split}
\xymatrix{0\ar[r]& X_1'\ar[d]|{D_2'\circ D_1'}\ar[r]^{\fI_1}& 
X_1\ar[d]|{D_2\circ D_1}\ar[r]^{\fJ_1}& 
X_1''\ar[d]|{D_2''\circ D_1''}\ar[r]&0\\
0\ar[r]& X_3'\ar[r]^{\fI_3}& X_3\ar[r]^{\fJ_3}& X_3''\ar[r]&0\,.}
\end{split}\EE
We note that 
\BE{cTcomm_e} (\pi_{\tnC},\pi_{\tnL},\pi_{\tnR})\circ\cC_{\cT}
=\big(\cC_{X_2}\!\circ\!(\pi_{\tnC}\!\circ\!\pi_1,\pi_{\tnC}\!\circ\!\pi_2),
\cC_{X_2'}\!\circ\!(\pi_{\tnL}\!\circ\!\pi_1,\pi_{\tnL}\!\circ\!\pi_2),
\cC_{X_2''}\!\circ\!(\pi_{\tnR}\!\circ\!\pi_1,\pi_{\tnR}\!\circ\!\pi_2)\big),\EE
where
\begin{equation*}\begin{split}
\pi_1,\pi_2\!:\cT(X_1,X_2;X_1',X_2';X_1'',X_2'')&\times
\cT(X_2,X_3;X_2',X_3';X_2'',X_3'')\\ 
&\lra \cT(X_1,X_2;X_1',X_2';X_1'',X_2''),\cT(X_2,X_3;X_2',X_3';X_2'',X_3'')
\end{split}\end{equation*}
are the projection maps.\\

\noindent
We associate the composition $D_2\!\circ\!D_1$ of Fredholm operators 
$D_1\!:X_1\!\lra\!X_2$ and \hbox{$D_2\!:X_2\!\lra\!X_3$} with the exact~triple
\BE{compEX_e}\begin{split}
\xymatrix{0\ar[r]& X_1\ar[d]^{D_1}
\ar[r]^>>>>>{\fI_X}& X_1\!\oplus\!X_2\ar[d]|{D_2\circ D_1\oplus\id_{X_2}}\ar[r]^<<<<<{\fJ_X}& 
X_2\ar[d]^{D_2}\ar[r]&0
&*\txt{$~~~~\fI_X(x_1)=(x_1,D_1x_1)$\\ $\fJ_X(x_1,x_2)=D_1x_1\!-\!x_2$}\\
0\ar[r]& X_2\ar[r]^>>>>>{\fI_Y}& X_3\!\oplus\!X_2\ar[r]^<<<<<{\fJ_Y}& X_3\ar[r]&0
&*\txt{$~~~~~\fI_Y(x_2)=(D_2x_2,x_2)$\\ $\fJ_Y(x_3,x_2)=x_3\!-\!D_2x_2$\,.}}
\end{split}\EE
This yields an embedding
\begin{gather*}
\io_{\cC}\!: \cF(X_1,X_2)\!\times\!\cF(X_2,X_3) \lra 
\cT(X_1\!\oplus\!X_2,X_3\!\oplus\!X_2;X_1,X_2;X_2,X_3) \qquad\hbox{s.t.}\\ 
\pi_{\tnC}\!\circ\!\io_{\cC}(D_1,D_2)=\cC_{X_2}(D_1,D_2)\oplus\id_{X_2}, \quad
\pi_{\tnL}\!\circ\!\io_{\cC}=\pi_1,\quad \pi_{\tnR}\!\circ\!\io_{\cC}=\pi_2.
\end{gather*}
Combining the first identity above with~\eref{addid_e}, we obtain 
$$\io_{\cC}^*\pi_{\tnC}^*{\det}_{X_1\oplus X_2,X_3\oplus X_2}=\cC_{X_2}^*{\det}_{X_1,X_3}.$$
If $D\!\in\!\cF(X,Y)$, the compositions $D\!\circ\!\id_X$ and $\id_Y\!\circ\!D$
correspond to elements~of
$$\cT'(X\!\oplus\!X,Y\!\oplus\!X;X,X;X,Y) \qquad\hbox{and}\qquad
\cT''(X\!\oplus\!Y,Y\!\oplus\!Y;X,Y;Y,Y),$$
respectively, with the isomorphisms $\cI_{\ft}'$ and $\cI_{\ft}''$ of~\eref{cIft1_e} 
and~\eref{cIft2_e} given~by
$$\cI_{\ft}'\big(1\!\otimes\!1^*\otimes x\!\otimes\!\be\big)= x\!\otimes\!\be 
\qquad\hbox{and}\qquad
\cI_{\ft}''\big(x\!\otimes\!\be \otimes 1\!\otimes\!1^*\big)=  x\!\otimes\!\be $$
under the identifications~\eref{addid_e}.\\

\noindent
{\bf\emph{Dualizations.}}
For each Banach vector space $X$, let $X^*$ denote the dual Banach vector space,
i.e.~the space $\Hom_{\R}(X,\R)$ of bounded linear functionals $X\!\lra\!\R$.
For each $D\!\in\!\cF(X,Y)$, let \hbox{$D^*\!\in\!\cF(Y^*,X^*)$} denote the dual operator,~i.e.
$$\{D^*\be\}(x)=\be(Dx) \qquad\forall\,\be\!\in\!Y^*,\,x\!\in\!X.$$
The map
$$\cD\!:\cF(X,Y)\lra \cF(Y^*,X^*), \qquad \cD(D)= D^*,$$
is then continuous.
For each  $D\!\in\!\cF(X,Y)$, the homomorphisms
\BE{fkfcdual_e}
\begin{aligned}
\cD_D\!:\ka(D)&\lra\fc(D^*)^*, &\quad 
  \big\{\cD_D(x)\big\}(\al\!+\!\Im\,D^*)&=\al(x)
  &~~&\forall\,x\!\in\!\ka(D),\,\al\!\in\!X^*,\\
\cD_D\!:\fc(D)^*&\lra\ka(D^*),&\quad  \big\{\cD_D(\be)\big\}(y)&=\be(y\!+\!\Im\,D)
&~~&\forall\,\be\!\in\!\fc(D)^*,\,y\!\in\!Y,
\end{aligned}\EE
are isomorphisms.\\

\noindent
For each exact triple $\ft$ of Fredholm operators as in \eref{cTdiag_e},
we define the dual triple~$\ft^*$ to be given by the diagram
\BE{cTdiagdual_e}\begin{split}
\xymatrix{0\ar[r]& Y''^*\ar[d]^{D''^*}\ar[r]^{\fJ_Y^*}& Y^*\ar[d]^{D^*}\ar[r]^{\fI_Y^*}& 
Y'^*\ar[d]^{D'^*}\ar[r]&0\\
0\ar[r]& X''^*\ar[r]^{\fJ_X^*}& X^*\ar[r]^{\fI_X^*}& X'^*\ar[r]&0\,.}
\end{split}\EE
This defines an embedding
\begin{gather}
\cD_{\cT}\!: \cT(X,Y;X',Y';X'',Y'')\lra\cT(Y^*,X^*;Y''^*,X''^*;Y'^*,X'^*)
\qquad\hbox{s.t.}\notag\\
\label{cDcTcomm_e} 
\pi_{\tnC}\!\circ\!\cD_{\cT}=\cD\!\circ\!\pi_{\tnC}, \quad
\pi_{\tnL}\!\circ\!\cD_{\cT}=\cD\!\circ\!\pi_{\tnR},\quad
\pi_{\tnR}\!\circ\!\cD_{\cT}=\cD\!\circ\!\pi_{\tnL}\,.
\end{gather}

\vspace{.15in}

\noindent
{\bf\emph{Complex Orientations.}} 
Let $\fI$ be a complex structure on a real vector space~$V$.
The homomorphism
$$\Re\!:\Hom_{\C}(V,\C)\lra V^*\!\equiv\!\Hom_{\R}(V,\R), \quad
\{\Re(\al)\}(v)=\Re\big(\al(v)\!\big)
~~\forall\,\al\!\in\!\Hom_{\C}(V,\C),\,v\!\in\!V,$$
is an isomorphism.
Via this isomorphism, $\fI$ induces a complex structure on~$V^*$,
which we still denote by~$\fI$, so~that
\BE{Cdualpair_e}\{\fI\al\}(v)=\al(\fI v) \quad\forall\,\al\!\in\!V^*,\,v\!\in\!V.\EE
If $V$ is finite-dimensional with $\C$-basis $e_1,\ldots,e_n$, then 
$e_1,\fI e_1,\ldots,e_n,\fI e_n$ is an $\R$-basis for~$V$
determining the \sf{complex orientation} of~$V$.
If $e_1^*,\ldots,e_n^*$ is the dual $\C$-basis for~$V^*$, then
$$e_1^*,\fI e_1^*\!=\!-(\fI e_1)^*,\ldots,e_n^*,\fI e_n^*\!=\!-(\fI e_n)^*$$
is an $\R$-basis for~$V^*$ determining the complex orientation of~$V^*$.
Thus,
\BE{Cdualpair_e2}\big(e_1\!\w\!\fI e_1\!\w\!\ldots\!\w\!e_n\!\w\!\fI e_n\big)^{\!*}
=\cP\big(e_1^*\!\w\!\fI e_1^*\!\w\!\ldots\!\w\!e_n^*\!\w\!\fI e_n^*\big)
\in\la^*(V),\EE
i.e.~the isomorphism~\eref{cPisom_e} with~$n$ replaced by~$2n$ intertwines 
the complex orientations of~$\la(V^*)$ and~$\la^*(V)$.\\ 

\noindent
Suppose $X,Y$ are Banach vector spaces with complex structures.
We then denote~by 
$$\cF_{\C}(X,Y)\subset\cF(X,Y)$$ 
the closed subspace of $\C$-linear Fredholm operators.
For each~$D\!\in\!\cF_{\C}(X,Y)$, $\ka(D)$ and $\fc(D)$ are finite-dimensional 
complex vector spaces.
Thus, the real lines $\la(\ka(D))$, $\la(\fc(D))$, and~$\la(D)$ 
have canonical orientations, which we will call the \sf{complex orientations}.
In this case, the Banach vector spaces~$X^*,Y^*$ inherit complex structures 
from~$X,Y$, $D^*\!\in\!\cF_{\C}(Y^*,X^*)$, and 
the isomorphisms~\eref{fkfcdual_e} are $\C$-linear.\\

\noindent
If $X,Y,X',Y',X'',Y''$ are Banach vector spaces with complex structures,
we denote~by
$$\cT_{\C}(X,Y;X',Y';X'',Y'')\subset\cT(X,Y;X',Y';X'',Y'')$$
the subspace of exact triples as in~\eref{cTdiag_e}
so that the Fredholm operators $D',D,D''$ and the homomorphisms
$\fI_X,\fJ_X,\fI_Y,\fJ_Y$ are $\C$-linear.

\subsection{Properties}
\label{detLBprop_subs}

\noindent
The topologies on the line bundles ${\det}_{X,Y}$ should satisfy 
a number of important compatibility properties, which we now describe.\\

\begin{mywidth}
{\bf\emph{Naturality I.}}\label{NaturI_prop} 
The bundle map~$\wt\cI_{\phi,\psi}$ in~\eref{wtphipsi_e} is continuous
for all isomorphisms \hbox{$\phi\!:X\!\lra\!X'$} and $\psi\!:Y\!\lra\!Y'$ 
between Banach vector spaces.\\
\end{mywidth}

\begin{mywidth}
{\bf\emph{Complex Orientations.}}\label{COrient_prop}
If $X,Y$ are Banach vector spaces with complex structures, the complex orientations of
the lines~$\la(D)$ with $D\!\in\!\cF_{\C}(X,Y)$ determine an orientation
on the restriction of~${\det}_{X,Y}$ to~$\cF_{\C}(X,Y)$.\\
\end{mywidth}

\noindent
The substance of the Complex Orientations property is that 
the complex orientations of the lines~$\la(D)$, wherever defined,
are continuous with respect to the topology of the relevant restriction 
of~${\det}_{X,Y}$.\\

\noindent
If $X,Y$ are Banach vector spaces and  $T\!\in\!\cB(Y,X)$, let
$$U_T=\big\{P\!\in\!\cB(X,Y)\!:\|TP\|\!<\!1\big\};$$
this is an open subset of $\cB(X,Y)$.
If in addition $P\!\in\!\cB(X,Y)$, define 
$$\Phi_{T;P}\!\in\!\cB(X,X) \qquad\hbox{by}\qquad \Phi_{T;P}(x)=x\!+\!TPx.$$
The operator $\Phi_{T;P}$ is invertible for every $P\!\in\!U_T$ and the~map
$$U_T\lra\cF(X,X), \qquad P\lra 
\Phi_{T;P}^{-1}=\sum_{r=0}^{\i}\!(-1)^r(TP)^r\,,$$
is continuous.\\

\noindent
We define
\begin{gather*}
\cF^*(X,Y)=\big\{D\!\in\!\cF(X,Y)\!:\,\fc(D)=\{0\}\big\},\\
\pi\!:\ka(X,Y)\!=\!\big\{(D,x)\!\in\!\cF^*(X,Y)\!\times\!X\!:Dx\!=\!0\big\}
\lra\cF^*(X,Y), \quad\pi(D,x)=D.
\end{gather*}
The first set above is an open subset of $\cF(X,Y)$.
For each $D\!\in\!\cF^*(X,Y)$ and each right inverse $T\!:Y\!\lra\!X$ of~$D$,
the~map
$$(D\!+\!U_T)\!\times\!X\lra (D\!+\!U_T)\!\times\!X, 
\qquad (D\!+\!P,x)\lra \big(D\!+\!P,\Phi_{T;P}(x)\big),$$
is continuous, linear on each fiber of the projection to the first component,
and restricts to a bijection
$$\ka(X,Y)\big|_{\pi^{-1}(D+U_T)}\lra (D\!+\!U_T)\!\times\!\ka(D).$$
Thus, $\ka(X,Y)$ is a subbundle of  the trivial Banach bundle 
$$\cF^*(X,Y)\!\times\!X\lra \cF^*(X,Y).$$

\vspace{.1in}

\begin{mywidth}
{\bf\emph{Normalization I.}}\label{NormalI_prop} 
The topology of $\det_{X,Y}|_{\cF^*(X,Y)}$ is the topology 
of the top exterior power of the vector bundle~$\ka(X,Y)$.\\
\end{mywidth}

\noindent
This condition can alternatively be described as follows.
For $D\!\in\!\cF^*(X,Y)$, each right inverse $T\!:Y\!\lra\!X$ of~$D$,
and $P\!\in\!\cB(X,Y)$, let 
\BE{Phidfn_e} 
\Phi_{D,T;P}\!:\,\ka(D\!+\!P)\lra \ka(D), \qquad
 \Phi_{D,T;P}(x)=\Phi_{T;P}(x).\EE
This map is a bijection for every $P\!\in\!U_T$ and thus induces an isomorphism
$$ \wt\cI_{D,T;D+P}\!:
\la(D)=\la(\ka(D))\otimes\R \lra \la\big(\ka(D\!+\!P)\big)\otimes\R=  \la(\ka(D\!+\!P)).$$
Putting these isomorphisms together, we obtain a bundle map
\BE{surjtriv_e}\wt\cI_{D,T}\!: (D\!+\!U_T)\!\times\!\la(D)\lra {\det}_{X,Y}\big|_{D+U_T}\EE
covering the identity on the open neighborhood $D\!+\!U_T$ of $D$ in~$\cF^*(X,Y)$.
Normalization~I is equivalent to the condition that
the map~$\wt\cI_{D,T}$ is continuous for every 
$D\!\in\!\cF^*(X,Y)$ and right inverse $T\!:Y\!\lra\!X$ of~$D$.\\

\noindent
If $X,Y$ are Banach vector spaces with complex structures, 
every surjective $\C$-linear Fredholm operator~$D$ admits 
a $\C$-linear right inverse~$T$.
The isomorphism~\eref{Phidfn_e} is then $\C$-linear whenever \hbox{$P\!\in\!U_T$}
is $\C$-linear.
Thus, the complex orientations of the lines~$\la(D\!+\!T)$ with 
$P\!\in\!U_T$ $\C$-linear determine an orientation on the restriction of~${\det}_{X,Y}$
over the intersection of $D\!+\!U_T$ with 
$$\cF_{\C}^*(X,Y)\equiv\cF_{\C}(X,Y)\!\cap\!\cF^*(X,Y)\,.$$
Thus, the Complex Orientations property over~$\cF_{\C}^*(X,Y)$
follows from the Normalization~I property.\\

\begin{mywidth}
{\bf\emph{Exact Triples.}} There exists a collection of (continuous) line bundle isomorphisms
\BE{ETisom_e}\Psi\!: 
\pi_{\tnL}^*{\det}_{X',Y'}\otimes \pi_{\tnR}^*{\det}_{X'',Y''}
\lra\pi_{\tnC}^*{\det}_{X,Y}\EE
over $\cT(X,Y;X',Y';X'',Y'')$ parametrized by the tuples $(X,Y;X',Y';X'',Y'')$ 
of Banach vector spaces with the following properties.\\ 
\end{mywidth}

\begin{mywidth2}
{\bf\emph{Naturality II.}}\label{NaturII_prop} The isomorphisms $\Psi$ commute
with the isomorphisms~\eref{cIphipsi_e} induced by 
quasi-isomorphisms of exact triples of Fredholm operators, i.e.~the diagram
$$\xymatrix{ \la(D_{\tnT}')\otimes\la(D_{\tnT}'') 
\ar[d]|{\wt\cI_{\phi',\psi';D_{\tnT}'}\otimes\wt\cI_{\phi'',\psi'';D_{\tnT}''}}
\ar[rrr]^>>>>>>>>>>>>>>>{\Psi_{\tnT}}&&& 
\la(D_{\tnT}) \ar[d]|{\wt\cI_{\phi,\psi;D_{\tnT}}}\\
\la(D_{\tnB}')\otimes\la(D_{\tnB}'') 
\ar[rrr]^>>>>>>>>>>>>>>>{\Psi_{\tnB}}&&& \la(D_{\tnB})\,,}$$
where $\Psi_{\tnT}$ and $\Psi_{\tnB}$
are the isomorphisms~\eref{ETisom_e} for the top and bottom exact triples of
Fredholm operators in~\eref{ETisom_e5}, commutes
for every quasi-isomorphism of exact triples of Fredholm operators as in~\eref{ETisom_e5}.\\
\end{mywidth2} 

\begin{mywidth2}
{\bf\emph{Normalization II.}}\label{NormalII_prop} 
For each $\ft\!\in\!\cT(X,Y;X',Y';X'',Y'')$ as in \eref{cTdiag_e}
with $D'\!\in\!\cF^*(X',Y')$ and $D''\!\in\!\cF^*(X'',Y'')$, 
the restriction $\Psi_{\ft}$ of~$\Psi$ to the fiber over~$\ft$
is the canonical isomorphism $\w_{\ka(D)}$ of Lemma~\ref{ses_lmm} 
for the short exact sequence
$$0\lra\ka(D')\lra\ka(D)\lra\ka(D'')\lra0$$
of finite-dimensional vector spaces.\\
\end{mywidth2}

\begin{mywidth2}
{\bf\emph{Normalization III.}}\label{NormalIII_prop} 
For each $\star\!=\!',''$ and $\ft\!\in\!\cT^{\star}(X,Y;X',Y';X'',Y'')$,
the restriction $\Psi_{\ft}$ of~$\Psi$ to the fiber over~$\ft$ is the corresponding 
isomorphism~$\cI_{\ft}^{\star}$ of~\eref{cIft1_e} or~\eref{cIft2_e}.\\
\end{mywidth2}

\noindent
With $\io_{\oplus}$ as below~\eref{sumEX_e},  $D'\!\in\!\cF(X',Y')$, and 
$D''\!\in\!\cF(X'',Y'')$, let
$$\wt\oplus_{D',D''}\equiv\Psi_{\io_{\oplus}(D', D'')}\!: 
\la(D')\otimes\la(D'')\lra\la(D'\!\oplus\!D'')$$
be the corresponding exact triples isomorphism~\eref{ETisom_e}.
With $\io_C$ as below~\eref{compEX_e}, $D_1\!\in\!\cF(X_1,X_2)$, and $D_2\!\in\!\cF(X_2,X_3)$,
let
$$\wt\cC_{D_1,D_2}\!\equiv\!\Psi_{\io_{\cC}(D_1,D_2)}\!:~\la(D_1)\!\otimes\!\la(D_2)
 \lra \la(D_2\!\circ\!D_1)$$
be the corresponding isomorphism~\eref{ETisom_e}.
By the next four properties, these isomorphisms provide liftings
of~\eref{commut_e0}, \eref{commut_e0b}, \eref{assoc_e0}, and~\eref{cTcomm_e}  
to determinant line bundles.\\

\begin{mywidth2}
{\bf\emph{Direct Sums~I.}}\label{DirSumI_prop} 
For all $D'\!\in\!\cF(X',Y')$ and $D''\!\in\!\cF(X'',Y'')$, the diagram
\BE{DirSum_diag1}\begin{split}
\xymatrix{\la(D')\otimes\la(D'')
\ar[rr]^>>>>>>>>>>>{\wt\oplus_{D',D''}}  \ar[d]_R &&
\la(D'\!\oplus\!D'') \ar[d]|{\wt\cI_{R_{X',X''},R_{Y',Y''};D'\oplus D''}} \\
 \la(D'')\otimes \la(D')  \ar[rr]^{\wt\oplus_{D'',D'}}&&
\la(D''\oplus D')}
\end{split}\EE
commutes.\\
\end{mywidth2}

\begin{mywidth2}
{\bf\emph{Direct Sums~II.}}\label{DirSumII_prop}
For all $D'\!\in\!\cF(X',Y')$, $D''\!\in\!\cF(X'',Y'')$, and $D'''\!\in\!\cF(X''',Y''')$,
the diagram 
\BE{DirSum_diag2}\begin{split}
\xymatrix{\la(D')\otimes\la(D'')\otimes\la(D''')
\ar[rr]^>>>>>>>>>>>{\id\otimes{\wt\oplus_{D'',D'''}}}  
\ar[d]|{\wt\oplus_{D',D''}\otimes\id} &&
\la(D')\otimes\la(D''\!\oplus\!D''') \ar[d]|{\wt\oplus_{D',D''\oplus D'''}} \\
 \la(D'\!\oplus\!D'')\otimes \la(D''')  
 \ar[rr]^{\wt\oplus_{D'\oplus D'',D'''}}&&
\la(D'\!\oplus\!D''\!\oplus\!D''')}
\end{split}\EE
commutes.\\
\end{mywidth2}

\begin{mywidth2}
{\bf\emph{Compositions~I.}}\label{composI_prop}
For all 
$D_1\!\in\!\cF(X_1,X_2)$, $D_2\!\in\!\cF(X_2,X_3)$,
and $D_3\!\in\!\cF(X_3,X_4)$, the diagram
\BE{compdiag_e1}\begin{split}
\xymatrix{\la(D_1)\otimes\la(D_2)\otimes\la(D_3)
 \ar[rr]^>>>>>>>>>>>{\id\otimes \wt\cC_{D_2,D_3}} 
 \ar[d]|{\wt\cC_{D_1,D_2}\otimes\id} &&
\la(D_1)\otimes \la(D_3\!\circ\!D_2) \ar[d]|{\wt\cC_{D_1,D_3\circ D_2}} \\
\la(D_2\!\circ\!D_1)\otimes\la(D_3) \ar[rr]^{\wt\cC_{D_2\circ D_1,D_3}}&&
\la(D_3\!\circ D_2\circ\!D_1)}
\end{split}\EE
commutes.\\
\end{mywidth2}
 
\begin{mywidth2}
{\bf\emph{Compositions~II.}}\label{composII_prop}
For all 
\hbox{$\ft_1\!\in\!\cT(X_1,X_2;X_1',X_2';X_1'',X_2'')$} and
$\ft_2\!\in\!\cT(X_2,X_3;X_2',X_3';X_2'',X_3'')$ as in~\eref{cCcTdiag_e}, 
the diagram
\BE{compdiag_e2}\begin{split}
\xymatrix{\la(D_1')\otimes\la(D_1'')\otimes\la(D_2')\otimes\la(D_2'')
 \ar[d]|{\wt\cC_{D_1',D_2'}\otimes\wt\cC_{D_1'',D_2''}\circ\id\otimes R\otimes\id} 
 \ar[rr]^>>>>>>>>>>>{\Psi_{\ft_1}\otimes\Psi_{\ft_2}} &&
\la(D_1)\otimes\la(D_2) \ar[d]|{\wt{C}_{D_1,D_2}}\\
\la(D_2'\!\circ\!D_1')\otimes \la(D_2''\!\circ\!D_1'') 
\ar[rr]^>>>>>>>>>>>>>>>>>>>{\Psi_{\cC_{\cT}(\ft_1,\ft_2)}} && \la(D_2\circ\!D_1)}
\end{split}\EE
commutes.\\
\end{mywidth2}

\begin{mywidth2}
{\bf\emph{Complex Exact Triples.}}\label{CET_prop} 
If $X,Y,X',Y',X'',Y''$ are Banach vector spaces with complex structures
and $\ft\!\in\!\cT_{\C}(X,Y;X',Y';X'',Y'')$,
the restriction~$\Psi_{\ft}$ of~$\Psi$ to the fiber over~$\ft$ 
intertwines the complex orientations of $\la(D'),\la(D''),\la(D)$.\\
\end{mywidth2}

\begin{mythm}\label{main_thm}
There exist a collection of topologies on the line bundles ${\det}_{X,Y}\!\lra\!\cF(X,Y)$
corresponding to pairs $(X,Y)$ of Banach spaces 
and a collection of continuous line-bundle isomorphisms~\eref{ETisom_e} 
which satisfy the Naturality I,II, Complex Orientations,
Normalization I,II,III,
Direct Sums~I,II,  Compositions~I,II, and Complex Exact Triples properties. 
\end{mythm}

\noindent
We will refer to the Naturality~II property restricted to the isomorphisms
of exact triples of Fredholm operators as the \sf{isomorphism Naturality~II}
property.\\

\noindent
Any family of exact triple isomorphisms~$\Psi_{\ft}$ as in~\eref{ETisom_e} satisfying 
Normalization~II also satisfies Normalization~III for triples of surjective Fredholm operators
and Complex Exact Triples for 
triples of surjective $\C$-linear Fredholm operators.
By Lemma~\ref{surET_lmm}, such a family induces a continuous bundle map
over the subspace 
$$\cT^*(X,Y;X',Y';X'',Y'')\subset \cT(X,Y;X',Y';X'',Y'')$$
of exact triples as in~\eref{cTdiag_e} with surjective Fredholm operators $D,D',D''$
with respect to the topologies determined by the Normalization~I property.\\

\noindent
For an isomorphism $(\phi,\psi)\!:D\!\lra\!D'$ between Fredholm operators
$D\!:X\!\lra\!Y$ and $D'\!:X'\!\lra\!Y'$ as above~\eref{cIphipsi_e}, define
$$\wt\cI_{\phi,\psi;D}\!:\la(\phi^{-1})\otimes\la(D)\otimes\la(\psi)\lra 
\la(D'), \quad
\wt\cI_{\phi,\psi;D}
\big((1\otimes1^*)\otimes\si\otimes(1\otimes1^*)\big)
=\wt\cI_{\phi,\psi;D}(\si).$$
By Normalization~III, the diagram 
$$\xymatrix{\la(\phi^{-1})\otimes\la(D)\otimes\la(\psi)
 \ar[rr]^>>>>>>>>>>>{\id\otimes \wt\cC_{D,\psi}} 
 \ar[d]|{\wt\cC_{\phi^{-1},D}\otimes\id} \ar[drr]|{\wt\cI_{\phi,\psi;D}}
&&
\la(\phi^{-1})\otimes \la(\psi\!\circ\!D) 
\ar[d]|{\wt\cC_{\phi^{-1},\psi\circ D}} \\
\la(D\!\circ\!\phi^{-1})\otimes\la(\psi) 
\ar[rr]^{\wt\cC_{D\circ\phi^{-1},\psi}}&&
\la(\psi\!\circ\!D\circ\!\phi^{-1})\!=\!\la(D')}$$
commutes.\\

\noindent
Some of the properties listed in Theorem~\ref{main_thm}
are similarly implied by other properties:
\begin{enumerate}[label=$\bullet$,leftmargin=*]

\item Naturality~I follows from the continuity of~$\Psi$ in~\eref{ETisom_e} 
 and Normalization~III applied to the diagrams
$$\xymatrix{0\ar[r]& X\ar[d]^D\ar[r]^{\phi}& X'\ar[d]^{D'}\ar[r]& 
\{0\}\ar[d]\ar[r]&0\\
0\ar[r]& Y\ar[r]^{\psi}& Y'\ar[r]& \{0\}\ar[r]&0}$$

\item Complex Orientations follows from the continuity of~$\Psi$,
Normalization~I, and Complex Exact Triples;

\item the isomorphism Naturality~II property follows from Compositions~II 
and Normalization~II,III applied to the diagrams
$$\xymatrix{0\ar[r]& X_B'\ar[d]^{\phi'^{-1}}\ar[r]& 
X_B\ar[d]^{\phi^{-1}}\ar[r]& X_B''\ar[d]^{\phi''^{-1}}\ar[r]&0& 
0\ar[r]& X_B'\ar[d]|{D_T'\circ\phi'^{-1}}\ar[r]& X_B\ar[d]|{D_T\circ\phi^{-1}}\ar[r]& 
X_B''\ar[d]|{D_T''\circ\phi''^{-1}}\ar[r]&0\\
0\ar[r]& X_T'\ar[d]^{D_T'}\ar[r]& X_T\ar[d]^D\ar[r]& X_T''\ar[d]^{D_T''}\ar[r]&0& 
0\ar[r]& Y_T'\ar[d]^{\psi'}\ar[r]& Y_T\ar[d]^{\psi}\ar[r]&Y_T''\ar[d]^{\psi''}\ar[r]&0\\
0\ar[r]& Y_T'\ar[r]& Y_T\ar[r]& Y_T''\ar[r]&0&
0\ar[r]& Y_B'\ar[r]& Y_B\ar[r]& Y_B''\ar[r]&0}$$

\item the Exact Squares property below is implied 
by the Naturality~II, Normalization~II, and Compositions~II properties
(see Corollary~\ref{ExSQ_crl});

\item the two Direct Sums properties follow from the Exact Squares property applied to
the two diagrams in Figure~\ref{ESQvsComp_fig}
and the isomorphism Naturality~II property applied to the diagram
$$\xymatrix{&0\ar[r]&D''\ar[d]\ar[r]&  D'\!\oplus\!D''\ar[d]\ar[r]& D'\ar[r]\ar[d]&0\\
&0\ar[r]&D''\ar[r]&  D''\!\oplus\!D'\ar[r]& D'\ar[r]&0}$$

\item Compositions~II is implied by 
the Normalization~III and Exact Squares properties (see Subsection~\ref{KM_subs});

\item Compositions~I is implied by 
the isomorphism Naturality~II, Normalization~III, and 
Exact Squares properties (see Subsection~\ref{KM_subs});

\item the full Naturality~II property is implied by 
the isomorphism Naturality~II, Normalization~II,III, and Exact Squares properties
(see Subsection~\ref{classify_subs}).

\end{enumerate}

\vspace{.15in}

\noindent
By Proposition~\ref{overlap_prp}, 
a collection of exact triple isomorphisms~$\Psi_{\ft}$ as in~\eref{ETisom_e}
determines topologies on the line bundles~${\det}_{X,Y}$ 
which satisfy the Naturality~I and Normalization~I properties
if this collection satisfies the Normalization~II,III and Compositions~I,II properties.
By Corollary~\ref{ETcont_crl},
all these isomorphisms~$\Psi_{\ft}$ are continuous in the resulting topologies 
if in addition they satisfy the Naturality~II property.
By the above, the full Naturality and Compositions~I,II property can be replaced
by the isomorphism Naturality~II and Exact Squares properties.\\

\noindent
In summary, a collection of exact triple isomorphisms~$(\Psi_{\ft})_{\ft}$ as in~\eref{ETisom_e}
determines a system of topologies on the determinant line bundles~${\det}_{X,Y}$
which satisfy the Naturality~I and Normalization~I properties and in which 
the isomorphisms~$\Psi_{\ft}$ are continuous if this collection satisfies
the Normalization~II,III properties along with either
\begin{enumerate}[label=$\bullet$,leftmargin=*]

\item the Exact Squares property or

\item the Naturality~II and Compositions~I,II properties.

\end{enumerate}
In either case, such a collection of isomorphisms necessarily satisfies
the Exact Squares, Naturality~II, Compositions~I,II, and Direct Sums~I,II properties.
This is consistent with \cite[Theorem~1]{KM}; see Subsection~\ref{KM_subs} for more details.
If a collection of isomorphisms as above also satisfies the Complex Exact Triples properties,
then the resulting topologies satisfy the Complex Orientations property as~well.
A collection of isomorphisms~$\Psi_{\ft}$ satisfying all of the above properties
is specified by~\eref{cUDDdfn_e}.
Theorem~\ref{classify_thm} on page~\pageref{classify_thm} describes all other collections
of isomorphisms~$\Psi_{\ft}$ with these properties.\\

\begin{mywidth}
{\bf\emph{Exact Squares.}}\label{ESQs_prop}
For every commutative diagram 
\BE{SQexact_e1}\begin{split}
\xymatrix{&0\ar[d]&0\ar[d]&0\ar[d]&\\
0\ar[r]& D_{\tnT\tnL} \ar[r]^{\fI_{\tnT}}\ar[d]^{\fI_{\tnL}}& 
D_{\tnT\tnM} \ar[r]^{\fJ_{\tnT}}\ar[d]^{\fI_{\tnM}}& D_{\tnT\tnR}\ar[r]\ar[d]^{\fI_{\tnR}}& 0\\
0\ar[r]& D_{\tnC\tnL} \ar[r]^{\fI_{\tnC}}\ar[d]^{\fJ_{\tnL}}& 
D_{\tnC\tnM} \ar[r]^{\fJ_{\tnC}}\ar[d]^{\fI_{\tnM}}& 
D_{\tnC\tnR}\ar[r]\ar[d]^{\fJ_{\tnR}}& 0\\
0\ar[r]& D_{\tnB\tnL} \ar[r]^{\fI_{\tnB}}\ar[d]& 
D_{\tnB\tnM} \ar[r]^{\fJ_{\tnB}}\ar[d]& D_{\tnB\tnR}\ar[r]\ar[d]& 0\\
&0&0&0&}
\end{split}\EE
of exact rows and columns of Fredholm operators, the diagram 
\BE{SQexact_e2}\begin{split}
\xymatrix{ \la(D_{\tnT\tnL})\otimes\la(D_{\tnB\tnL}) 
 \otimes\la(D_{\tnT\tnR})\otimes \la(D_{\tnB\tnR})
\ar[d]|{~~\Psi_{\tnL}\otimes\Psi_{\tnR}} 
\ar[rrrr]^>>>>>>>>>>>>>>>>>>>>>{\Psi_{\tnT}\otimes\Psi_{\tnB}\,\circ\,\id \otimes R\otimes \id} 
&&&& \la(D_{\tnT\tnM})\otimes\la(D_{\tnB\tnM})\ar[d]^{\Psi_{\tnM}}\\ 
\la(D_{\tnC\tnL})\otimes\la(D_{\tnC\tnR})\ar[rrrr]^{\Psi_{\tnC}}&&&& \la(D_{\tnC\tnM})}
\end{split}\EE
of graded lines, where $\Psi_{\star}$ are the isomorphisms~\eref{ETisom_e}
corresponding to the \sf{t}op, \sf{c}enter, and \sf{b}ottom rows and
\sf{l}eft, \sf{m}iddle, and \sf{r}ight columns of the diagram~\eref{SQexact_e1},
commutes.\\
\end{mywidth}

\begin{figure}
\begin{gather*}
\xymatrix{&0\ar[d]&0\ar[d]&0\ar[d]&\\
0\ar[r]& 0 \ar[r]\ar[d]&  D' \ar[r]\ar[d]& D'\ar[r]\ar[d]& 0\\
0\ar[r]& D''\ar[r]\ar[d]& D'\!\oplus\!D'' \ar[r]\ar[d]& D'\ar[r]\ar[d]& 0\\ 
0\ar[r]& D''\ar[r]\ar[d]& D''\ar[r]\ar[d]& 0\ar[r]\ar[d]&\\
&0&0&0&\\
&0\ar[d]&0\ar[d]&0\ar[d]&\\
0\ar[r]& D' \ar[r]\ar[d]&  D' \ar[r]\ar[d]& 0\ar[r]\ar[d]& 0\\
0\ar[r]& D'\!\oplus\!D''\ar[r]\ar[d]& D'\!\oplus\!D''\!\oplus\!D''' \ar[r]\ar[d]& D'''\ar[r]\ar[d]& 0\\
0\ar[r]& D''\ar[r]\ar[d]& D''\!\oplus\!D'''\ar[r]\ar[d]& D'''\ar[r]\ar[d]& 0\\
&0&0&0&}
\end{gather*}
\caption{Exact squares of Fredholm operators corresponding to the two Direct Sums properties}
\label{ESQvsComp_fig}
\end{figure}

\noindent
By the proof of Lemma~\ref{NormalIdual_lmm}, the Normalization~I property can be replaced 
by a dual version.
Let 
$$\cF'(X,Y)\equiv\big\{D\!\in\!\cF(X,Y)\!:\,\ka(D)\!=\!0\big\}$$
be the space of injective Fredholm operators.
For each $D_0\!\in\!\cF'(X,Y)$, right inverse $S\!:\fc(D_0)\!\lra\!Y$ for 
$$q_{D_0}\!:\,Y\lra \fc(D_0), \qquad q_{D_0}(y)=y+\Im\,D_0,$$
and $D\!\in\!\cF(X,Y)$ sufficiently close to~$D_0$,
the homomorphism 
$$q_D\!\circ\!S\!:\,\fc(D_0)\lra\fc(D)$$
is an isomorphism and thus induces an isomorphism
\BE{cIdual_e}\wt\cI_{D_0,S;D}\!: \la(D_0)\lra\la(D)\,,
\quad \wt\cI_{D_0,S;D}(1\otimes\al)=  1\otimes \big(\al\!\circ\!\la(q_D\!\circ\!S)^{-1}\big)\,.\EE
Putting these isomorphisms together, we obtain a bundle map
\BE{cIdual_e2}\wt\cI_{D_0,S}\!:\,U_{D_0,S}\!\times\!\la(D_0)
\lra {\det}_{X,Y}|_{U_{D_0,S}}\EE
covering the identity on an open neighborhood $U_{D_0,S}$ of $D_0$ in $\cF'(X,Y)$.\\

\begin{mywidth}
{\bf\emph{Normalization I$'$.}}\label{NormalIpr_prop} The map $\wt\cI_{D_0,S}$ is continuous 
for every $D_0\!\in\!\cF'(X,Y)$, right inverse $S\!:\fc(D_0)\!\lra\!Y$ of~$q_{D_0}$, 
and sufficiently small open neighborhood $U_{D_0,S}$ of $D_0$ in $\cF'(X,Y)$.\\
\end{mywidth}

\noindent
The determinant line bundle is also compatible with dualizations of Fredholm operators.
Let~$\cD_D$ and~$\cP$ be as in~\eref{fkfcdual_e} and~\eref{cPisom_e}, respectively.\\

\begin{mywidth}
{\bf\emph{Dualizations.}} There exists a collection of (continuous) line bundle isomorphisms
\BE{wtcD_e}\wt\cD\!: {\det}_{X,Y}\lra \cD^*{\det}_{Y^*,X^*}\EE
over $\cF(X,Y)$ parametrized by pairs $(X,Y)$ 
of Banach vector spaces with the following properties.\\ 
\end{mywidth}

\begin{mywidth2}
{\bf\emph{Normalization~IV.}}\label{NormalIV_prop}
For every homomorphism $\de\!:L\!\lra\!\{0\}$ from a line,
$$\wt\cD_{\de}\big(x\!\otimes\!1^*\big)= 1\otimes \cP\big(\cD_{\de}(x)\big)
\qquad\forall~x\in\la(L)=L.$$
\end{mywidth2}

\vspace{.1in}

\begin{mywidth2}
{\bf\emph{Dual Exact Triples.}}\label{compdual_prop}
The isomorphisms~\eref{ETisom_e} and~\eref{wtcD_e} provide a lifting of \eref{cDcTcomm_e}
to determinant line bundles, i.e.~the diagram
\BE{compdiagdual_e}\begin{split}
\xymatrix{\la(D')\otimes\la(D'')
 \ar[d]|{\wt\cD_{D''}\otimes\wt\cD_{D'}\circ R}\ar[rr]^>>>>>>>>>{\Psi_{\ft}}&& 
\la(D)\ar[d]^{\wt\cD_D}\\
\la(D''^*)\otimes\la(D'^*)\ar[rr]^>>>>>>>>>{\Psi_{\ft^*}}&& \la(D^*)}
\end{split}\EE
commutes for every $\ft\!\in\!\cT(X,Y;X',Y';X'',Y'')$ as in \eref{cTdfn_e}.\\
\end{mywidth2}

\begin{mywidth}{\bf\emph{Dual Complex Orientations.}}\label{DualCOrient_prop}
If $X,Y$ are Banach vector spaces with complex structures and \hbox{$D\!\in\!\cF_{\C}(X,Y)$}, 
the isomorphism~\eref{wtcD_e} intertwines the complex orientations of~$\la(D)$
and~$\la(D^*)$.\\
\end{mywidth}

\noindent
By Corollary~\ref{compdual_crl} and Section~\ref{classify_subs}, 
each determinant line bundle system as in~Theorem~\ref{main_thm} on page~\pageref{main_thm}
 determines a unique
system of isomorphisms~$\wt\cD$ satisfying the above three properties.
Furthermore, there is a somewhat smaller family of determinant line bundle systems 
that satisfy a stronger
version of the Normalization~IV property:\\

\begin{mywidth2}
{\bf\emph{Normalization~IV$^{\star}$.}}\label{NormalIVst_prop}
For each $D\!\in\!\cF^*(X,Y)$, $\wt\cD_D$ is the canonical isomorphism 
induced by the first equation in~\eref{fkfcdual_e} and the pairing~\eref{cPisom_e}:
\BE{cDdfn_e0}\la(D)\lra\la(D^*), \qquad
x\!\otimes\!1^* \lra 1\otimes \cP\big(\la(\cD_D)x\big)\,.\EE
\end{mywidth2}

\vspace{.1in}

\noindent
By Lemma~\ref{surDual_lmm}, the isomorphisms~\eref{cDdfn_e0} give rise to a continuous
bundle map over $\cF^*(X,Y)$ for any system of topologies on determinant line bundles
as in Theorem~\ref{main_thm}.
In the proof of Corollary~\ref{compdual_crl}, we use this to show that 
the continuity of~\eref{wtcD_e} is implied by the Dual Exact Triples property.
However, the isomorphisms~\eref{wtcD_e} are 
compatible with the Normalization~IV$^{\star}$ property
only for some determinant line bundle systems,
including the one specified by the isomorphisms~$\Psi_{\ft}$ of~\eref{cUDDdfn_e}.\\

\noindent
The dualization isomorphisms~$\wt\cD_D$ given~by~\eref{cDdfn_e}
and the identity isomorphisms $A_{i,1}$ in~\eref{Aicisom_e} seem rather natural.
However, by Theorem~\ref{classify_thm} on page~\pageref{classify_thm}, the number of systems of topologies
on determinant line bundles compatible with these choices is still infinite.\\

\noindent
Combining the Dual Exact Triples property with the isomorphism Naturality~II property 
applied to the diagram
\BE{dualsumisom_e}\begin{split}
\xymatrix{&0\ar[r]&D''^*\ar[d]^{\id}\ar[r]^>>>>>{\fJ^*}&  
(D'\!\oplus\!D'')^*\ar[d]^{(T_X,T_Y)}\ar[r]^>>>>>{\fI^*}&  D'^*\ar[r]\ar[d]^{\id}&0
& T_X(\be)=(\be|_{Y''},\be|_{Y'})\\
&0\ar[r]&D''^*\ar[r]^>>>>>{\fI}&  D''^*\!\oplus\!D'^*\ar[r]^>>>>>{\fJ}&  
D'^*\ar[r]&0,& T_Y(\al)=(\al|_{X''},\al|_{X'})}
\end{split}\EE
where $\fI\!=\!(\fI_X,\fI_Y)$ and  $\fJ\!=\!(\fJ_X,\fJ_Y)$ 
are as in \eref{sumEX_e}, we find that the diagram
$$\xymatrix{\la(D')\otimes\la(D'')
 \ar[d]|{\wt\cD_{D''}\otimes\wt\cD_{D'}\circ R}\ar[rrr]^>>>>>>>>>>>>>>>{\wt\oplus_{D',D''}}&&& 
\la(D'\!\oplus\!D'')\ar[d]|{\wt\cI_{T_X,T_Y;(D'\oplus D'')^*}\circ \wt\cD_{D'\oplus D''} }\\
\la(D''^*)\otimes\la(D'^*)\ar[rrr]^>>>>>>>>>>>>>>>{\wt\oplus_{D''^*,D'^*}}&&& 
\la(D''^*\!\oplus\!D'^*)}$$
commutes, i.e.~the dualization and direct sum isomorphisms, $\wt\cD$ and $\wt\oplus$,
on the determinant lines are compatible.
Combining the Dual Exact Triples property with the isomorphism Naturality~II property applied 
to~\eref{dualsumisom_e} with $(D',D'')\!=\!(D_2\!\circ\!D_1,\id_{X_2})$, 
we find that the diagram
$$\xymatrix{\la(D_1)\otimes\la(D_2)
 \ar[d]|{\wt\cD_{D_2}\otimes\wt\cD_{D_1}\circ R}\ar[rrr]^>>>>>>>>>>>>>>>{\wt\cC_{D_1,D_2}}&&& 
\la(D_2\!\circ\!D_1)\ar[d]|{\wt\cD_{D_2\circ D_1} }\\
\la(D_2^*)\otimes\la(D_1^*)\ar[rrr]^>>>>>>>>>>>>>>>{\wt\cC_{D_2^*,D_1^*}}&&& 
\la(D_1^*\!\circ\!D_2^*)}$$
commutes, i.e.~the dualization and composition isomorphisms, $\wt\cD$ and $\wt\cC$,
on the determinant lines are compatible.\\

\noindent
Section~\ref{ET_subs} provides explicit formulas for the above isomorphisms $\Psi_{\ft}$,
 $\wt\oplus_{D',D''}$, $\wt\cC_{D_1,D_2}$, and~$\wt\cD_D$; 
see \eref{cUDDdfn_e},  \eref{cUDDdfn_e1}, \eref{cUDDdfn_e2}, and~\eref{cDdfn_e}, respectively.
Such formulas may be useful in some applications.

\section{Conceptual considerations and comparison of conventions}
\label{concept_sec}

\subsection{Topologizing determinant line bundles}
\label{outline_subs}

\noindent
For any Banach vector spaces $X$ and $Y$, the overlap maps 
between the trivializations~$\wt\cI_{D,T}$ of ${\det}_{X,Y}$
in~\eref{surjtriv_e} are continuous.
Thus, the trivializations~$\wt\cI_{D,T}$ topologize ${\det}_{X,Y}\big|_{\cF^*(X,Y)}$
as a line bundle over $\cF^*(X,Y)$, as required by the Normalization~I property
on page~\pageref{NormalI_prop}.
By Lemma~\ref{surET_lmm}, the resulting topology is compatible with
the Normalization~II property on page~\pageref{NormalII_prop}.\\ 

\noindent
For any Banach vector space $X$ and $N\!\in\!\Z^{\ge0}$, let 
$\io_{X;N}\!:X\!\lra\!X\!\oplus\!\R^N$ be the natural inclusion.
If $Y$ is another Banach vector space, $D\!\in\!\cF(X,Y)$, and 
$\Th\!:\R^N\!\lra\!Y$ is any homomorphism, define
$$\io_{\Th}\!:\cF(X,Y)\lra\cF(X\!\oplus\!\R^N,Y) \qquad\hbox{by}\quad
\io_{\Th}(D)=D_{\Th},~~D_{\Th}(x,u)=Dx+\Th(u);$$
the map $\io_{\Th}$ is an embedding.
The exact triple 
\BE{cIhatdiag_e}\begin{split}
\xymatrix{ 0\ar[r]& X\ar[d]^D\ar[r]^<<<<<{\io_{X;N}}& 
X\oplus\R^N \ar[d]^{D_{\Th}}\ar[r]^<<<<{\pi_2}&  \R^N \ar[d]\ar[r] & 0\\
0\ar[r]& Y \ar[r]^{\id_Y}& Y\ar[r]& 0\ar[r]& 0}
\end{split}\EE
and \eref{ETisom_e} give rise to the isomorphism
\BE{cIThD_e2}\hat\cI_{\Th;D}\!: \la(D)\lra \la(D_{\Th}), \qquad
\hat\cI_{\Th;D}(\si)=
\Psi_{\ft}\big(\si\otimes \OmN\!\otimes\!1^*\big),\EE
where $\OmN$ is the standard volume tensor on~$\R^N$, i.e.
$$\OmN= e_1\w\ldots\w e_N$$
if $e_1,\ldots,e_N$ is the standard basis for~$\R^N$.\\

\noindent
By the continuity requirement on the family of isomorphisms~$\Psi_{\ft}$ in~\eref{ETisom_e}
and the Normalization~I property, the isomorphisms $\hat\cI_{\Th;D}$ topologize 
${\det}_{X,Y}$ over the open subset
$$U_{X;\Th}\equiv\big\{D\!\in\!\cF(X,Y)\!:\,\fc(D_{\Th})=0\big\}.$$
Since these open subsets cover $\cF(X,Y)$ as $\Th$ ranges over all 
homomorphisms $\R^N\!\lra\!Y$ and $N$ ranges over all nonnegative integers,
the isomorphisms $\hat\cI_{\Th;D}$ completely specify the topology on~${\det}_{X,Y}$.
However, the overlap map
$$\label{cITh2Th1D_e}
\hat\cI_{\Th_2;D}\!\circ\!\hat\cI_{\Th_1;D}^{-1}\!:
\io_{\Th_1}^*{\det}_{X\oplus\R^{N_1},Y}\lra \io_{\Th_2}^*{\det}_{X\oplus \R^{N_2},Y}$$
must be continuous over $U_{X;\Th_1}\!\cap\!U_{X;\Th_2}$ for any pair of
homomorphisms $\Th_1\!:\R^{N_1}\!\lra\!Y$ and $\Th_2\!:\R^{N_2}\!\lra\!Y$.
By Proposition~\ref{overlap_prp}, this is indeed the case 
if the collection of isomorphisms~$\Psi_{\ft}$ satisfies 
the Normalization~II,III and Compositions~I,II properties. 
By Corollary~\ref{ETcont_crl}, every isomorphism in such a collection is
continuous in the resulting topologies if this collection also 
satisfies the Naturality~II property.
A collection of isomorphisms~$\Psi_{\ft}$ satisfying all these properties
is provided by~\eref{cUDDdfn_e}.
All other such collections are described by Theorem~\ref{classify_thm}
on page~\pageref{classify_thm}.\\

\noindent
For $D\!\in\!U_{X;\Th}$, the exact triple~\eref{cIhatdiag_e} induces an exact sequence
$$0\lra \ka(D)\lra\ka(D_{\Th})\lra \R^N\lra\fc(D)\lra0$$
of vector spaces.
A homomorphism $\de:V\!\lra\!W$ between finite-dimensional vector spaces also 
induces an exact sequence
$$0\lra \ka(\de)\lra V\stackrel{\de}{\lra} W\lra\fc(\de)\lra0$$
of vector spaces.
There is an isomorphism
\BE{detisom_e} \cI_{\de}\!:\la(\de)\lra \la({\bf 0})\equiv\la(V)\otimes\la^*(W).\EE
As suggested in \cite{Quillen}, 
a suitable collection of these isomorphisms is fundamental to constructing
a system of determinant line bundles for Fredholm operators.
Unfortunately, \cite{Quillen} makes no mention of what properties 
of a system of isomorphisms~\eref{detisom_e} are needed for such a construction
and gives no explicit formula for these isomorphisms.
The discussion in~\cite{Quillen} is also limited to Cauchy-Riemann operators on
Riemann surfaces.\\

\noindent
In the convention~\eref{DLdfn_e}, 
which is also used in \cite[Section~20.2]{KrMr} and \cite[Section~7.4]{MW},
the exact triple isomorphisms~\eref{cUDDdfn_e} correspond to the isomorphisms~\eref{detisom_e} 
given~by
\begin{gather}\label{detisom_e2}
\cI_{\de}\!:\la(\de)\lra \la({\bf 0}), \qquad 
x\otimes y^*\lra (-1)^{(\fd(W)-\fd(\fc(\de)))\fd(\fc(\de))}
(x\!\w_V\!v)\otimes \big(\la(\de)v\w_Wy\big)^*,\\
\qquad \forall~x\in\la\big(\ka(\de)\big)-0,~
y\in\la\big(\fc(\de)\big)-0,~v\in\la\bigg(\frac{V}{\ka(\de)}\bigg)-0,\notag 
\end{gather}
where $\w_V$ and $\w_W$ are the isomorphisms of Lemma~\ref{ses_lmm}; see Remark~\ref{4term_rmk}.
This is precisely the isomorphism of \cite[Lemma~7.4.7]{MW}
and is used directly to topologize determinant line bundles in the proof of 
\cite[Proposition~7.4.8]{MW}.\footnote{As shown in the proof of Proposition~\ref{overlap_prp}
in this paper, the restriction to injective homomorphisms~$\Th$ in 
the proof of \cite[Proposition~7.4.8]{MW} is unnecessary.}
While the properties of~\eref{detisom_e2} necessary for this construction
are verified in~\cite{MW}, few of the important properties of 
the resulting determinant line bundles are checked in~\cite{MW}.
The isomorphism~\eref{detisom_e2} appears only indirectly in the construction
of this paper.\\

\noindent
There are alternative ways of constructing a system of determinant line bundles 
satisfying the properties in Subsection~\ref{detLBprop_subs}.
\begin{enumerate}[label=(\arabic*),leftmargin=*]

\item A system of determinant line bundles for bounded complexes of vector bundles 
and isomorphisms for exact triples of such complexes is constructed in 
\cite[Chapter~I]{KM}.
A system of determinant line bundles for Fredholm operators can then be obtained
by associating each Fredholm operator with a two-term complex, 
deducing the Exact Squares property for Fredholm operators from that for 
bounded complexes and the two algebraic Compositions properties from the Exact Squares
property, and deriving explicit formulas for all isomorphisms.
This approach is described in detail in Section~\ref{KM_subs}.

\item One could explicitly specify a collection of isomorphisms~$\hat\cI_{\Th;D}$ 
as in~\eref{cIThD_e2} that are compatible with compositions.
This is essentially the approach taken in~\cite{MS}, \cite{MW}, \cite{Salamon},
and \cite{Seidel} to topologize determinant line bundles,
without verifying the properties in Subsection~\ref{detLBprop_subs}. 
The isomorphisms~\eref{cIThD_e2} can be used to define Exact Triples 
isomorphisms~\eref{ETisom_e} from the Normalization~II property, 
imposing the commutativity property of Lemma~\ref{StabTrip_lmm} by definition,
and to derive an explicit formula for these isomorphisms.
The Exact Squares property for Fredholm operators can then be obtained
from the basic Exact Squares property of Lemma~\ref{ses_lmm2}
as in the proof of Corollary~\ref{ExSQ_crl} and used to confirm
the two algebraic Compositions properties.

\item The commutativity property of Lemma~\ref{StabTrip_lmm} could be verified
for the isomorphism~\eref{cUDDdfn_e} directly, without using Proposition~\ref{CompET_prp2},
and used to obtain the Exact Squares property as in the proof of Corollary~\ref{ExSQ_crl}.
The two algebraic Compositions properties 
could then be deduced either from the Exact Squares property
 or from the corresponding properties for vector spaces
by an argument similar to the proof of Corollary~\ref{ExSQ_crl}.
Unfortunately, the proof of the special case of Proposition~\ref{CompET_prp2}
corresponding to Lemma~\ref{StabTrip_lmm} is as elaborate as the proof of 
Proposition~\ref{CompET_prp2} itself;
the former involves a bit less notation, but exactly the same steps.

\end{enumerate}

\vspace{.15in}

\noindent
In all three approaches, the Dual Exact Triples property can be either checked directly 
or deduced from more general  considerations.
The above listed alternatives can be used to replace parts of Section~\ref{lin_alg}
in this paper,
but most of Section~\ref{mainthmpf_sec} would still be needed. 
It appears the overall approach of this paper is more efficient than
the three alternatives described above.\\

\noindent
The equivalence of the topologies arising from the algebraic approach of~\cite{KM}
and the analytic approach of~\cite{Quillen} in many complex-geometric settings
is shown in the trilogy \cite{BGS1,BGS2,BGS3};
see \cite[Theorem~0.1]{BGS1} in particular. 
Combined with earlier work \cite{GS1,GS2}, this trilogy leads to
an arithmetic version of the Grothendick-Riemann-Roch Theorem; see \cite{GS3}.
A thorough discussion of the determinant line bundle in Arakelov geometry,
which is outside of the scope of this paper,
is contained in the books \cite{Faltings,SouleBook}.

\subsection{Relation with Knudsen-Mumford}
\label{KM_subs}

\noindent
The existence of a determinant line bundle system
satisfying the properties in Subsection~\ref{detLBprop_subs} 
follows most readily (but still with some work) from the proof of \cite[Theorem~1]{KM}, 
which constructs determinant line bundles for bounded complexes of vector bundles.
Unfortunately, a complete construction of a determinant line bundle based on~\cite{KM}
with a verification of all of the properties listed in Subsection~\ref{detLBprop_subs} and
with explicit formulas for the relevant isomorphisms does not seem to appear elsewhere;
we describe it below.\\

\noindent
For each homomorphism $\Th\!:\R^N\!\lra\!Y$, 
$$\cK_{\Th}\equiv \big\{(D,x,u)\!\in\!U_{X;\Th}\times X\!\oplus\!\R^N\!:\,
(x,u)\!\in\!\ka(D_{\Th})\big\}\lra U_{X;\Th}$$
is a vector bundle.
For each $D\!\in\!U_{X;\Th}$, the commutative diagram~\eref{cIhatdiag_e} gives rise to
an exact sequence
\BE{KMseq_e}0\lra \ka(D)\lra \ka(D_{\Th})\stackrel{\de_{\Th}}{\lra} 
\R^N\stackrel{\Th}{\lra} \fc(D)\lra0.\EE
Thus, each homomorphism $\Th\!:\R^N\!\lra\!Y$ determines a two-term graded complex
\BE{VBcomp_e}
\ldots\lra0\lra\cK_{\Th}\stackrel{\de_{\Th}}{\lra} U_{X;\Th}\!\times\!\R^N\lra0\lra\ldots
\EE
of vector bundles over $U_{X;\Th}$, with $\cK_{\Th}$ placed at the 0-th position,
and a $\Z_2$-graded line bundle
$$\cL_{\Th}\equiv\la(\cK_{\Th})\otimes\la^*(U_{X;\Th}\!\times\!\R^N),$$
\textsf{the determinant line bundle of the two-term complex \eref{VBcomp_e}}.\\

\noindent
For each $D\!\in\!U_{X;\Th}$, let $\Xi\!:\fc(D)\!\lra\!\R^N$ be a right inverse 
for the surjective map 
\BE{ThfcD_e}\R^N\lra \fc(D), \qquad u\lra\Th(u)+\Im\,D.\EE
The diagram
$$\xymatrix{\ldots\ar[r]& 0 \ar[d]\ar[r]& 
\ka(D) \ar[d]^{\fI_{D;X}}\ar[r]^>>>>>>0&
\fc(D) \ar[d]^{\Xi}\ar[r]&  0 \ar[d]\ar[r]&\ldots& \fI_{D;X}(x)=(D,x,0)\\
\ldots\ar[r]& 0 \ar[r]& \cK_{\Th}|_D \ar[r]^>>>>>>{\de_{\Th}}&
\{D\}\!\times\!\R^N \ar[r]&  0 \ar[r]&\ldots}$$
is then a \textsf{quasi-isomorphism of graded complexes} over~$\{D\}$, 
i.e.~a homomorphism of graded complexes of vector bundles that 
induces an isomorphism in homology.
By \cite[Theorem~1]{KM}, there is then a canonical isomorphism
$$\hat\cI_{\Th;D}'\!: \la(D)\lra  \cL_{\Th}|_D\approx\la(D_{\Th}).$$
Since any other right inverse for the homomorphism~\eref{ThfcD_e} is of the form
$\Xi+\de_{\Th}\ti\Xi$ for some homomorphism $\ti\Xi\!:\fc(D)\!\lra\!\ka(D_{\Th})$,
$\hat\cI_{\Th;D}'$ is independent of the choice of~$\Xi$ by \cite[Proposition~2]{KM}.
If $\Th'\!:\R^{N'}\!\lra\!Y$ is another homomorphism and $\io\!:\R^N\!\lra\!\R^{N'}$
is a homomorphism such that $\Th\!=\!\Th'\!\circ\!\io$,
$$\xymatrix{\ldots\ar[r]& 0 \ar[d]\ar[r]& 
\cK_{\Th} \ar[d]|{\id\times\id\times\io}\ar[r]^>>>>>>{\de_{\Th}}&
U_{X;\Th}\!\times\!\R^N \ar[d]|{\id\times\io}\ar[r]&  0 \ar[d]\ar[r]&\ldots\\
\ldots\ar[r]& 0 \ar[r]& \cK_{\Th'} \ar[r]^>>>>>>{\de_{\Th'}}&
U_{X;\Th}\!\times\!\R^{N'} \ar[r]&  0 \ar[r]&\ldots}$$
is also a quasi-isomorphism of graded complexes.
By the proof of \cite[Theorem~1]{KM}, it also induces a canonical isomorphism
$$\cI_{\Th',\Th}\!: \cL_{\Th}\lra\cL_{\Th'}$$
of line bundles over $U_{X;\Th}$. 
By the functoriality of the determinant construction of \cite[Theorem~1]{KM},
$$\hat\cI_{\Th';D}'=\cI_{\Th',\Th}\!\circ\!\hat\cI_{\Th;D}'\!: 
\la(D)\lra  \cL_{\Th'}|_D\approx\la(D_{\Th'}).$$
Since the line bundle maps $\cI_{\Th',\Th}$ are continuous, 
the isomorphisms $\hat\cI_{\Th;D}'$ topologize ${\det}_{X,Y}$ over $U_{X;\Th}$
and endow ${\det}_{X,Y}$ with a well-defined topology of a line bundle over~$\cF(X,Y)$,
which satisfies the Normalization~I and Naturality~I properties.\\

\noindent
The proof of \cite[Theorem~1]{KM} produces analogues of the isomorphisms~\eref{ETisom_e}
for exact triples of graded complexes~\eref{VBcomp_e} of vector bundles.
These isomorphisms satisfy analogues of the Normalization~II,III, Naturality~II,
and Exact Squares properties.
By the proof of Corollary~\ref{ETcont_crl}, an exact triple of Fredholm operators gives
rise to an exact triple of two-term complexes (over a point).
By the analogue of the Naturality~II property for two-term complexes, 
the isomorphisms of \cite[Theorem~1]{KM} 
then induce via the isomorphisms~$\hat\cI_{\Th;D}'$ isomorphisms~$\Psi_{\ft}$ for exact triples of 
Fredholm operators which satisfy the Normalization~II,III and Naturality~II properties.
These isomorphisms depend continuously on~$\ft$ by the proofs of Lemma~\ref{surET_lmm}
and Corollary~\ref{ETcont_crl}.
By the proof of Corollary~\ref{ExSQ_crl}, an exact square of Fredholm operators as 
in~\eref{SQexact_e1} gives rise to an exact square of two-term complexes.
By the analogue of the Exact Squares property for two-term complexes and 
the proof of Corollary~\ref{ExSQ_crl}, 
the induced isomorphisms for exact triples of Fredholm operators satisfy 
the Exact Squares property for Fredholm operators.
The proof of \cite[Theorem~1]{KM} implies the existence of 
the bundle maps~$\wt\cD_D$ as in~\eref{wtcD_e} satisfying the analogue of 
the Dual Exact Triples property on page~\pageref{compdual_prop} for two-term complexes.
These bundle maps~$\wt\cD_D$ satisfy the analogue of the Normalization~IV$^{\star}$
property on page~\pageref{NormalIVst_prop} in the case of the system explicitly
constructed in the proof of \cite[Theorem~1]{KM}; this can be seen
from the last paragraph of this section and Section~\ref{classify_subs}.\\

\noindent
We show below that the Compositions~I,II properties on page~\pageref{composI_prop}
follow from the isomorphism Naturality~II, Normalization~III, and Exact Squares properties.
By Subsection~\ref{classify_subs}, the full Naturality~II property follows from
the isomorphism Naturality~II, Normalization~II,III, and Exact Squares properties.
Thus, \cite[Theorem~1]{KM} gives rise to a determinant line bundle 
system satisfying all properties in Subsection~\ref{detLBprop_subs},
with the possible exceptions of the Complex Orientations, Complex Exact Triples,
and Dual Complex Orientations properties.

\begin{proof}[{\bf\emph{Exact Squares and Normalization~III imply Compositions~II}}]
Let~$\ft_1$ and~$\ft_2$ be exact triples as in~\eref{compdiag_e2}.
For  $\star\!=\!',''$ or blank, let
$$\Psi^{\star}\!:\la(D_1^{\star})\!\otimes\!\la(D_2^{\star})
\lra \la\big(D_2^{\star}\!\circ\!D_1^{\star}\!\oplus\!\id_{X_2^{\star}}\big)$$ 
be the isomorphism~\eref{ETisom_e} 
for the exact triple~\eref{compEX_e} corresponding to the composition 
$D_2^{\star}\!\circ\!D_1^{\star}$.
For \hbox{$i\!=\!1,3$}, define
$$\io_{X_i^{\star};X_2^{\star}}\!:X_i^{\star}\lra X_i^{\star}\!\oplus\!X_2^{\star}, 
\qquad \io_{X_i^{\star};X_2^{\star}}(x)=(x,0).$$
Thus, $(\io_{X_1^{\star};X_2^{\star}},\io_{X_3^{\star};X_2^{\star}})$ is 
a quasi-isomorphism from $D_2^{\star}\!\circ\!D_1^{\star}$ to
$D_2^{\star}\!\circ\!D_1^{\star}\!\oplus\!\id_{X_2^{\star}}$.
Let
$$\cI^{\star}\!=\!
\wt\cI_{\io_{X_1^{\star};X_2^{\star}},\io_{X_3^{\star};X_2^{\star}};
D_2^{\star}\circ D_1^{\star}}\!:\la(D_2^{\star}\!\circ\!D_1^{\star})\lra 
\la\big(D_2^{\star}\!\circ\!D_1^{\star}\!\oplus\!\id_{X_2^{\star}}\big)$$
be the corresponding isomorphism~\eref{cIphipsi_e}.
In the top diagram of Figure~\ref{KM_fig2},
the rows are the exact triple~$\ft_1$, 
the direct sum of~$\cC_{\cT}(\ft_1,\ft_2)$ with 
$$0\lra\id_{X_2'}\lra\id_{X_2}\lra\id_{X_2''}\lra0,$$
and the exact triple~$\ft_2$.
The columns in this commutative square of Fredholm operators are
the exact triples~\eref{compEX_e} corresponding to the compositions 
$D_2'\!\circ\!D_1'$, $D_2\!\circ\!D_1$, and~$D_2''\!\circ\!D_1''$.\\

\begin{figure}
\begin{gather*}
\xymatrix{& 0\ar[d] & 0\ar[d] & 0\ar[d] &\\ 
0\ar[r]& D_1' \ar[r]\ar[d]& D_1\ar[r]\ar[d]& D_1''\ar[r]\ar[d]& 0\\
0\ar[r]& D_2'\!\circ\!D_1'\oplus\id_{X_2'}\ar[r]\ar[d]& 
D_2\!\circ\!D_1\oplus\id_{X_2} \ar[r]\ar[d]&
D_2''\!\circ\!D_1''\oplus\id_{X_2''} \ar[r]\ar[d]& 0\\
0\ar[r]& D_2' \ar[r]\ar[d]& D_2\ar[r]\ar[d]& D_2''\ar[r]\ar[d]& 0\\
&0&0&0& \\
& 0\ar[d] & 0\ar[d] & 0\ar[d] &\\ 
0\ar[r]& D_2'\!\circ\!D_1' \ar[r]\ar[d]& D_2\!\circ\!D_1\ar[r]\ar[d]& 
D_2''\!\circ\!D_1''\ar[r]\ar[d]& 0\\
0\ar[r]& D_2'\!\circ\!D_1'\oplus\id_{X_2'}\ar[r]\ar[d]& 
D_2\!\circ\!D_1\oplus\id_{X_2} \ar[r]\ar[d]&
D_2''\!\circ\!D_1''\oplus\id_{X_2''} \ar[r]\ar[d]& 0\\
0\ar[r]& \id_{X_2'} \ar[r]\ar[d]& \id_{X_2}\ar[r]\ar[d]& \id_{X_2''}\ar[r]\ar[d]& 0\\
&0&0&0&}
\end{gather*}
\begin{gather*}
\xymatrix{\la(D_1')\!\otimes\!\la(D_1'')\!\otimes\!\la(D_2')\!\otimes\!\la(D_2'')
\ar[rr]^>>>>>>>>>>>>>>{\Psi_{\ft_1}\otimes\Psi_{\ft_2}}
\ar[d]|{\Psi'\!\otimes\!\Psi''\!\circ\id\otimes R\otimes\id} \ar@/_9pc/[dd]&& 
\la(D_1)\otimes\la(D_2)\ar[d]^{\Psi}\ar@/^6pc/[dd]|{\wt\cC_{D_1,D_2}}\\
\la(D_2'\!\circ\!D_1'\!\oplus\!\id_{X_2'})\otimes
\la(D_2''\!\circ\!D_1''\!\oplus\!\id_{X_2''})\ar[rr]&& \la(D_2\!\circ\!D_1\!\oplus\!\id_{X_2})\\ 
\la(D_2'\!\circ\!D_1')\!\otimes\!\la(D_2''\!\circ\!D_1'')\ar[u]|{\cI'\otimes\cI''} 
\ar[rr]^>>>>>>>>>>>>>>>>>>{\Psi_{\cC_{\cT}(\ft_1,\ft_2)}}
&&\la(D_2\!\circ\!D_1)\ar[u]_{\cI}}
\end{gather*}
\caption{Derivation of the Compositions~II property
from the Exact Squares and Normalization~III properties}
\label{KM_fig2}
\end{figure}

\noindent
Applying the Exact Squares property to this diagram,
we obtain the top commutative square in the last diagram in Figure~\ref{KM_fig2}.
Applying the Exact Squares and Normalization~III properties to the center 
diagram in Figure~\ref{KM_fig2} and using the identification 
\BE{laext_e}\la(D_2^{\star}\!\circ\!D_1^{\star})\lra 
\la(D_2^{\star}\!\circ\!D_1^{\star})\!\otimes\!\la(\id_{X_2^{\star}})\,,\qquad
\si\lra \si\otimes1\!\otimes\!1^*\,,\EE
with $\star\!=\!',''$ or blank, 
we obtain the bottom commutative square in this diagram.
The two round arrows are the vertical arrows in~\eref{compdiag_e2};
the two half-disk diagrams commute by the definition 
of~$\wt\cC_{D_1^{\star},D_2^{\star}}$.
Thus, the diagram~\eref{compdiag_e2}, which consists of the outermost arrows
in the last diagram in Figure~\ref{KM_fig2}, commutes.
\end{proof}

\noindent
{\bf\emph{Exact Squares, Normalization~III, and isomorphism Naturality~II imply 
Compositions~I.}}
Let $D_1,D_2,D_3$ be Fredholm operators as in~\eref{compdiag_e1}.
Let
\begin{alignat*}{2}
\Psi_{1,2}\!:\la(D_1)\!\otimes\!\la(D_2)&\lra\la(D_2\!\circ\!D_1\!\oplus\!\id_{X_2}),
&~~
\Psi_{1,23}\!:\la(D_1)\!\otimes\!\la(D_3\!\circ\!D_2)
&\lra\la(D_3\!\circ\!D_2\!\circ\!D_1\oplus\!\id_{X_2}),\\
\Psi_{2,3}\!:\la(D_2)\!\otimes\!\la(D_3) &\lra\la(D_2\!\circ\!D_3\!\oplus\!\id_{X_3}),
&~~
\Psi_{12,3}\!:\la(D_2\!\circ\!D_1)\!\otimes\!\la(D_3)
&\lra\la(D_3\!\circ\!D_2\!\circ\!D_1\oplus\!\id_{X_3})
\end{alignat*}
be the isomorphisms~\eref{ETisom_e} for the exact triple~\eref{compEX_e} 
corresponding to the compositions  $D_2\!\circ\!D_1$, $(D_3\!\circ\!D_2)\!\circ\!D_1$,
$D_3\!\circ\!D_2$, and $D_3\!\circ\!(D_2\!\circ\!D_1)$.
We define isomorphisms
\begin{alignat*}{2}
\cI_{1,2}\!:\la(D_2\!\circ\!D_1)&\lra 
\la\big(D_2\!\circ\!D_1\!\oplus\!\id_{X_2}\big), 
&\quad
\cI_{1,23}\!:\la(D_3\!\circ\!D_2\!\circ\!D_1)&\lra 
\la\big(D_3\!\circ\!D_2\!\circ\!D_1\!\oplus\!\id_{X_2}\big), \\
\cI_{2,3}\!:\la(D_3\!\circ\!D_2)&\lra 
\la\big(D_3\!\circ\!D_2\!\oplus\!\id_{X_3}\big), 
&\quad
\cI_{12,3}\!:\la(D_3\!\circ\!D_2\!\circ\!D_1)&\lra 
\la\big(D_3\!\circ\!D_2\!\circ\!D_1\!\oplus\!\id_{X_3}\big)
\end{alignat*}
analogously to $\cI^{\star}$ above.\\

\noindent
The left column in the top diagram of Figure~\ref{KM_fig1},
the bottom row in this diagram, and the top row 
in the middle diagram of Figure~\ref{KM_fig1}
are the exact triples~\eref{compEX_e} corresponding to the compositions 
$D_2\!\circ\!D_1$, $D_3\!\circ\!D_2$, and $D_3\!\circ\!(D_2\!\circ\!D_1)$,
respectively.
The middle rows in these two figures are the direct sum of the exact triple~\eref{compEX_e} 
corresponding to the composition $D_3\!\circ\!(D_2\!\circ\!D_1)$ with the exact triple
$$0\lra\id_{X_2}\lra \id_{X_2}\lra0\lra0.$$
The homomorphisms $\fI\!\equiv\!(\fI_X,\fI_Y)$ and $\fI\!\equiv\!(\fJ_X,\fJ_Y)$
in the center column of the top diagram of Figure~\ref{KM_fig1} are given~by
\begin{alignat*}{2}
\fI_X(x_1)&=(x_1,D_1x_1,D_2D_1x_1), &\qquad 
\fJ_X(x_1,x_2,x_3)&=(D_1x_1\!-\!x_2,x_3\!-\!D_2x_2),\\
\fI_Y(x_2)&=(D_3D_2x_2,x_2,D_2x_2),&\qquad 
\fJ_Y(x_4,x_2,x_3)&=(x_4\!-\!D_3D_2x_2,x_3\!-\!D_2x_2).
\end{alignat*}
The left and center columns in the middle and bottom diagrams of this figure
and the top and middle rows in the bottom diagram are the direct sums 
of the obvious exact triples.\\

\begin{figure}
\begin{gather*}
\xymatrix{ & 0\ar[d] & 0\ar[d] & 0\ar[d] &\\ 
0\ar[r]& D_1 \ar[d]\ar[r]^{\id}&  D_1\ar[d]^{\fI}\ar[r]& {\bf0}\ar[d]\ar[r]& 0\\
0\ar[r]& D_2\!\circ\!D_1\oplus\id_{X_2}  \ar[r]\ar[d]& 
D_3\!\circ\!D_2\!\circ\!D_1\oplus\id_{X_2}\!\oplus\!\id_{X_3} \ar[r]\ar[d]^{\fJ}& 
 D_3\ar[r]\ar[d]^{\id}& 0\\
0\ar[r]& D_2 \ar[r]\ar[d]& D_3\!\circ\!D_2\oplus\id_{X_3} \ar[r]\ar[d]& 
D_3\ar[r]\ar[d]& 0\\
&0&0&0& \\
& 0\ar[d] & 0\ar[d] & 0\ar[d] &\\ 
0\ar[r]& D_2\!\circ\!D_1 \ar[d]\ar[r]&  
D_3\!\circ\!D_2\!\circ\!D_1\oplus\id_{X_3}\ar[d]\ar[r]& D_3\ar[d]^{\id}\ar[r]& 0\\
0\ar[r]& D_2\!\circ\!D_1\oplus\id_{X_2}  \ar[r]\ar[d]& 
D_3\!\circ\!D_2\!\circ\!D_1\oplus\id_{X_2}\!\oplus\!\id_{X_3} \ar[r]\ar[d]& 
 D_3\ar[r]\ar[d]& 0\\
0\ar[r]& \id_{X_2} \ar[r]^{\id}\ar[d]& \id_{X_2} \ar[r]\ar[d]& {\bf0}\ar[r]\ar[d]& 0\\
&0&0&0&\\
& 0\ar[d] & 0\ar[d] & 0\ar[d] &\\ 
0\ar[r]& D_3\!\circ\!D_2\!\circ\!D_1 \ar[d]\ar[r]&  
D_3\!\circ\!D_2\!\circ\!D_1\oplus\id_{X_3}\ar[d]\ar[r]& \id_{X_3}\ar[d]^{\id}\ar[r]& 0\\
0\ar[r]& D_3\!\circ\!D_2\!\circ\!D_1\oplus\id_{X_2}  \ar[r]\ar[d]& 
D_3\!\circ\!D_2\!\circ\!D_1\oplus\id_{X_2}\!\oplus\!\id_{X_3} \ar[r]\ar[d]& 
 \id_{X_3}\ar[r]\ar[d]& 0\\
0\ar[r]& \id_{X_2} \ar[r]^{\id}\ar[d]& \id_{X_2} \ar[r]\ar[d]& {\bf0}\ar[r]\ar[d]& 0\\
&0&0&0& }
\end{gather*}
\caption{Exact squares of Fredholm operators used in the derivation 
of the Compositions~I property}
\label{KM_fig1}
\end{figure}

\begin{sidewaysfigure}
\thisfloatpagestyle{empty}
\begin{gather*}
\xymatrix{\la(D_1)\!\otimes\!\la(D_2)\!\otimes\!\la(D_3) 
\ar[rr]^{\id\otimes\wt\cC_{D_2,D_3}} \ar[dr]|{\id\otimes\Psi_{2,3}} 
\ar[dd]|{\Psi_{1,2}\otimes\id} 
\ar@/_8pc/[dddd]|{\wt\cC_{D_2,D_1}\otimes\id} &&
\la(D_1)\otimes\la(D_3\!\circ\!D_2)\ar[dl]|{\id\otimes\cI_{2,3}}
\ar[dd]^{\Psi_{1,23}} \ar@/^8pc/[dddd]|{\wt\cC_{D_1,D_3\circ D_2}}\\
& \la(D_1)\!\otimes\!\la(D_3\!\circ\!D_2\!\oplus\!\id_{X_3})
\ar@/_1pc/[d]_{\Psi_{\tnM}} \ar@/^1pc/[d]^{\wt\Psi_{1,23}} &\\
\la(D_2\!\circ\!D_1\!\oplus\!\id_{X_2})\!\otimes\!\la(D_3)\ar[r]^>>>>>{\wt\Psi_{12,3}}&
\la\big(D_3\!\circ\!D_2\!\circ\!D_1\!\oplus\!\id_{X_2}\!\oplus\!\id_{X_3}\big)&
\la\big(D_3\!\circ\!D_2\!\circ\!D_1\!\oplus\!\id_{X_2}\big)\ar[l]_>>>>>>{\wt\cI_{12,3}}\\
& \la\big(D_3\!\circ\!D_2\!\circ\!D_1\!\oplus\!\id_{X_3}\big)\ar[u]^{\wt\cI_{1,23}} &\\
\la(D_2\!\circ\!D_1)\!\otimes\!\la(D_3) \ar[rr]^{\wt\cC_{D_2\circ D_1,D_3}} 
\ar[ur]|{\Psi_{12,3}}\ar[uu]|{\cI_{1,2}\otimes\id}&& 
\la(D_3\!\circ\!D_2\!\circ\!D_1)\ar[ul]|{\cI_{12,3}}\ar[uu]_{\cI_{1,23}}}
\end{gather*}
\caption{Commutative diagram used in the derivation of the Compositions~I 
property from the Exact Squares, Normalization~III, and isomorphism Naturality~II properties }
\label{ESQcomp1_fig}
\end{sidewaysfigure}

\noindent
Applying the Exact Squares and Normalization~III properties to 
the top exact square in Figure~\ref{KM_fig1}
and using identifications similar to~\eref{laext_e},
we find that the top left quadrilateral in Figure~\ref{ESQcomp1_fig},
where~$\wt\Psi_{12,3}$ and~$\Psi_{\tnM}$ are the isomorphisms~\eref{ETisom_e} 
corresponding to the middle row and the center column in this square, commutes.
The commuting bottom left quadrilateral in Figure~\ref{ESQcomp1_fig}, 
where $\wt\cI_{1,23}$ is the isomorphism~\eref{ETisom_e} 
corresponding to the center column in the second diagram of Figure~\ref{KM_fig1},
is obtained from this diagram by applying the Exact Squares and Normalization~III properties
as~well.
A similar exact square gives the commuting top right quadrilateral in Figure~\ref{ESQcomp1_fig},
where~$\wt\Psi_{1,23}$ and~$\wt\cI_{12,3}$ are the isomorphisms~\eref{ETisom_e} corresponding to 
the direct sum of the exact triple~\eref{compEX_e} for the composition 
$(D_3\!\circ\!D_2)\!\circ\!D_1$  with the exact triple
$$0\lra0\lra\id_{X_3}\lra \id_{X_3}\lra0$$
and to the middle row in the last diagram in Figure~\ref{KM_fig1}, respectively.
The bottom right quadrilateral in Figure~\ref{ESQcomp1_fig} arises from the last diagram in 
Figure~\ref{KM_fig1}.\\

\noindent
The two arrows that run between the same objects in the middle of Figure~\ref{ESQcomp1_fig}
are related by the isomorphism of exact triples of Fredholm operators,
$$\xymatrix{0\ar[r]& D_1\ar[r]\ar[d]^{\id}& 
D_3\!\circ\!D_2\!\circ\!D_1\oplus\id_{X_2}\!\oplus\!\id_{X_3}\ar[r]\ar[d]|{(\phi,\psi)}& 
D_3\!\circ\!D_2\!\oplus\!\id_{X_3} \ar[r]\ar[d]^{\id}&0\\
0\ar[r]& D_1\ar[r]^>>>>>>{\fI}& 
D_3\!\circ\!D_2\!\circ\!D_1\oplus\id_{X_2}\!\oplus\!\id_{X_3}\ar[r]^>>>>>>{\fJ}& 
D_3\!\circ\!D_2\!\oplus\!\id_{X_3} \ar[r]&0\,,}$$
where the top row is the exact triple~\eref{compEX_e} corresponding
to the composition $(D_3\!\circ\!D_2)\!\circ\!D_1$ augmented by~$\id_{X_3}$, 
$$\phi(x_1,x_2,x_3)=(x_1,x_2,x_3\!+\!D_2x_2),\qquad
\psi(x_4,x_2,x_3)=(x_4,x_2,x_3\!+\!D_2x_2).$$
Since $\wt\cI_{\phi,\psi;D_3\circ D_2\circ D_1\oplus\id_{X_2}\oplus\id_{X_3}}\!=\!\id$, 
these two arrows are in fact the same
by the isomorphism Naturality~II property.
The two half-disk and two triangular diagrams in Figure~\ref{ESQcomp1_fig}
commute by the definition of~$\wt\cC$.
Thus, the diagram~\eref{compdiag_e1}, which consists of the outermost arrows
of the diagram in Figure~\ref{ESQcomp1_fig}, commutes.
\qed\\

\noindent
The determinant for a complex of vector bundles in \cite[p31]{KM}
corresponds to reversing the two factors in~\eref{DLdfn_e} and thus interchanges
the roles of the kernels and cokernels of linear operators.
The isomorphism~\eref{detisom_e2} should then be replaced~by
\BE{detisom_e3}\begin{split}
\la^*\big(\fc(\de)\big)\!\otimes\!\la\big(\ka(\de)\big) &\lra \la^*(W)\!\otimes\!\la(V),\\ 
y^*\!\otimes\!x&\lra (-1)^{(\fd(V)-\fd(\ka(\de)))\fd(\ka(\de))}
\big(\la(\de)v\w_W\!y\big)^*\otimes x\!\w_V\!v,
\end{split}\EE
with $x,y,v$ as before.
This isomorphism differs from the isomorphism~\eref{detisom_e2} conjugated by the
isomorphisms~\eref{Risom_e} by $(-1)$ to the power of $\fd(\Im\,\de)$,
which equals $N\!-\!\fd(\fc(D))$ in the case of~\eref{KMseq_e}. 
The dependence on~$\fd(\fc(D))$ drops out when taking the overlap maps,
analogous to~$\hat\cI_{\Th_2;D}\!\circ\!\hat\cI_{\Th_1;D}^{-1}$
on page~\pageref{cITh2Th1D_e}, 
for the trivializations of the new version of the determinant line bundle,
and so the isomorphisms~\eref{detisom_e3} still give rise to a well-defined topology
on this bundle.
The two versions of the determinant line bundle are isomorphic
by the maps~\eref{Risom_e} composed with 
the multiplication by $A_{\ind D,\fd(\fc(D))}\!\equiv\!(-1)^{\fc(D)}$ 
in the fiber over $D\!\in\!\cF(X,Y)$.
Neither of the last two maps is continuous, but the composite is continuous; 
see Section~\ref{classify_subs} for a systematic discussion of such isomorphisms. 
The isomorphism~$\Psi_{\ft}$ for exact triples of Fredholm operators
described by~\eref{cUDDdfn_e}
for the topology on ${\det}_{X,Y}$ specified by~\eref{detisom_e2}
should then be conjugated by the above isomorphism between the two versions
of the determinant line bundle.
In particular, this changes the sign exponent in \eref{cUDDdfn_e1}
to~$(\ind\,D')\fd(\fc(D''))$, in addition to interchanging the kernel and 
cokernel factors.

\subsection{Other conventions}
\label{othset_subs}

\noindent
In \cite[Section~3.1]{EES}, $\la(D)$ is defined as the tensor product of 
$\la(\ka(D))$ and $\la(\fc(D)^*)$.
In \cite[Appendix~A.2]{MS}, $\la(D)$ is defined as the tensor product of 
$\la(\ka(D))$ and $\la(\ka(D^*))$.
In light of the second isomorphism in~\eref{fkfcdual_e}, 
these two conventions are essentially identical.
They implicitly identify $\la^*(\fc(D))$ with~$\la(\fc(D)^*)$.
Such an identification is determined by a pairing of  $\la(V^*)$ with $\la(V)$
for a finite-dimensional vector space~$V$.
There are two such standard pairings:
\BE{pairdfn_e}\begin{split}
\la(V^*)\otimes\la(V)\lra\R, \qquad
\al_1\!\w\!\ldots\!\w\!\al_n\otimes v_1\!\w\!\ldots\!\w\!v_n
&\lra \det\big(\al_i(v_j))_{i,j=1,\ldots,n} \quad\hbox{and}\\
&\lra (-1)^{\binom{n}{2}}\det\big(\al_i(v_j))_{i,j=1,\ldots,n}.
\end{split}\EE
Along with~\eref{detisom_e2}, these two pairings topologize the new version
of the determinant line bundle in two different ways.
The resulting line bundles are isomorphic by the multiplication
by $(-1)$ to the power of $\binom{\fd(\fc(D))}{2}$ in the fiber over $D\!\in\!\cF(X,Y)$.
Under the second pairing in~\eref{pairdfn_e}, the isomorphism~\eref{detisom_e2}
precisely corresponds  to the isomorphism~\cite[(3.1)]{EES}.
On the other hand, the analogue of~\eref{detisom_e2} used in the proof 
of \cite[Theorem~A.2.2]{MS} corresponds under the first pairing in~\eref{pairdfn_e}
to~\eref{detisom_e2} without the~sign; see \cite[Exercise~A.2.3]{MS}.
In the case of~\eref{KMseq_e}, the exponent of this sign is $(N\!-\!\fc(\fd))\fc(D)$,
which changes the overlap maps between the trivializations of the determinant line
bundle by $(-1)$ to the power of $(N'\!-\!N)\fd(\fc(D))$.
The overlap maps in the proof of \cite[Theorem~A.2.2]{MS} thus need not be continuous 
if $N\!-\!N'$ is odd and so do not topologize the determinant 
line bundles.\footnote{The 2017 revision of~\cite{MS} contains a modification
to address this issue.
There is inconsistency in the definition of~$\la(D)$ at the beginning of
the revised Appendix~A.2 in~\cite{MS} and the notation in Theorem~A.2.1ab.
Assuming the intended definition is as stated, 
Theorem~A.2.1ab defines a collection of exact triple isomorphisms~$\Psi_{\ft}$
corresponding to the collection $A_{i,c}\!\equiv\!(-1)^{ic}$ of
our Theorem~\ref{classify_thm} on page~\pageref{classify_thm} under 
the second pairing in~\eref{pairdfn_e};
this pairing is consistent with the notation in Theorem~A.2.1ab.
These isomorphisms correspond to the isomorphisms~\eref{detisom_e2}
with the sign as in~\eref{detisom_e3}.
This sign is more consistent with reversing the factors in the definition of~$\la(D)$ in~\cite{MS},
as in Theorem~A.2.1ab and in~\cite{KM}.}\\

\noindent
In \cite[Section~(11a)]{Seidel}, $\la(D)$ is defined as the tensor product of 
$\la(\fc(D)^*)$ and~$\la(\ka(D))$.
In \cite[Section~1.2]{Salamon}, $\la(D)$ is defined as the tensor product of 
$\la(\ka(D^*))$ and $\la(\ka(D))$.
In light of the second isomorphism in~\eref{fkfcdual_e},
these conventions are essentially identical.
Under the second pairing in~\eref{pairdfn_e}, the isomorphism~\eref{detisom_e3}
becomes~\cite[(11.3)]{Seidel}.
Under the same pairing, the isomorphism~\eref{detisom_e3} corresponds to
the isomorphism of \cite[Theorem~2.1]{Salamon} multiplied by $(-1)$
to the power~of 
$$\big(\fd(W)\!-\!\fd(\fc(\de))\big)(\ind\,\de)
+\fd\big(\ka(\de)\big)\fd\big(\fc(\de)\big)
\cong\fd(W)(\ind\,\de)+\fc(\de)\mod2.$$
In the case of~\eref{KMseq_e}, the sign exponent reduces to 
$N(\ind\,D)\!+\!\fd(\fc(D))$.
The dependence on $\fd(\fc(D))$ drops out when taking the overlap maps for 
the trivializations of this version of the determinant line bundle,
and so the isomorphism of \cite[Theorem~2.1]{Salamon}
gives rise to a well-defined topology on this bundle.
It is isomorphic to the determinant line bundle of \cite[Section~(11a)]{Seidel}
by the multiplication by $(-1)^{\fc(D)}$ in the fiber over $D\!\in\!\cF(X,Y)$.
The interchange of factors in~$\la(D)$ accounts for the change of the sign exponent
in the direct sum formulas, \cite[(3)]{Salamon} and \cite[(11.2)]{Seidel},
from~\eref{cUDDdfn_e1}, as explained at the end of the last paragraph 
in Section~\ref{KM_subs}.\\

\noindent
In \cite[Section~1]{Quillen} and \cite[Appendix~D.2]{Huang},
$\la(D)$ is defined to be either 
$$\la(\ka(D)^*)\otimes\la(\fc(D))\qquad\hbox{or}\qquad
\la^*(\ka(D))\otimes\la(\fc(D));$$ 
the notation is somewhat ambiguous, but looks more like the former.
The usage in~\cite{Quillen} is more consistent with the latter convention;
the usage in~\cite{Huang} is sometimes more consistent with the latter
and sometimes more consistent with the former.\footnote{For example, the last equality
in the last displayed expression in the proof of \cite[Proposition~D.2.2]{Huang}
uses the latter definition, while \cite[(D.2.9)]{Huang} uses the former.}
The latter definition of~$\la(D)$ is used in \cite[Section~(f)]{BF}. 
While $\la(\ka(D)^*)$ and $\la^*(\ka(D))$ are canonically isomorphic,
there are at least two choices of such canonical isomorphisms,
the two provided by the pairings~\eref{pairdfn_e}. 
The ``construction" of the determinant line bundle in~\cite{Quillen} consists
of mentioning that each homomorphism $\de:V\!\lra\!W$ between finite-dimensional vector spaces
gives rise to a natural isomorphism
$$\la(\ka(D)^*)\otimes\la(\fc(D))\lra \la(V^*)\otimes\la(W)
\qquad\hbox{or}\qquad
\la^*(\ka(D))\otimes\la(\fc(D))\lra \la^*(V)\otimes\la(W),$$
but no indication is given what it~is.
In the proof of \cite[Proposition~D.2.2]{Huang}, this isomorphism is described
as a composition of other isomorphisms, 
but some of them are not specified.\footnote{In addition, $(\det F)^{-1}$ should
be $\det F$ at the end of the statement of this proposition and $\det H_2$
should be $(\det H_2)^*$ in the second-to-last displayed equation in the proof;
the first change is necessary for the section~\eref{Quillensec_e} to be continuous
in the finite-dimensional case.}
The construction in \cite[Appendix~D.2]{Huang} is fundamentally based on 
\cite[Proposition~D.2.6]{Huang}, though its proof appears to be incomplete;
see Remark~\ref{Huang_rmk} for details.
However, the statement of this proposition is the basis for the construction 
of the determinant line bundle in this paper and a close cousin of this proposition,
Proposition~\ref{CompET_prp2}, is used to verify the continuity of 
the bundle map~\eref{ETisom_e} for families of exact triples of Fredholm operators.
The construction in~\cite{BF} is limited to Hilbert spaces and still omits some details.  
Neither \cite[Section~(f)]{BF}, \cite[Appendix~D]{Huang}, nor \cite{Quillen} confirms most 
of the properties of the determinant line bundle stated in Subsection~\ref{detLBprop_subs}.\\

\noindent
As noted in \cite[Section~2]{Quillen}, 
the section of ${\det}_{X,Y}$ in the definitions of \cite[Section~1]{Quillen} 
and \cite[Appendix~D.2]{Huang} given~by
\BE{Quillensec_e}
\si(D)=\begin{cases} 1^*\!\otimes\!1,&\hbox{if}~D~\hbox{is isomorphism};\\
0,&\hbox{otherwise};\end{cases}\EE
is continuous.
There is no such section if ${\det}_{X,Y}$ is defined as in~\eref{DLdfn_e},
\cite{EES}, \cite{MW}, \cite{Salamon}, or~\cite{Seidel}.
The definition of ${\det}_{X,Y}$  in \cite[Section~1]{Quillen} and \cite[Appendix~D.2]{Huang}
thus comes with a natural normalization for the topology, but it does not restrict
the topology of ${\det}_{X,Y}$ any further than the properties in Subsection~\ref{detLBprop_subs};
see Section~\ref{classify_subs}.
The alternative definitions seem more natural from the geometric viewpoint,
as typically the spaces $\ka(D)$ describe tangent spaces of some, ideally smooth, 
moduli spaces, and so it seems desirable not to dualize~them.
The alternative definitions also lead to a somewhat nicer appearance of formulas describing 
key properties of the determinant line bundle system.
For example, \cite[Proposition~D.2.2]{Huang} reverses the order of the factors
in the isomorphism of Lemma~\ref{ses_lmm}.

\subsection{Classification of determinant line bundles}
\label{classify_subs}

\noindent
For each exact triple~$\ft$ of Fredholm operators, we denote by~$\Psi_{\ft}$ 
the isomorphism~\eref{cUDDdfn_e}.
By Subsections~\ref{ET_subs}-\ref{DualProp_subs},
this collection of exact triple isomorphisms satisfies all properties
for systems of determinant line bundles listed in Subsection~\ref{detLBprop_subs}.
Let $\{\Psi_{\ft}'\}_{\ft}$ be another collection of isomorphisms for exact 
triples of Fredholm operators satisfying the isomorphism Naturality~II,
Normalization~II,III, and Exact Squares properties.
We show that this collection is as in~\eref{PsicDnew_e} for some $A_{i,c}\!\in\!\R^*$ 
and satisfies the remaining properties for systems of determinant line bundles 
listed in Subsection~\ref{detLBprop_subs}, 
if $A_{i,c}$ are chosen appropriately.
Theorem~\ref{classify_thm} describes all collections of isomorphisms
for exact triples of Fredholm operators satisfying all properties
listed  in Subsection~\ref{detLBprop_subs}.\\

\noindent
For \hbox{$i\!\in\!\Z$} and $c\!\in\!\Z^{\ge0}$ with $c\!\ge\!-i$,  
let $\Psi_{i,c}$ be the isomorphism~$\Psi_{\ft}$ in~\eref{cUDDdfn_e} for the exact~triple
\BE{kafcET_e}\begin{split}
\xymatrix{0\ar[r]& \R^{i+c}\ar[r]\ar[d]^{\bf0}& \R^{i+c}\!\oplus\!\R^c \ar[r]\ar[d]& \R^c\ar[r]\ar[d]&0\\
0\ar[r]& \R^c\ar[r]& \R^c\ar[r]& 0\ar[r]&0\,,}
\end{split}\EE
where the top right and middle arrows are the projections onto the last $c$ coordinates.
Thus,
$$\Psi_{i,c}\big(\Om_{i+c}\!\otimes\!\Om_c^*\otimes\Om_c\!\otimes\!1^*\big)
=(-1)^c\, \Om_{i+2c}\!\otimes\!1^*.$$
Let $\Psi_{i,c}'$ be the isomorphism~$\Psi_{\ft}'$ for the exact triple~\eref{kafcET_e} 
and $A_{i,c}\!\in\!\R^*$ be such that 
\BE{newPsikl_e} \Psi_{i,c}'=A_{i,c}\Psi_{i,c}\,.\EE 
In particular,
\BE{newPsikl_e2} 
\Psi_{i,c}'\big(\Om_{i+c}\!\otimes\!\Om_c^*\otimes\Om_c\!\otimes\!1^*\big)
=(-1)^c A_{i,c}\,\Om_{i+2c}\!\otimes\!1^*.\EE
By  the Normalization~II property on page~\pageref{NormalII_prop},
$A_{i,0}\!=\!1$ for all $i\!\in\!\Z^{\ge0}$.
By the Complex Exact Triples property, $A_{i,c}\!\in\!\R^+$ if $i,c\!\in\!2\Z$.\\

\noindent
Let $\ft$ be an exact triple as in~\eref{cTdiag_e} and
$$\Th'\!:\R^{N'}\lra Y' \qquad\hbox{and}\qquad  \ti\Th''\!:\R^{N''}\lra Y$$ 
be homomorphisms such that $D'\!\in\!U_{X';\Th'}$ and
$D''\!\in\!U_{X'';\fJ_{Y\circ\ti\Th''}}$.
Let $N\!=\!N'\!+\!N''$, $\fI\!:\R^{N'}\!\lra\!\R^N$
be the inclusion as $\R^{N'}\!\times\!0^{N''}$, and
$\fJ\!:\R^N\lra\R^{N''}$ be the projection onto
the last $N''$ coordinates.
We define
\begin{alignat*}{3}
\Th\!:\R^N&\lra X, &\qquad 
\Th(x',x'')&=\fI_Y\big(\Th'(x')\big)+\ti\Th''(x'')
&\quad &\forall\,(x',x'')\in\R^{N'}\!\oplus\!\R^{N''}\,,\\
\Th''\!:\R^{N''}&\lra X'', &\qquad 
\Th''(x'')&=\fJ_Y\big(\ti\Th''(x'')\big)
&\quad &\forall\,x''\in\!\R^{N''}\,.
\end{alignat*}
Thus, the first diagram in Figure~\ref{classify_fig}, where the right column is the exact triple
$$\xymatrix{0\ar[r]& \R^{N'}\ar[r]^{\fI}\ar[d]^{j_{N'}}& \R^N\ar[r]^{\fJ}\ar[d]^{j_N}& \R^{N''}\ar[r]\ar[d]^{j_{N''}}&0\\
0\ar[r]& 0\ar[r]& 0\ar[r]& 0\ar[r]&0\,,}$$
is an exact square of Fredholm operators.
By the Exact Squares and Normalization~II properties, the collection 
$\{\Psi_{\ft}'\}$ is thus determined by the isomorphisms 
$$\hat\cI_{\Th;D}'\!: \la(D)\lra \la(D_{\Th}), \qquad
\hat\cI_{\Th;D}'(\si)=
\Psi_{\ft}'\big(\si\otimes \OmN\!\otimes\!1^*\big),$$
corresponding to the exact triples~\eref{cIhatdiag_e} with $D\!\in\!U_{X;\Th}$.\\

\begin{figure}
\begin{gather*}
\xymatrix{&0\ar[d]&0\ar[d]&0\ar[d]&\\
0\ar[r]& D'\ar[r]\ar[d]& D_{\Th'}'\ar[r]\ar[d]& j_{N'}\ar[r]\ar[d]&0 \\
0\ar[r]& D\ar[r]\ar[d]& D_{\Th}\ar[r]\ar[d]& j_N\ar[r]\ar[d]&0 \\
0\ar[r]& D''\ar[r]\ar[d]& D_{\Th''}''\ar[r]\ar[d]& j_{N''}\ar[r]\ar[d]&0 \\
&0&0&0&\\
&0\ar[d]&0\ar[d]&0\ar[d]&\\
0\ar[r]& \dot{D}\ar[r]\ar[d]& D\ar[r]\ar[d]& {\bf0}_{\ka(D),\fc(D)}\ar[r]\ar[d]&0 \\
0\ar[r]& \dot{D}\ar[r]\ar[d]& D_{\Th_D}\ar[r]\ar[d]& 
\big({\bf0}_{\ka(D),\fc(D)}\big)_{\Th_D}\ar[r]\ar[d]&0 \\
0\ar[r]& {\bf0}\ar[r]\ar[d]& j_{N_D}\ar[r]\ar[d]& j_{N_D}\ar[r]\ar[d]&0 \\
&0&0&0&\\
&0\ar[d]&0\ar[d]&0\ar[d]&\\
0\ar[r]& D\ar[r]\ar[d]& D_{\Th_D}\ar[r]\ar[d]& j_{N_D}\ar[r]\ar[d]&0 \\
0\ar[r]& D\ar[r]\ar[d]& D_{\Th}\ar[r]\ar[d]& j_N\ar[r]\ar[d]&0 \\
0\ar[r]& {\bf0}\ar[r]\ar[d]& j_{N-N_D}\ar[r]\ar[d]& j_{N-N_D}\ar[r]\ar[d]&0 \\
&0&0&0&}
\end{gather*}
\caption{Exact squares of Fredholm operators specifying a determinant line bundle system}
\label{classify_fig}
\end{figure}

\noindent
Given $D\!\in\!\cF(X,Y)$, let $\dot{X}\!\subset\!X$ be a closed linear subspace such that 
the~operator
$$\dot{D}\!:\dot{X}\lra\Im\,D, \qquad \dot{D}(x)= Dx,$$
is an isomorphism and $\Th_D\!:\R^{N_D}\!\lra\!Y$ be a homomorphism 
inducing an isomorphism to~$\fc(D)$ when composed with the projection $Y\!\!\lra\!\fc(D)$.
There is an exact square of Fredholm operators as in the second diagram in Figure~\ref{classify_fig}, 
where the right column is the exact triple
$$\xymatrix{0\ar[r]& \ka(D)\ar[r]\ar[d]^{\bf0}& 
\ka(D)\oplus\R^{N_D}\ar[r]\ar[d]^{{\bf0}_{\Th_D}}& \R^{N_D}\ar[r]\ar[d]^{j_{N_D}}&0\\
0\ar[r]& \fc(D)\ar[r]& \fc(D)\ar[r]& 0\ar[r]&0\,.}$$
By the Exact Squares, isomorphism Naturality~II, and Normalization~III properties and~\eref{newPsikl_e}, 
the isomorphisms~$\hat\cI_{\Th_D;D}'$ above are determined by the isomorphisms~$\Psi_{i,c}'$
and satisfy
$$\hat\cI_{\Th_D;D}'=A_{\ind\,D,\fd(\fc(D))}\hat\cI_{\Th_D;D}\,.$$
For each homomorphism $\Th\!:\R^N\!\lra\!Y$ with $D\!\in\!U_{X;\Th}$, there is
an exact square of Fredholm operators as in the last diagram in Figure~\ref{classify_fig}.
By the Exact Squares and Normalization~II,III  properties 
and the above equation,
\BE{hatcIch_e}\hat\cI_{\Th;D}'=A_{\ind\,D,\fd(\fc(D))}\hat\cI_{\Th;D}\,.\EE

\vspace{.2in}

\noindent
The overlap maps $\hat\cI_{\Th_2;D}'\!\circ\!\hat\cI_{\Th_1;D}'^{\,-1}$
between the above isomorphisms are still
$\hat\cI_{\Th_2;D}\!\circ\!\hat\cI_{\Th_1;D}^{\,-1}$
and in particular are continuous.
By Remark~\ref{4term_rmk}, the isomorphisms~\eref{hatcIch_e} are compatible 
with the isomorphisms~\eref{detisom_e2} given~by 
\BE{hatcIde_e}\cI_{\de}'=
\frac{A_{\fd(V)-\fd(W),\fd(\fc(\de))}}{A_{\fd(V)-\fd(W),\fd(W)}}\,\cI_{\de}\!:
\la(\de)\lra\la({\bf0})\,, \EE
whenever $\de\!:V\!\lra\!W$ is a homomorphism between finite-dimensional vector spaces.
The isomorphisms 
$$\cI_D\!:\la(D)\lra\la(D), \qquad \si\lra A_{\ind\,D,\fd(\fc(D))}^{-1}\si\,,$$
give rise to continuous isomorphisms between the determinant line bundles
in the original and new topologies. 
The suitable exact triples and dualization isomorphisms are given by
\BE{PsicDnew_e}\begin{split}
\Psi_{\ft}'&=\cI_D\circ\Psi_{\ft}\circ \cI_{D'}^{-1}\!\otimes\!\cI_{D''}^{-1}
=\frac{A_{\ind\,D',\fd(\fc(D'))}A_{\ind\,D'',\fd(\fc(D''))}}{A_{\ind\,D,\fd(\fc(D))}}
\,\Psi_{\ft}\,,\\
\quad
\wt\cD_D'&= A_{-1,1}^{\ind\,D}\,\cI_{D^*}\!\circ\!\wt\cD_D\!\circ\!\cI_D^{-1}
=A_{-1,1}^{\ind\,D}\,\frac{A_{\ind\,D,\fd(\fc(D))}}{A_{-\ind\,D,\fd(\ka(D))}}\, \wt\cD_D,
\end{split}\EE
if $\ft$ is as in~\eref{cTdiag_e}.
The extra factors of~$A_{-1,1}$ in the second equation above are needed to achieve 
the Normalization~IV property on page~\pageref{NormalIV_prop}, while
preserving the Dual Exact Triples property.
In the case of the exact triple~\eref{cIhatdiag_e}, $\Psi_{\ft}'\!=\!\hat\cI_{\Th;D}$,
as the case should~be.
The new determinant line bundle system also satisfies the Normalization~IV$^{\star}$ property
if and only if $A_{-k,k}\!=\!A_{-1,1}^k$ for every $k\!\in\!\Z^+$.\\

\noindent
The above argument also implies that the Normalization~IV and Dual Exact Triples
properties on page~\pageref{NormalII_prop} determine the dualization
isomorphisms~$\wt\cD_D$ completely.
Putting everything together, we obtain a complete description of systems
of determinant line bundles.

\begin{mythm}\label{classify_thm}
The map specified by~\eref{newPsikl_e2} sends each system of determinant line bundles 
satisfying the properties in
Subsection~\ref{detLBprop_subs}, other than Normalization~IV$^{\star}$, to the functions
$$\big\{(i,c)\!:\,i\!\in\!\Z,~c\!\in\!\Z^+,~c\!\ge\!-i\big\}\lra\R^*\,,
\qquad (i,c)\lra A_{i,c}\,,\quad
A_{i,c}\in\R^+~\hbox{if}~i,c\!\in\!2\Z,$$
and is a bijection with the set of all such functions.
The determinant line bundle systems that also satisfy
the Normalization~IV$^{\star}$ property correspond to the subset of 
the above functions satisfying $A_{-k,k}\!=\!A_{-1,1}^k$ for all $k\!\in\!\Z^+$.
In particular, the compatible systems of topologies on determinant 
line bundles are in one-to-one correspondence with 
the admissible systems of isomorphisms~$\cI_{\de}$ as 
in~\eref{detisom_e}, \eref{detisom_e2}, and~\eref{hatcIde_e}
and with admissible systems of isomorphisms~$\Psi_{\ft}$ as in~\eref{ETisom_e}.
\end{mythm}

\noindent
By Theorem~\ref{classify_thm} and the preceding discussion,  
the section~$S$ of ${\det}_{X,Y}^*$ given~by
$$S_D(\si)=\begin{cases}c,&\hbox{if}~\si=c\,1\!\otimes\!1^*;\\
0,&\hbox{if}~D~\hbox{is not an isomorphism};\end{cases}$$
is continuous.
This is the analogue of the section~\eref{Quillensec_e} for
the convention~\eref{DLdfn_e}.

\begin{rmk}\label{Seidel_rmk}
According to \cite[Remark~11.1]{Seidel}, 
there are two possible sign conventions for the determinant line bundle 
and the sign convention in  \cite[Section~(11a)]{Seidel} is the same as in~\cite{KM}.
As noted in Subsection~\ref{othset_subs}, 
the setup in \cite[Section~(11a)]{Seidel} corresponds
to the setup in \cite[Chapter~I]{KM} via the second pairing in~\eref{pairdfn_e}.
The alternatives for \cite[(11.2)]{Seidel} and \cite[(11.3)]{Seidel} specified
in  \cite[Remark~11.1]{Seidel} for the ``other" sign convention 
do not satisfy the key commutativity requirement on the preceding page in~\cite{Seidel}.
In order for this requirement to be satisfied, the sign in \cite[(11.2)]{Seidel} must 
be kept precisely the same (contrary to what is explicitly stated in \cite[Remark~11.1]{Seidel});
This ``other" convention would then correspond to the setup in \cite[Chapter~I]{KM} 
via the first pairing in~\eref{pairdfn_e}.
Furthermore, by Theorem~\ref{classify_thm}, there are infinitely many possible sign conventions,
at least several of which seem quite natural. 
The isomorphisms~\eref{hatcIch_e} satisfy the two requirements above the diagram on 
page~150 in~\cite{Seidel} provided $A_{0,1}\!>\!0$.
These systems of isomorphisms can be narrowed down by replacing 
the Normalization~IV property on page~\pageref{NormalIV_prop} 
with the Normalization~IV$^{\star}$ property ($A_{-k,k}\!=\!A_{-1,1}^k$ for all $k\!\in\!\Z^+$),
by specifying the  dualization  or direct sum isomorphisms, i.e.
$$A_{-i,i+c}=A_{-1,1}^iA_{i,c} \quad\hbox{or}\quad A_{i,c}=A_{0,1}^c
\qquad\forall~i\!\in\!\Z,\,c\!\in\!\Z^{\ge0},\,c\!\ge\!-i,$$
and/or by requiring the isomorphisms~$\cI_{\de}$ to be given~by
$$\cI_{\de}\!:\la(\de)\lra\la({\bf 0}), \qquad 
1\!\otimes\!1^*\lra (\det\de)^{-1}\,v\!\otimes\!v^*\,,$$
whenever $\de\!:V\!\lra\!V$ is an isomorphism and  $v\!\in\!\la(V)\!-\!0$
($A_{0,c}\!=\!1$ for all $c\!\in\!\Z^+$).
The strongest of these additional conditions, specifying the isomorphisms 
for direct sums of Fredholm operators, seems to be the least natural requirement
to make.
\end{rmk}

\section{Linear algebra}
\label{lin_alg}

\subsection{Finite-dimensional vector spaces}
\label{FinDim_subs}

\noindent
In this subsection, we make a number of purely algebraic observations 
concerning finite-dimensional vector spaces
that lie behind the determinant line construction.

\begin{lmm}[{\cite[Proposition 1(i)]{KM}}]\label{ses_lmm}
Every short exact sequence
\BE{ses_e}0\lra V'\stackrel{\fI}{\lra} V\stackrel{\fJ}{\lra} V''\lra 0\EE
induces a natural isomorphism
$$\w_V\!:\la(V')\otimes\la(V'')\lra \la(V) .$$
\end{lmm}

\begin{proof}
If $v_1',\ldots,v_k'$ is a basis for $V'$ and $v_1,\ldots,v_{\ell}\!\in\!V$ are such that 
$\fJ(v_1),\ldots,\fJ(v_{\ell})$ is a basis for~$V''$,
$v_1'\!\w\!\ldots\!\w\!v_k'$ and $\fJ(v_1)\!\w\!\ldots\!\w\!\fJ(v_{\ell})$
span $\la(V')$ and $\la(V'')$, respectively.
By the exactness of~\eref{ses_e}, $\fI(v_1'),\ldots,\fI(v_k'),v_1,\ldots,v_{\ell}$ is 
basis for $V$ and so the map
\BE{sesmap_e} \w_V\!: v_1'\!\w\!\ldots\!\w\!v_k'
\otimes \fJ(v_1)\!\w\!\ldots\!\w\!\fJ(v_{\ell})
\lra \fI(v_1')\!\w\!\ldots\!\w\!\fI(v_k')\w
v_1\!\w\!\ldots\!\w\!v_{\ell}\EE
induces an isomorphism $\la(V')\!\otimes\!\la(V'')\lra \la(V)$.
By the exactness of~\eref{ses_e}, each $v_i\!\in\!V$ is determined by $\fJ(v_i)\!\in\!V''$
up to a linear combination of $\fI(v_1'),\ldots,\fI(v_k')$, and so the right-hand side of 
\eref{sesmap_e} is determined by $v_1',\ldots,v_k'\!\in\!V'$ and 
$\fJ(v_1),\ldots,\fJ(v_{\ell})\!\in\!V''$.
Changing the collections $v_1',\ldots,v_k'\!\in\!V'$ and $v_1,\ldots,v_{\ell}\!\in\!V$
by a $k\!\times\!k$-matrix $A'$ and an $\ell\!\times\!\ell$-matrix $A$, respectively,
changes the wedge products of the first $k$ vectors and the last $\ell$ vectors by
$\det A'$ and $\det A$, respectively, on both sides of~\eref{sesmap_e}.
Thus, the isomorphism induced by~\eref{sesmap_e} is independent of the choices of 
collections  $v_1',\ldots,v_k'\!\in\!V'$ and $v_1,\ldots,v_{\ell}\!\in\!V$ as above.
It clearly commutes with isomorphisms of short exact sequences.
\end{proof}

\noindent
The next lemma follows immediately from the definitions of $\cP$ in~\eref{cPisom_e} 
and of~$\w_V$ above.

\begin{lmm}\label{sesdual_lmm}
For every finite-dimensional vector space~$V$,
\BE{sesdual_e1}\cP(v^*)= (\cP v)^* \qquad\forall\,v\in\la(V)\!-\!0.\EE
For every isomorphism $\de\!:V\!\lra\!W$ between finite-dimensional vector spaces,
\BE{sesdual_e2}\la(\de^*)\cP\big((\la(\de)v)^*\big)=\cP(v^*) \qquad\forall\,v\in\la(V)\!-\!0.\EE
For every short exact sequence~\eref{ses_e},
\BE{sesdual_e3}\cP\big((\la(\fI)v'\w_V\!v'')^*\big)
=\cP(v''^*)\w_{V^*}\la(\fI^*)^{-1}\cP(v'^*) 
\quad\forall\,v'\!\in\!\la(V')\!-\!0,\,v''\!\in\!\la\big(V/\fI(V')\big)\!-\!0.\EE
\end{lmm}

\noindent
From \eref{sesmap_e}, we immediately find that the isomorphisms $\w_V$ 
of Lemma~\ref{ses_lmm} satisfy graded commutativity, 
as described by the next lemma.
Corollary~\ref{ses_crl} below is a special case of this lemma (either $V_{\tnT\tnR}\!=\!0$
or $V_{\tnB\tnL}\!=\!0$).

\begin{lmm}[{\cite[Proposition~1(ii)]{KM}}]\label{ses_lmm2}
For every commutative diagram 
$$\xymatrix{&0\ar[d]&0\ar[d]&0\ar[d]&\\
0\ar[r]& V_{\tnT\tnL} \ar[r]\ar[d]& V_{\tnT\tnM} \ar[r]\ar[d]& V_{\tnT\tnR}\ar[r]\ar[d]& 0\\
0\ar[r]& V_{\tnC\tnL} \ar[r]\ar[d]& V_{\tnC\tnM} \ar[r]\ar[d]& V_{\tnC\tnR}\ar[r]\ar[d]& 0\\
0\ar[r]& V_{\tnB\tnL} \ar[r]\ar[d]& V_{\tnB\tnM} \ar[r]\ar[d]& V_{\tnB\tnR}\ar[r]\ar[d]& 0\\
&0&0&0&}$$
of exact rows and columns, the diagram 
$$\xymatrix{ \la(V_{\tnT\tnL})\otimes\la(V_{\tnB\tnL}) 
  \otimes\la(V_{\tnT\tnR})\otimes\la(V_{\tnB\tnR})
\ar[d]|{~~\w_{V_{\tnC\tnL}}\otimes\w_{V_{\tnC\tnR}}}
\ar[rrrr]^>>>>>>>>>>>>>>>>>>>>>{\w_{V_{\tnT\tnM}}\otimes\w_{V_{\tnB\tnM}}\,\circ\,
\id \otimes R\otimes \id}
&&&& \la(V_{\tnT\tnM})\otimes\la(V_{\tnB\tnM})\ar[d]|{\w_{V_{\tnC\tnM}}}\\
\la(V_{\tnC\tnL})\otimes\la(V_{\tnC\tnR})\ar[rrrr]^{\w_{V_{\tnC\tnM}}}&&&& \la(V_{\tnC\tnM})}$$
commutes.
\end{lmm}

\begin{crl}\label{ses_crl}
For every commutative diagram 
$$\xymatrix{&& 0\ar@/^1pc/[r] & V_{\tnL\tnR}\ar@/^1pc/[r] \ar[dr]& 0\\
0\ar@{.>}[dr]&0\ar@{-->}@/^1pc/[r]& V_{\tnL\tnC} \ar[ur] \ar@{-->}[dr] && 
V_{\tnC\tnR}\ar[dr] \ar@{.>}@/^1pc/[r]&0&0\\
0\ar@/_1pc/[r] &V_{\tnL\tnL}\ar[ur] \ar@{.>}@/_1pc/[rr] && V_{\tnC\tnC} \ar@{-->}@/_1pc/[rr] 
\ar@{.>}[ru] && V_{\tnR\tnR} \ar@/_1pc/[r] \ar@{-->}[ru]& 0}$$
of 4 exact short sequences, the diagram
$$\xymatrix{\la(V_{\tnL\tnL})\otimes\la(V_{\tnL\tnR})\otimes\la(V_{\tnR\tnR})
\ar[rr]^>>>>>>>>>>>{\w_{V_{\tnL\tnC}}\otimes\id}
\ar[d]|{\id\otimes\w_{V_{\tnC\tnR}}} && \la(V_{\tnL\tnC})\otimes\la(V_{\tnR\tnR}) 
\ar[d]|{~\w_{V_{\tnC\tnC}}}\\
\la(V_{\tnL\tnL})\otimes\la(V_{\tnC\tnR}) \ar[rr]^{\w_{V_{\tnC\tnC}}}&& \la(V_{\tnC\tnC})}$$
commutes.
\end{crl}

\subsection{Exact triples of Fredholm operators}
\label{ET_subs}

\noindent
We begin this subsection by extending the isomorphism of Lemma~\ref{ses_lmm}
to exact triples of Fredholm operators.
It is immediate from the explicit formula~\eref{cUDDdfn_e} for the new isomorphism 
that it satisfies the Naturality~II, Normalization~II,III, Direct Sums~I,II,
and Complex Exact Triples properties in Subsection~\ref{detLBprop_subs}.
We verify that it also satisfies the Dual Exact Triples property 
with $\wt\cD_D$ given by~\eref{cDdfn_e} and the Compositions~I,II properties.\\

\noindent
We will use the natural pairing of a one-dimensional vector space $L$
with its dual given~by
$$L^*\otimes L\lra \R, \qquad \al\otimes v\lra\al(v).$$
If $V$ is a finite-dimensional vector space and $v\!\in\!\la(V)$, we denote by
$$\lr{v}\equiv \dim V+2\Z\in\Z_2$$
the degree of~$v$ as an element of the $\Z_2$-line $\la(V)$.

\begin{prp}[{\cite[Proposition D.2.3]{Huang}}]\label{QuillenIsom_prp}
Every exact triple~$\ft$ of Fredholm operators as in~\eref{cTdiag_e}
induces a natural isomorphism
$$\Psi_{\ft}\!:\la(D')\otimes \la(D'')\lra\la(D).$$
\end{prp}

\begin{proof} 
By the Snake Lemma, \eref{cTdiag_e} induces an exact sequence 
\BE{ETles_e} 0\lra \ka(D')\stackrel{\fI_X}{\lra} \ka(D) \stackrel{\fJ_X}{\lra} \ka(D'')
\stackrel{\de}{\lra} \fc(D')  \stackrel{\fI_Y}{\lra} \fc(D)
\stackrel{\fJ_Y}{\lra} \fc(D'')\lra0.\EE
By Lemma~\ref{ses_lmm}, there are then natural isomorphisms
\BE{sesisom_e}\begin{aligned}
\la(\ka(D))&\approx \la(\ka(D'))\otimes\la(\Im\,\fJ_X), &\qquad
\la(\ka(D''))&\approx \la(\Im\,\fJ_X)\otimes\la(\Im\,\de), \\
\la(\fc(D'))&\approx \la(\Im\,\de)\otimes\la(\Im\,\fI_Y), &\qquad
\la(\fc(D))&\approx \la(\Im\,\fI_Y)\otimes\la(\fc(D'')).
\end{aligned}\EE
Putting these isomorphisms together and using the natural evaluation isomorphisms, we obtain
\BE{SnakeLmm_e}\begin{split}
\la(D')\otimes \la(D'') &\equiv 
\la(\ka(D'))\otimes\la^*(\fc(D')) \otimes
\la(\ka(D''))\otimes  \la^*(\fc(D''))\\
&\approx \la(\ka(D))\otimes\la^*(\Im\,\fJ_X) 
\otimes\la^*(\Im\,\fI_Y)\otimes\la^*(\Im\,\de)\\
&\quad \otimes
\la(\Im\,\fJ_X)\otimes\la(\Im\,\de)
\otimes \la^*(\fc(D))\otimes \la(\Im\,\fI_Y)
\approx \la(\ka(D))\otimes \la^*(\fc(D)).
\end{split}\EE
This establishes the claim. 
\end{proof}

\noindent
For computational purposes, it is essential to specify the isomorphism of 
Proposition~\ref{QuillenIsom_prp} explicitly. 
With the notation as in~\eref{cTdiag_e} and \eref{ETles_e},
let
\BE{signdfn_e}\ep_{\ft}=(\ind\,D'')\fd(\fc(D'))+\fd(\fc(D))\fd(\Im\,\de).\EE
For $\ft$ corresponding to~\eref{cTdiag_e}, we define 
\BE{cUDDdfn_e}\begin{split}
&\Psi_{\ft}\big( x\!\otimes\!(\la(\de)v\w_{\fc(D')}\!w)^*\otimes
(\la(\fJ_X)u\w_{\ka(D'')}\!v)\!\otimes\!(\la(\fJ_Y)y)^* \big)\\
&\hspace{2in}=
(-1)^{\ep_{\ft}} \big(\la(\fI_X)x\!\w_{\ka(D)}\!u\big)\otimes 
\big(\la(\fI_Y)w\w_{\fc(D)}\!y\big)^*,
\end{split}\EE
whenever
\begin{gather*}
x\!\in\!\la(\ka(D')),~
u\!\in\!\la\bigg(\frac{\ka(D)}{\fI_X(\ka(D'))}\bigg),~
v\!\in\!\la\bigg(\frac{\ka(D'')}{\fJ_X(\ka(D))}\bigg),~
w\!\in\!\la\bigg(\frac{\fc(D')}{\de(\ka(D''))}\bigg),
~y\!\in\!\la\bigg(\frac{\fc(D)}{\fI_Y(\fc(D'))}\bigg),\\
x,u,v,w,y\neq0.
\end{gather*}
In particular, $\Psi_{\ft}$ satisfies the Naturality~II and Normalization~II,III  properties.
By~\eref{Cdualpair_e2}, it also satisfies the Complex Exact Triples property.

\begin{rmk}\label{4term_rmk}
If $\de\!:V\!\lra\!W$ is a homomorphism between finite-dimensional vector spaces,
the isomorphism~\eref{cUDDdfn_e} applied to the exact sequence
\BE{4term_les}
0\lra0 \lra \ka(\de)\lra V\stackrel{\de}{\lra} W\stackrel{q}{\lra} \fc(\de)\lra0\lra0\EE
induces the isomorphism
$$\Psi_{\de}\!: \la^*(W)\otimes\la(V)\lra \la(\de), \qquad
\Psi_{\de}(\be\otimes x)=
\Psi_{\ft_{\de}}(1\!\otimes\!\be\otimes x\!\otimes\!1^*),$$
where $\Psi_{\ft_{\de}}$ is the isomorphism~\eref{cUDDdfn_e} for the exact 
sequence~\eref{4term_les}.
Explicitly,
\begin{gather*}
\Psi_{\de}\big((\la(\de)v\!\w_W\!w)^*\otimes(u\!\w_V\!v)\big)
=(-1)^{\fd(V)\fd(W)+(\fd(W)-\fd(\fc(\de)))\fd(\fc(\de))}u\otimes w^*,\\
\hbox{if}\qquad
u\in\la(\ka(\de))-0,\quad v\in\la(V/\ka(\de))-0, \quad w\in\la(\fc(\de))-0.
\end{gather*}
Thus,
\begin{equation*}\begin{split}
\Psi_{\bf0}\circ\Psi_{\de}^{-1}\!:\la(\de)&\lra\la({\bf0})\equiv\la(V)\otimes\la^*(W),\\
u\otimes w^*&\lra  (-1)^{(\fd(W)-\fd(\fc(\de)))\fd(\fc(\de))}
(u\!\w_V\!v)\otimes (\la(\de)v\!\w_W\!w)^*,
\end{split}\end{equation*}
is precisely the isomorphism~\eref{detisom_e2}.
\end{rmk}

\noindent 
For any $D'\!\in\!\cF(X',Y')$ and $D''\!\in\!\cF(X'',Y'')$, let
$$\wt\oplus_{D',D''}\!:\,\la(D')\otimes\la(D'')\lra\la(D'\!\oplus\!D'')$$
be the isomorphism $\Psi_{\ft}$
in~\eref{cUDDdfn_e} corresponding to the diagram~\eref{sumEX_e}.
Thus,
\BE{cUDDdfn_e1}\begin{split}
&\wt\oplus_{D',D''}\big( (x_1'\!\w\!\ldots\!\w\!x_{k'}')\!\otimes\!  
(y_1'\!\w\!\ldots\!\w\!y_{\ell'}')^*\otimes 
(x_1''\!\w\!\ldots\!\w\!x_{k''}'')\!\otimes\!(y_1''\!\w\!\ldots\!\w\!y_{\ell''}'')^*\big)\\
&\hspace{.2in}=(-1)^{(\ind D'')\fd(\fc(D'))} 
\big((x_1',0)\!\w\!\ldots\!\w\!(x_{k'}',0)\w
(0,x_1'')\!\w\!\ldots\!\w\!(0,x_{k''}'')\big)\\
&\hspace{1.8in}\otimes
\big((y_1',0)\!\w\!\ldots\!\w\!(y_{\ell'}',0)\w
(0,y_1'')\!\w\!\ldots\!\w\!(0,y_{\ell''}'')\big)^*,
\end{split}\EE 
whenever
\begin{alignat*}{2}
x_1'\!\w\!\ldots\!\w\!x_{k'}'&\in \la(\ka(D'))-0,&\qquad 
y_1'\!\w\!\ldots\!\w\!x_{\ell'}'&\in \la(\fc(D'))-0,\\
x_1''\!\w\!\ldots\!\w\!x_{k''}''&\in \la(\ka(D''))-0,&\qquad 
y_1''\!\w\!\ldots\!\w\!y_{\ell''}''&\in \la(\fc(D''))-0.
\end{alignat*}
The two Direct Sums properties on page~\pageref{DirSumI_prop} follow 
immediately from~\eref{cUDDdfn_e1}.\\

\noindent
The isomorphism
\BE{cDdfn_e}\wt\cD_D\!:\la(D)\lra\la(D^*), \qquad
x\!\otimes\!\al\lra (-1)^{(\ind D)\fd(\fc(D))}
 \la(\cD_D)(\cP\al) \otimes \cP\big(\la(\cD_D)x\big),\EE
satisfies the Normalization~IV$^{\star}$ property on page~\pageref{NormalIVst_prop}.
By~\eref{Cdualpair_e2}, it satisfies the Dual Complex Orientations properties as~well.
The next proposition shows that it also satisfies the Dual Exact Triples property.  
The extra factor of $(-1)^{\fd(\fc(D))}$ in~\eref{cDdfn_e}
arises for the same reason as in the paragraph containing~\eref{detisom_e3}.
Due to this extra factor, the compositions of $\wt\cD_D$ with $\wt\cD_{D^*}$
are the multiplication by $(-1)^{\ind D}$, not necessarily the identity,
whenever the Banach spaces $X$ and $Y$ are reflexive.

\begin{prp}[Dual Exact Triples]\label{CompETdual_prp}
For every exact triple~\eref{cTdiag_e} of Fredholm operators,
the diagram~\eref{compdiagdual_e} determined by 
the isomorphisms~\eref{cUDDdfn_e} and~\eref{cDdfn_e} commutes.
\end{prp}

\begin{proof} 
With notation as in~\eref{cTdiag_e} and~\eref{compdiagdual_e}, we define
\begin{equation*}\begin{split}
\ep_{\tnL}&=(\ind\,D')(\ind\,D'')+(\ind\,D')\fd(\fc(D'))+(\ind\,D'')\fd(\fc(D''))
+\ep_{\ft^*}, \quad
\ep_{\tnR}=
\ep_{\ft}+(\ind\,D)\fd(\fc(D)).
\end{split}\end{equation*}
The isomorphisms~\eref{fkfcdual_e} intertwine the analogue of 
the exact sequence~\eref{ETles_e} for~$\ft^*$ and 
the dual of~\eref{ETles_e}:
\BE{dualles_e}\begin{split}
\xymatrix{0\ar[r]& \ka(D''^*)\ar[r]^{\fJ_Y^*}& \ka(D^*)\ar[r]^{\fI_Y^*}& 
\ka(D'^*)\ar[r]^{\de^*}& \fc(D''^*)  \ar[r]^{\fJ_X^*}\ar[d]^{\cD_{D''}^*}&
\fc(D^*)\ar[r]^{\fI_X^*}\ar[d]^{\cD_D^*}& \fc(D'^*)\ar[r]\ar[d]^{\cD_{D'}^*}& 0\\
0\ar[r]& \fc(D'')^*\ar[r]^{\fJ_Y^*}\ar[u]_{\cD_{D''}}& 
\fc(D)^*\ar[r]^{\fI_Y^*}\ar[u]_{\cD_D}& \fc(D')^*\ar[r]^{\de^*}\ar[u]_{\cD_{D'}}& 
\ka(D'')^* \ar[r]^{\fJ_X^*}& \ka(D)^*\ar[r]^{\fI_X^*}& \ka(D')^*\ar[r]& 0 \,.}
\end{split}\EE
In particular,
$$\fd(\Im\,\de^*)=\fd(\Im\,\de)=\fd(\ka(D'))+\fd(\ka(D''))+\fd(\ka(D))
=\fd(\fc(D'))+\fd(\fc(D''))+\fd(\fc(D))$$
and so $2|(\ep_{\tnL}\!-\!\ep_{\tnR})$.\\

\noindent
Let $x,u,v,w,y$ be as in~\eref{cUDDdfn_e}.
By \eref{dualles_e}, we can compute $\Psi_{\ft^*}$ using  
\begin{equation*}\begin{aligned}
\ch{x}&=\la(\cD_{D''})\cP\big((\la(\fJ_Y)y)^*\big)\in\la\big(\ka(D''^*)\big), &
\ch{u}&=\la(\cD_D)\cP\big((\la(\fI_Y)w)^*\big)\in\la\bigg(\frac{\ka(D^*)}{\fJ_Y^*(\ka(D''^*))}\bigg),\\
\ch{v}&=\la(\cD_{D'})\cP\big((\la(\de)v)^*\big)
\in\la\bigg(\frac{\ka(D'^*)}{\io_Y^*(\ka(D^*))}\bigg), &
\ch{w}&=\la(\cD_{D''}^*)^{-1}\cP\big((\la(\fJ_X)u)^*\big)
\in\la\bigg(\frac{\fc(D''^*)}{\de^*(\ka(D'^*))}\bigg),\\
\ch{y}&=\la(\cD_D^*)^{-1}\cP\big((\la(\fI_X)x)^*\big)
\in\la\bigg(\frac{\fc(D^*)}{\fJ_X^*(\fc(D''^*))}\bigg).
\end{aligned}\end{equation*}
By \eref{sesdual_e1}, \eref{sesdual_e2}, and the commutativity 
of the diagram~\eref{dualles_e}, 
\begin{gather}
\label{xyrel_e}
\cP\big(\la(\cD_{D'})x\big)=\big(\la(\fI_X^*)\ch{y}\big)^*,\\
\begin{aligned}
\la(\cD_{D'})\cP(w^*)&=\la(\fI_Y^*)\ch{u}, &  
\la(\cD_{D'})\la(\de^*)^{-1}\cP(v^*)&=\ch{v}, \notag\\
\la(\cD_D)\cP(y^*)&=\la(\fJ_Y^*)\ch{x},  &
\la(\cD_D)\la(\fI_Y^*)^{-1}\cP(w^*)&=\ch{u}, \notag
\end{aligned}\\
\begin{aligned}
\la(\cD_{D''})\la(\fJ_X)u&=\cP(\ch{w}^*),  &\quad
\la(\cD_{D''})v&=\la(\de)^{-1}\cP(\ch{v}^*) \notag\\
\la(\cD_D)\la(\fI_X)x&=\cP(\ch{y}^*),&\quad
\la(\cD_D)u&=\la(\fJ_X)^{-1}\cP(\ch{w}^*). \notag
\end{aligned}
\end{gather}
Combining each pair of identities on the last four lines above with
\eref{sesdual_e3}, we obtain
\begin{alignat}{1}
\label{CompETdual_e5a}
\la(\cD_{D'})\cP\big((\la(\de)v\w_{\fc(D')}\!w)^*\big)
&=\la(\fI_Y^*)\ch{u}\w_{\ka(D'^*)}\!\ch{v},\\
\label{CompETdual_e5b}
\la(\cD_D)\cP\big((\la(\fI_Y)w\w_{\fc(D)}\!y)^*\big)
&=\la(\fJ_Y^*)\ch{x}\w_{\ka(D^*)}\!\ch{u},\\
\label{CompETdual_e5c}
\cP\big(\la(\cD_{D''})(\la(\fJ_X)u\w_{\ka(D'')}\!v)\big)&=
\big(\la(\de^*)\ch{v}\w_{\fc(D''^*)}\!\ch{w}\big)^*\,,\\
\label{CompETdual_e5d}
\cP\big(\la(\cD_D)(\la(\fI_X)x\w_{\ka(D)}\!u)\big)&=
\big(\la(\fJ_X^*)\ch{w}\w_{\fc(D^*)}\!\ch{y}\big)^*,
\end{alignat}
respectively. 
By \eref{cDdfn_e}, \eref{xyrel_e}, \eref{CompETdual_e5a}, \eref{CompETdual_e5c},
and \eref{cUDDdfn_e},  the image~of 
\BE{CompETdual_e9} x\!\otimes\!(\la(\de)v\w_{\fc(D')}\!w)^*\otimes
(\la(\fJ_X)u\w_{\ka(D'')}\!v)\!\otimes\!(\la(\fJ_Y)y)^*\in\la(D')\otimes\la(D'')\EE
under $\Psi_{\ft^*}\!\circ\!\wt\cD_{D''}\!\otimes\!\wt\cD_{D'}\!\circ\!R$ is
$$(-1)^{\ep_{\tnL}} \big(\la(\fJ_Y^*)\ch{x}\w_{\ka(D^*)}\!\ch{u}\big) \otimes
\big(\la(\fJ_X^*)\ch{w}\w_{\fc(D^*)}\!\ch{y}\big)^*
\in \la(D^*).$$
By  \eref{cUDDdfn_e}, \eref{cDdfn_e}, \eref{CompETdual_e5b}, 
and \eref{CompETdual_e5d}, the image~of the element~\eref{CompETdual_e9}
under $\wt\cD_D\!\circ\!\Psi_{\ft}$ is
$$(-1)^{\ep_{\tnR}} \big(\la(\fJ_Y^*)\ch{x}\w_{\ka(D^*)}\!\ch{u}\big) \otimes
\big(\la(\fJ_X^*)\ch{w}\w_{\fc(D^*)}\!\ch{y}\big)^*
\in \la(D^*).$$
Since $2|(\ep_{\tnL}\!-\!\ep_{\tnR})$, this establishes the claim.
\end{proof}

\noindent
For any $D_1\!\in\!\cF(X_1,X_2)$ and $D_2\!\in\!\cF(X_2,X_3)$, let
$$\wt\cC_{D_1,D_2}\!:\,\la(D_1)\otimes\la(D_2)\lra\la(D_2\!\circ\!D_1)$$
be the isomorphism $\Psi_{\ft}$
in~\eref{cUDDdfn_e} corresponding to the diagram~\eref{compEX_e}.
The exact sequence~\eref{ETles_e} in this case specializes~to
\begin{gather}\label{ETles_e2} 
0\lra \ka(D_1)\lra \ka(D_2\!\circ\!D_1) \stackrel{D_1}{\lra} \ka(D_2)
\stackrel{\de}{\lra} \fc(D_1)  \stackrel{D_2}{\lra} \fc(D_2\!\circ\!D_1)\lra \fc(D_2)\lra0,\\
\de(x_2)=-x_2+\Im\,D_1.\notag
\end{gather}
Let
$$\ep_{D_1,D_2}=(\ind\,D_2)\fd(\fc(D_1))+
\big(\fd(\fc(D_1))+\fd(\fc(D_2))\big)\fd(\Im\,\de).$$
Then,
\BE{cUDDdfn_e2}\begin{split}
&\wt\cC_{D_1,D_2}\big( x_1\!\otimes\!(v\w_{\fc(D_1)}\!w)^*
\otimes
(\la(D_1)u\w_{\ka(D_2)}\!v)\!\otimes\!y_2^* \big)\\
&\hspace{1.5in}=
(-1)^{\ep_{D_1,D_2}} (x_1\!\w_{\ka(D_2\circ D_1)}\!u)\otimes 
(\la(D_2)w\w_{\fc(D_2\circ D_1)}\!y_2)^*,
\end{split}\EE
whenever
\begin{gather*}
x_1\!\in\!\la(\ka(D_1))-0,\quad y_2\!\in\!\la(\fc(D_2))-0,\quad
u\!\in\!\la\bigg(\frac{\ka(D_2\!\circ\!D_1)}{\ka(D_1)}\bigg)-0,\\
v\!\in\!\la\bigg(\frac{\ka(D_2)}{\ka(D_2)\cap(\Im\, D_1)}\bigg)-0,\quad
w\!\in\!\la\bigg(\frac{X_2}{\ka(D_2)+(\Im\,D_1)}\bigg)-0.
\end{gather*}

\begin{prp}[Compositions~I, {\cite[Proposition D.2.6]{Huang}}]\label{CompET_prp}
For any triple of Fredholm operators $D_1\!:X_1\!\lra\!X_2$, $D_2\!:X_2\!\lra\!X_3$, 
and $D_3\!:X_3\!\lra\!X_4$, the diagram~\eref{compdiag_e1}
induced by the isomorphisms~\eref{cUDDdfn_e} commutes.
\end{prp}

\begin{proof} We denote 
by $D''D'$ the composition $D''\!\circ\!D'$ of two maps~$D'$ and~$D''$ and define
$$\ep_{\tnL}= \ep_{D_1,D_2}+\ep_{D_2D_1,D_3}\,,\qquad
\ep_{\tnR}= \ep_{D_2,D_3}+\ep_{D_1,D_3D_2}\,.$$
For $i\!=\!1,2,3$, let
$$x_i\!\in\!\la(\ka(D_i))-0 \qquad\hbox{and}\qquad 
y_i\!\in\!\la(\fc(D_i))-0.$$
For $(i,j)\in\{(1,2),(2,3),(1,23),(12,3)\}$, let
$$u_{i,j}\in\la\bigg(\frac{\ka(D_jD_i)}{\ka(D_i)}\bigg)-0,\quad
v_{i,j}\in\la\bigg(\frac{\ka(D_j)}{\ka(D_j)\cap\Im(D_i)}\bigg)-0,\quad
w_{i,j}\in\la\bigg(\frac{X_j}{\ka(D_j)+\Im(D_i)}\bigg)-0,$$
where $D_{12}\!=\!D_2D_1$, $D_{23}\!=\!D_3D_2$, and $X_{23}\!=\!X_2$;
see Figure~\ref{CompDiags_fig}. 
Below we choose these elements in a compatible~way.\\

\noindent
Applying Lemma~\ref{ses_lmm} to the exact sequence~\eref{ETles_e2} with $D_1$ and $D_2$ 
replaced by $D_i$ and $D_j$ with $(i,j)$ as above, we obtain
\begin{alignat*}{2}
\fd(\ka(D_3))&=\lr{u_{12,3}}+\lr{v_{12,3}}, &\qquad
\fd(\fc(D_1))&=\lr{v_{1,23}}+\lr{w_{1,23}}\,,\\
\fd(\fc(D_jD_i))&=\fd(\fc(D_i))+\fd(\fc(D_j))-\lr{v_{i,j}},&\qquad
\ind\,D_jD_i&=\ind\,D_i+\ind\,D_j\,,
\end{alignat*}
where $(i,j)\!=\!(1,2),(2,3)$.
From this, we find that 
\BE{sign_e}\begin{split}
\ep_{\tnL}&=A+C(\lr{v_{1,2}}\!+\!\lr{v_{12,3}})
+\lr{u_{12,3}}\lr{v_{1,2}} \mod2,\\
\ep_{\tnR}&=A+C(\lr{v_{2,3}}\!+\!\lr{v_{1,23}})
+\lr{v_{2,3}}\lr{w_{1,23}}\mod2,
\end{split}\EE
where
$$A=(\ind\,D_3D_2)\cdot\fd(\fc(D_1))+(\ind\,D_3)\cdot\fd(\fc(D_2)),
\quad C=\fd(\fc(D_1))+\fd(\fc(D_2))+\fd(\fc(D_3)).$$\\

\noindent
In light of the top row in the first diagram in Figure~\ref{CompDiags_fig}, 
the bottom row in the second diagram,
and Lemma~\ref{ses_lmm}, we can take 
\BE{uwrel_e} u_{1,23}=u_{1,2}\w_{\frac{\ka(D_3D_2D_1)}{\ka(D_1)}}u_{12,3}
\qquad\hbox{and}\qquad
w_{12,3}=\la(D_2)w_{1,23}\w_{\frac{X_3}{\ka(D_3)+\Im(D_2D_1)}}w_{2,3}\,.\EE
Along with Corollary~\ref{ses_crl}, these equalities insure that 
\BE{value_e}\begin{split}
&\big((x_1\!\w_{\ka(D_2D_1)}u_{1,2})\w_{\ka(D_3D_2D_1)}u_{12,3}\!\big)\otimes
\big(\la(D_3)w_{12,3}\w_{\fc(D_3D_2D_1)}\!y_3\big)^*\\
&\hspace{.3in}=\big(x_1\w_{\ka(D_3D_2D_1)}\!u_{1,23}\big)\otimes
\big(\la(D_3D_2)w_{1,23}\w_{\fc(D_3D_2D_1)}(\la(D_3)w_{2,3}\w_{\fc(D_3D_2)}y_3\!)\big)^*
\end{split}\EE
in $\la(D_3D_2D_1)$.
In light of the right column and bottom row in the first diagram in Figure~\ref{CompDiags_fig},
the top row and left column in the second diagram in Figure~\ref{CompDiags_fig}, 
and Lemma~\ref{ses_lmm}, we can take 
\BE{munew_e}\begin{aligned}
u_{2,3}&=\la(D_1)u_{12,3}\w_{\frac{\ka(D_3D_2)}{\ka(D_2)}}\mu,&\qquad
v_{1,23}&=v_{1,2}\w_{\frac{\ka(D_3D_2)}{\ka(D_3D_2)\cap\Im(D_1)}}\mu,\\
v_{12,3}&=\la(D_2)\mu\w_{\frac{\ka(D_3)}{\ka(D_3)\cap\Im(D_2D_1)}}v_{2,3}, &\qquad
w_{1,2}&=\mu\w_{\frac{X_2}{\ka(D_2)+\Im(D_1)}}w_{1,23}
\end{aligned}\EE
for some
$$\mu\in\la\bigg(\frac{\ka(D_3D_2)}{\ka(D_2)+\ka(D_3D_2)\cap\Im(D_1)}\bigg)-0.$$\\

\begin{figure}
\begin{gather*}
\xymatrix{ & 0\ar[d] & 0\ar[d] & 0\ar[d] &\\ 
0\ar[r]& \frac{\ka(D_2D_1)}{\ka(D_1)}\ar[d]^{D_1}\ar[r]& 
\frac{\ka(D_3D_2D_1)}{\ka(D_1)}\ar[d]^{D_1}\ar[r]& 
\frac{\ka(D_3D_2D_1)}{\ka(D_2D_1)}\ar[d]^{D_1}\ar[r]& 0\\
0\ar[r]& \ka(D_2) \ar[r]\ar[d]& \ka(D_3D_2) \ar[r]\ar[d]& 
\frac{\ka(D_3D_2)}{\ka(D_2)}\ar[r]\ar[d]& 0\\
0\ar[r]& \frac{\ka(D_2)}{\ka(D_2)\cap\Im(D_1)}\ar[r]\ar[d]&
\frac{\ka(D_3D_2)}{\ka(D_3D_2)\cap\Im(D_1)}\ar[r]\ar[d]&
\frac{\ka(D_3D_2)}{\ka(D_2)+\ka(D_3D_2)\cap\Im(D_1)}\ar[r]\ar[d]& 0\\
&0&0&0& \\
& 0\ar[d] & 0\ar[d] & 0\ar[d] &\\ 
0\ar[r]& \frac{\ka(D_3D_2)}{\ka(D_2)+\ka(D_3D_2)\cap\Im(D_1)} \ar[r]^>>>>>>{D_2}\ar[d]&
\frac{\ka(D_3)}{\ka(D_3)\cap\Im(D_2D_1)}\ar[r]\ar[d]& 
\frac{\ka(D_3)}{\ka(D_3)\cap\Im(D_2)}\ar[r]\ar[d]& 0\\
0\ar[r]& \frac{X_2}{\ka(D_2)+\Im(D_1)}\ar[r]^{D_2}\ar[d]&
\fc(D_2D_1)\ar[r]\ar[d]& \fc(D_2)\ar[r]\ar[d]& 0\\
0\ar[r]& \frac{X_2}{\ka(D_3D_2)+\Im(D_1)}\ar[r]^{D_2}\ar[d]&
\frac{X_3}{\ka(D_3)+\Im(D_2D_1)}\ar[r]\ar[d]& 
\frac{X_3}{\ka(D_3)+\Im(D_2)}\ar[d]\ar[r]& 0\\
&0&0&0& }
\end{gather*}
\rput{45}(3.3,12.4){$u_{1,2}\ltimes$}\rput{45}(7.3,12.3){$u_{1,23}\ltimes$}
\rput{45}(11.1,12.3){$u_{12,3}\ltimes$}
\rput{45}(3.5,10.8){$x_2\ltimes$}\rput{45}(11.4,10.8){$u_{2,3}\ltimes$}
\rput{45}(2.9,9){$v_{1,2}\ltimes$}\rput{45}(6.7,8.9){$v_{1,23}\ltimes$}
\rput{45}(10.4,9.1){$\mu\ltimes$}
\rput{45}(2.3,4.8){$\mu\ltimes$}\rput{45}(6.7,4.7){$v_{12,3}\ltimes$}
\rput{45}(11,4.7){$v_{2,3}\ltimes$}
\rput{45}(2.9,3){$w_{1,2}\ltimes$}\rput{45}(11.4,3.2){$y_2\ltimes$}
\rput{45}(2.6,1.4){$w_{1,23}\ltimes$}\rput{45}(6.7,1.4){$w_{12,3}\ltimes$}
\rput{45}(10.9,1.4){$w_{2,3}\ltimes$}
\caption{Commutative diagrams of exact sequences used in the proof of Proposition~\ref{CompET_prp}}
\label{CompDiags_fig}
\end{figure}

\noindent
In light of the left column of the first diagram and the right column of 
the second diagram in Figure~\ref{CompDiags_fig}, \eref{munew_e}, and Corollary~\ref{ses_crl},
we can take
\BE{xydfn_e}\begin{aligned}
x_2&=\la(D_1)u_{1,2}\w_{\ka(D_2)}v_{1,2}, &\quad
x_3&=\la(D_2D_1)u_{12,3}\w_{\ka(D_3)}v_{12,3}
=\la(D_2)u_{2,3}\w_{\ka(D_3)}v_{2,3}\,,\\
y_2&=v_{2,3}\w_{\fc(D_2)}w_{2,3}, &\quad
y_1&=v_{1,2}\w_{\fc(D_1)}w_{1,2}=v_{1,23}\w_{\fc(D_1)}w_{1,23}\,.
\end{aligned}\EE
Combining the above definitions of $x_2$ and $y_2$ with \eref{munew_e} and 
applying Lemma~\ref{ses_lmm2} to the two diagrams in Figure~\ref{CompDiags_fig}, we find that 
\BE{switchrel_e}\begin{split}
\la(D_1)x_2\w_{\ka(D_3D_2)}u_{2,3}&=
(-1)^{\lr{u_{12,3}}\lr{v_{1,2}}}\la(D_1)u_{1,23}\w_{\ka(D_3D_2)}v_{1,23}\,,\\
\la(D_2)w_{1,2}\w_{\fc(D_2D_1)}y_2&=
(-1)^{\lr{v_{2,3}}\lr{w_{1,23}}}v_{12,3}\w_{\fc(D_2D_1)}w_{12,3}.
\end{split}\EE
By \eref{cUDDdfn_e2},   \eref{xydfn_e}, and~\eref{switchrel_e}, the images of 
$$ x_1\!\otimes\!y_1^*\otimes x_2\!\otimes\!y_2^*\otimes x_3\!\otimes\!y_3^*
\in \la(D_1)\otimes\la(D_2)\otimes\la(D_3)$$
under $\wt{C}_{D_2\circ D_1,D_3}\circ\wt{C}_{D_1,D_2}\!\otimes\!\id$ and
$\wt{C}_{D_1,D_3\circ D_2}\circ\id\!\otimes\!\wt{C}_{D_2,D_3}$ are
\begin{equation*}\begin{split}
&(-1)^{\ep_{\tnL}+\lr{v_{2,3}}\lr{w_{1,23}}}
\big((x_1\!\w_{\ka(D_2D_1)}u_{1,2})\w_{\ka(D_3D_2D_1)}u_{12,3}\!\big)\otimes
\big(\la(D_3)w_{12,3}\w_{\fc(D_3D_2D_1)}\!y_3\big)^*
\quad\hbox{and}\\
&(-1)^{\ep_{\tnR}+\lr{u_{12,3}}\lr{v_{1,2}}}
\big(x_1\w_{\ka(D_3D_2D_1)}\!u_{1,23}\big)\otimes
\big(\la(D_3D_2)w_{1,23}\w_{\fc(D_3D_2D_1)}(\la(D_3)w_{2,3}\w_{\fc(D_3D_2)}y_3\!)\big)^*,
\end{split}\end{equation*}
respectively.
By \eref{sign_e}, the second and third identities in~\eref{munew_e},  and \eref{value_e},
these two elements of $\la(D_3D_2D_1)$ are the same, which establishes the claim.
\end{proof}

\begin{rmk}\label{Huang_rmk}
The proof of this crucial proposition in \cite[Appendix~D.2]{Huang} 
does not appear to establish anything.
Up to notational differences, it describes an expression for 
$$\big\{\wt{C}_{D_2\circ D_1,D_3}\circ\wt{C}_{D_1,D_2}\!\otimes\!\id\}
\big(x_1\!\otimes\!y_1^*\otimes x_2\!\otimes\!y_2^*\otimes x_3\!\otimes\!y_3^*\big)
\in \la(D_3\!\circ\!D_2\!\circ\!D_1)$$
without any signs and simply claims that 
$$\big\{\wt{C}_{D_1,D_3\circ D_2}\circ\id\!\otimes\!\wt{C}_{D_2,D_3}\big\}
\big(x_1\!\otimes\!y_1^*\otimes x_2\!\otimes\!y_2^*\otimes x_3\!\otimes\!y_3^*\big)
\in \la(D_3\!\circ\!D_2\!\circ\!D_1)$$
is given by the same expression, without providing an explicit formula for 
$\wt{C}_{D_1,D_2}$, using the statement of Lemma~\ref{ses_lmm2},
or indicating the significance of the grading of the lines~$\la(V)$.
As illustrated by the proof of Proposition~\ref{CompET_prp} above,
the two expressions require auxiliary terms from different vectors spaces
and it takes significant care to show that it is possible to choose them compatibly.
Furthermore, there are two typos at the end of the proof of the closely 
related \cite[Corollary~D.2.4]{Huang} with two subscripts that should be different
being the same and resulting in the order of two factors switched between the statements
of \cite[Proposition~D.2.3]{Huang} and \cite[Corollary~D.2.4]{Huang}. 
\end{rmk}

\begin{prp}[Compositions~II]\label{CompET_prp2}
For any pair $(\ft_1,\ft_2)$ of exact triples of Fredholm operators as in~\eref{cCcTdiag_e},
the diagram~\eref{compdiag_e2} induced by the isomorphisms~\eref{cUDDdfn_e}
commutes.
\end{prp}

\begin{proof} 
We continue with the notation described in the first sentence of the proof of 
Proposition~\ref{CompET_prp} and~define
$$\ft_{12}=\cC_{\cT}(\ft_1,\ft_2),\quad 
\ep_{\tnL}=(\ind\,D_1'')(\ind\,D_2')+\ep_{D_1',D_2'}+\ep_{D_1'',D_2''}+\ep_{\ft_{12}}\,,
\quad \ep_{\tnR}=\ep_{\ft_1}+\ep_{\ft_2}+\ep_{D_1,D_2}\,.$$
For $k\!=\!1,2$ and $\star=',''$, let
$$x_k^{\star}\in\la(\ka(D_k^{\star}))-0, \qquad 
y_k'\in\la(\fc(D_k'))-0,\qquad
y_k''\in\la\bigg(\frac{X_{k+1}}{\Im(\fI_{k+1})+\Im(D_k)}\bigg)-0\,.$$
With $\star$ denoting $',''$ or a blank, let
$$u^{\star}\in\la\bigg(\frac{\ka(D_2^{\star}D_1^{\star})}{\ka(D_1^{\star})}\bigg)-0,\quad
v^{\star}\in\la\bigg(\frac{\ka(D_2^{\star})}{\ka(D_2^{\star})\cap\Im(D_1^{\star})}\bigg)-0,\quad
w^{\star}\in\la\bigg(\frac{X_2^{\star}}{\ka(D_2^{\star})+\Im(D_1^{\star})}\bigg)-0;$$
see Figures~\ref{CompDiags_fig2} and~\ref{CompDiags_fig3}.
For $k\!=\!1,2,12$, let
$$\de_k\!:\ka(D_k'')\lra \fc(D_k'),$$
where $D_{12}'\!=\!D_2'D_1'$ and $D_{12}''\!=\!D_2''D_1''$, be the connecting 
homomorphisms in the sequences~\eref{ETles_e} 
corresponding to~$\ft_1$, $\ft_2$, and $\ft_{12}$, respectively, and 
$$u_k\!\in\!\la\bigg(\frac{\ka(D_k)}{\ka(D_k)\!\cap\!\Im(\fI_k)}\bigg)-0,\quad
v_k\!\in\!\la\bigg(\frac{\ka(D_k'')}{\fJ_k(\ka(D_k))}\bigg)-0,\quad
w_k\!\in\!\la\bigg(\frac{X_{k+1}'}{\fI_{k+1}^{-1}(\Im(D_k))}\bigg)-0,$$
with $\fI_{12}\!=\!\fI_1$, $\fJ_{12}\!=\!\fJ_1$, and $12\!+\!1\!\equiv\!3$; 
see Figures~\ref{CompDiags_fig2} and~\ref{CompDiags_fig3}.
Define
$$\wt{w}''\in\la\bigg(\frac{X_2}{\fJ_2^{-1}(\ka(D_2''))+\Im(D_1)}\bigg)-0
\qquad\hbox{by}\quad w''=\la(\fJ_2)\wt{w}''.$$
Below we choose these elements in a compatible~way.\\

\noindent
In order to describe the two relevant signs, we define
\begin{equation*}\begin{split}
A&=(\ind\,D_2'')(\fc_1'\!+\!\fc_1''\!+\!\fc_2'\big)
+\big(\ind\,D_1''+\ind\,D_2'\big)\fc_1'+\fc_1''\ka_2',
\qquad C=\fc_1'+\fc_1''+\fc_2'+\fc_2''\,,\\
A_{\tnL}&=\ka_1''\ka_2'+(\ka_1''\!+\!\ka_2'')\lr{v'}
+(\fc_1'\!+\!\fc_2'\big)\lr{v''}
+\big(\lr{v'}\!+\!\lr{v''}\big)\lr{v_{12}},\\
A_{\tnR}&=\fc_1''\fc_2'+(\ka_2'\!+\!\ka_2'')\lr{v_1}
+(\fc_1'\!+\!\fc_1'')\lr{v_2} +\big(\lr{v_1}\!+\!\lr{v_2}\big)\lr{v},
\end{split}\end{equation*}
where $\ka_i^{\star}\!=\!\fd(\ka(D_i^{\star}))$ and 
$\fc_i^{\star}\!=\!\fd(\fc(D_i^{\star}))$ with $i\!=\!1,2$ and $\star\!=\!',''$.
Applying Lemma~\ref{ses_lmm} to the exact sequences~\eref{ETles_e} with $D^{\star}$
replaced by $D_k^{\star}$, for $\star\!=\!',''$ and blank and $k\!=\!1,2,12$, 
and \eref{ETles_e2} with $D_k$ replaced by $D_k^{\star}$, 
for $\star\!=\!',''$ and blank and $k\!=\!1,2$, we obtain
\BE{presign_e}\begin{aligned}
\ind\,D_2&=\ind\,D_2'+\ind\,D_2'',&\quad
\fd(\fc(D_k))&=\fd(\fc(D_k'))+\fd(\fc(D_k''))-\lr{v_k}\,,\\
\ind\,D_2''D_1''&=\ind\,D_1''+\ind\,D_2'',&\quad
\fd(\fc(D_2^{\star}D_1^{\star}))&=\fd(D_1^{\star})+\fd(\fc(D_2^{\star}))-\lr{v^{\star}}.
\end{aligned}\EE
From this, we find that 
\BE{sign2_e}\begin{split}
\ep_{\tnL}&=A+C\big(\lr{v'}\!+\!\lr{v''}\!+\!\lr{v_{12}}\big)
+A_{\tnR}+\lr{v_{12}},\\
\ep_{\tnR}&=A+C\big(\lr{v_1}\!+\!\lr{v_2}\!+\!\lr{v}\big)
+A_{\tnL}+\lr{v_1}\!+\!\lr{v_2},
\end{split}\EE
modulo 2.
By the identities in the second column in~\eref{presign_e},
\BE{sign_e4} \lr{v}+\lr{v_1}+\lr{v_2} =C-\fd(\fc(D_2D_1))
=\lr{v_{12}}+\lr{v'}+\lr{v''}\,.\EE
From the exact sequences~\eref{ETles_e} and~\eref{ETles_e2}, we also find  
\begin{equation*}
\ka_i''=\lr{u_i}+\lr{v_i},\quad \fc_i'=\lr{v_i}+\lr{w_i},\qquad
\ka_2^{\star}=\lr{u^{\star}}+\lr{v^{\star}},\quad
\fc_1^{\star}=\lr{v^{\star}}+\lr{w^{\star}},
\end{equation*}
where $i\!=\!1,2$ and $\star\!=\!',''$.
From this, we find that 
\BE{sign3_e}\begin{split}
&\lr{u'}\lr{u_1}+\lr{w''}\lr{w_2}+\lr{w'}\lr{u_2}+\lr{u''}\lr{w_1}
+\lr{v}\lr{v_{12}}\\
&\hspace{1in}
\cong A_{\tnL}+A_{\tnR}+
\big(\lr{v_1}\!+\!\lr{v_2}\!+\!\lr{v_{12}}\big)\big(\lr{v'}\!+\!\lr{v''}\!+\!\lr{v}\big)
\end{split}\EE
modulo~2.\\

\noindent
In light of the bottom row and right column in the first diagram in Figure~\ref{CompDiags_fig2},
the top row and left column in the second diagram in Figure~\ref{CompDiags_fig2},
and Lemma~\ref{ses_lmm}, we can take 
\BE{muetanew_e}\begin{aligned}
u&=\la(\fI_1)u'\w_{\frac{\ka(D_2D_1)}{\ka(D_1)}}\mu,&\qquad
u_{12}&=u_1\w_{\frac{\ka(D_2D_1)}{\ka(D_2D_1)\cap\Im(\fI_1)}}\mu,\\
w&=\eta\w_{\frac{X_2}{\ka(D_2)+\Im(D_1)}}\wt{w}'', &\qquad
w_{12}&=\la(\fI_3^{-1}\!\circ\!D_2)\eta\w_{\frac{X_3'}{\fI_3^{-1}(\Im(D_2D_1))}}w_2
\end{aligned}\EE
for some
$$\mu\in\la\bigg(\frac{\ka(D_2D_1)}{\ka(D_1)+\ka(D_2D_1)\cap\Im(\fI_1)}\bigg)-0\,,
\qquad
\eta\in\la\bigg(\frac{\fJ_2^{-1}(\ka(D_2''))}{\ka(D_2)+
\fJ_2^{-1}(\ka(D_2''))\cap\Im(D_1)}\bigg)-0.$$
Along with Lemma~\ref{ses_lmm2} applied to the two diagrams in Figure~\ref{CompDiags_fig2}, 
these equalities insure that 
\BE{value2_e}\begin{split}
&\big((\la(\fI_1)x_1'\!\w_{\ka(D_1)}\!u_1)\!\w_{\ka(D_2D_1)}u\big)\otimes
\big(\la(D_2)w\!\w_{\fc(D_2D_1)}(\la(\fI_3)w_2\!\w_{\fc(D_2)}\!y_2'')\big)^*\\
&\hspace{.3in}=(-1)^{\lr{u'}\lr{u_1}+\lr{\wt{w}''}\lr{w_2}}
\big(\la(\fI_1)(x_1'\!\w_{\ka(D_2'D_1')}u')\w_{\ka(D_2D_1)}u_{12}\big)\\
&\hspace{2in}
\otimes \big(\la(\fI_3)w_{12}\w_{\fc(D_2D_1)}(\la(D_2)\wt{w}''
\w_{\frac{X_3}{\Im(\fI_3)+\Im(D_2D_1)}}y_2'')\big)^*,
\end{split}\EE 
in $\la(D_2D_1)$.\\

\begin{figure}
\begin{gather*}
\xymatrix{&0\ar[d]&0\ar[d]&0\ar[d]&\\
0\ar[r]& \ka(D_1')\ar[r]^{\fI_1}\ar[d]& \ka(D_1)\ar[r]\ar[d]& 
\frac{\ka(D_1)}{\ka(D_1)\cap\Im(\fI_1)}\ar[d]\ar[r]&0\\
0\ar[r]& \ka(D_2'D_1')\ar[r]^{\fI_1}\ar[d]& \ka(D_2D_1)\ar[r]\ar[d]& 
\frac{\ka(D_2D_1)}{\ka(D_2D_1)\cap\Im(\fI_1)}\ar[d]\ar[r]&0\\
0\ar[r]& \frac{\ka(D_2'D_1')}{\ka(D_1')}\ar[r]^{\fI_1}\ar[d]& 
\frac{\ka(D_2D_1)}{\ka(D_1)}\ar[r]\ar[d]& 
\frac{\ka(D_2D_1)}{\ka(D_1)+\ka(D_2D_1)\cap\Im(\fI_1)}\ar[d]\ar[r]&0\\
&0&0&0&\\
&0\ar[d]&0\ar[d]&0\ar[d]&\\
0\ar[r]& \frac{\fJ_2^{-1}(\ka(D_2''))}{\ka(D_2)+\fJ_2^{-1}(\ka(D_2''))\cap\Im(D_1)}
\ar[r]\ar[d]^{\fI_3^{-1}\circ D_2}&
\frac{X_2}{\ka(D_2)+\Im(D_1)}\ar[r]\ar[d]^{D_2}&
\frac{X_2}{\fJ_2^{-1}(\ka(D_2''))+\Im(D_1)}\ar[r]\ar[d]^{D_2}\ar[r]\ar[d]&0\\
0\ar[r]& \frac{X_3'}{\fI_3^{-1}(\Im(D_2D_1))}\ar[d]\ar[r]^{\fI_3}& \fc(D_2D_1)\ar[r]\ar[d]&
\frac{X_3}{\Im(\fI_3)+\Im(D_2D_1)}\ar[d]\ar[r]& 0\\
0\ar[r]& \frac{X_3'}{\fI_3^{-1}(\Im(D_2))}\ar[d]\ar[r]^{\fI_3}& \fc(D_2)\ar[r]\ar[d]&
\frac{X_3}{\Im(\fI_3)+\Im(D_2)}\ar[d]\ar[r]& 0\\
&0&0&0&}
\end{gather*}
\rput{45}(3.7,13){$x_1'\ltimes$}\rput{45}(11.1,12.8){$u_1\ltimes$}
\rput{45}(10.8,11.1){$u_{12}\ltimes$}
\rput{45}(3.7,9.6){$u'\ltimes$}\rput{45}(7.7,9.6){$u\ltimes$}
\rput{45}(10.5,9.6){$\mu\ltimes$}
\rput{45}(2.4,5.1){$\eta\ltimes$}\rput{45}(7.3,5.2){$w\ltimes$}
\rput{45}(10.6,5.1){$\wt{w}''\ltimes$}
\rput{45}(3.1,3.3){$w_{12}\ltimes$}
\rput{45}(3.4,1.5){$w_2\ltimes$}\rput{45}(11,1.6){$y_2''\ltimes$}
\caption{Commutative diagrams of exact sequences used in the proof of Proposition~\ref{CompET_prp2}}
\label{CompDiags_fig2}
\end{figure}

\noindent
We next make use of the three commutative diagrams in Figure~\ref{CompDiags_fig3}.
They can be viewed as the three coordinate planes in $\Z^3$,
with all three diagrams sharing the center and any pair sharing an axis.
We choose $v',w',u_2,v_2,v_1,u'',\mu$ arbitrarily, then find $y_1',w_1,v_{12},v$
so that 
\BE{comp1_e}\begin{split}
v'\w_{\fc(D_1)}w'=y_1'=\la(\de_1)v_1\w_{\fc(D_1)}w_1, &\quad
v_1\w_{\frac{\ka(D_2''D_1'')}{\fJ_1(\ka(D_1))}}u''
=(-1)^{\lr{u''}\lr{w_1}}\la(\fJ_1)\mu\w_{\frac{\ka(D_2''D_1'')}{\fJ_1(\ka(D_2))}}v_{12},\\
\la(\fI_2)v'\w_{\frac{\ka(D_2)}{\ka(D_2)\cap\Im(\fI_2D_1)}}u_2
&=(-1)^{\lr{w'}\lr{u_2}}\la(D_1)\mu\w_{\frac{\ka(D_2)}{\ka(D_2)\cap\Im(\fI_2D_1)}}v\,,
\end{split}\EE
and finally take $\eta,x_2'',v''$ so that 
\BE{comp2_e}\begin{split}
\la(\fI_2)w'\w_{\frac{\fJ_2^{-1}(\ka(D_2''))}{\ka(D_2)+\Im(\fI_2D_1')}}v_2
&=(-1)^{\lr{v}\lr{v_{12}}}
\la(D_1\!\circ\!\fJ_1^{-1})v_{12}\w_{\frac{\fJ_2^{-1}(\ka(D_2''))}{\ka(D_2)+\Im(\fI_2D_1')}}\eta\,,\\
\la(\fJ_2)u_2\w_{\ka(D_2'')}v_2&=x_2''=\la(D_1'')u''\w_{\ka(D_2'')}v''\,.
\end{split}\EE
By Lemma~\ref{ses_lmm2} applied to the three commutative squares in Figure~\ref{CompDiags_fig3},
\eref{comp1_e}, and \eref{comp2_e},
\begin{equation*}\begin{split}
&\la(D_1\!\circ\!\fJ_1^{-1})\big(v_1\w_{\frac{\ka(D_2''D_1'')}{\fJ_1(\ka(D_1))}}u''\big)
\w_{\frac{\fJ_2^{-1}(\ka(D_2''))}{\Im(\fI_2D_1')}}
\big(\la(\fI_2)w_1\w_{\frac{\fJ_2^{-1}(\ka(D_2''))}{\fJ_2^{-1}(\ka(D_2''))\cap\Im(D_1)}}v''\big)\\
&\quad =(-1)^{\lr{u''}\lr{w_1}} \la(\fI_2)\big(\la(\de_1)v_1\w_{\fc(D_1)}w_1\big)
\w_{\frac{\fJ_2^{-1}(\ka(D_2''))}{\Im(\fI_2D_1')}}
\big(\la(D_1'')u''\w_{\ka(D_2'')}v''\big)\\
&\quad=(-1)^{\lr{u''}\lr{w_1}}  \la(\fI_2)\big(v'\w_{\fc(D_1)}w'\big) 
\w_{\frac{\fJ_2^{-1}(\ka(D_2''))}{\Im(\fI_2D_1')}}
\big(\la(\fJ_2)u_2\w_{\ka(D_2'')}v_2\big)\\
&\quad=(-1)^{\lr{u''}\lr{w_1}+\lr{w'}\lr{u_2}}\big(\la(\fI_2)v'
\w_{\frac{\ka(D_2)}{\ka(D_2)\cap\Im(\fI_2D_1)}}u_2\big)
\w_{\frac{\fJ_2^{-1}(\ka(D_2''))}{\Im(\fI_2D_1')}}
\big(\la(\fI_2)w'\w_{\frac{\fJ_2^{-1}(\ka(D_2''))}{\ka(D_2)+\Im(\fI_2D_1')}}v_2\big)\\
&\quad 
=(-1)^{\lr{u''}\lr{w_1}+\lr{v}\lr{v_{12}}}
\big(\la(D_1)\mu\w_{\frac{\ka(D_2)}{\ka(D_2)\cap\Im(\fI_2D_1)}}v\big)
\w_{\frac{\fJ_2^{-1}(\ka(D_2''))}{\Im(\fI_2D_1')}}
\big(\la(D_1\!\circ\!\fJ_1^{-1})v_{12}
 \w_{\frac{\fJ_2^{-1}(\ka(D_2''))}{\ka(D_2)+\Im(\fI_2D_1')}}\eta\big)\\
&\quad=(-1)^{\lr{u''}\lr{w_1}}
\la(D_1\!\circ\!\fJ_1^{-1})
\big(\la(\fJ_1)\mu\w_{\frac{\ka(D_2''D_1'')}{\fJ_1(\ka(D_2))}}v_{12}\big)
\w_{\frac{\fJ_2^{-1}(\ka(D_2''))}{\Im(\fI_2D_1')}} 
\big(v\w_{\frac{\fJ_2^{-1}(\ka(D_2''))}{\fJ_2^{-1}(\ka(D_2''))\cap\Im(D_1)}}\!\eta\big).
\end{split}\end{equation*}
Along with the second equation in~\eref{comp1_e}, this gives
\BE{comp3_e}
\la(\fI_2)w_1\w_{\frac{\fJ_2^{-1}(\ka(D_2''))}{\fJ_2^{-1}(\ka(D_2''))\cap\Im(D_1)}}v''
=
\big(v\w_{\frac{\fJ_2^{-1}(\ka(D_2''))}{\fJ_2^{-1}(\ka(D_2''))\cap\Im(D_1)}}\eta\big).
\EE\\

\begin{figure}
\begin{gather*}
\xymatrix{&0\ar[d]&0\ar[d]&0\ar[d]&\\
0\ar[r]& \frac{\ka(D_2')}{\ka(D_2')\cap\Im(D_1')}\ar[d]\ar[r]^{\fI_2}&
\frac{\ka(D_2)}{\ka(D_2)\cap\Im(\fI_2D_1')}\ar[d]\ar[r]&
\frac{\ka(D_2)}{\ka(D_2)\cap\Im(\fI_2)}\ar[d]^{\fJ_2}\ar[r]& 0\\
0\ar[r]& \fc(D_1')\ar[d]\ar[r]^{\fI_2}&
\frac{\fJ_2^{-1}(\ka(D_2''))}{\Im(\fI_2D_1')}\ar[d]\ar[r]^{\fJ_2}& \ka(D_2'')\ar[d]\ar[r]& 0\\
0\ar[r]& \frac{X_2'}{\ka(D_2')+\Im(D_1')}\ar[d]\ar[r]^{\fI_2}&
\frac{\fJ_2^{-1}(\ka(D_2''))}{\ka(D_2)+\Im(\fI_2D_1')}\ar[d]\ar[r]^{\fJ_2}&
\frac{\ka(D_2'')}{\fJ_2(\ka(D_2))}\ar[d]\ar[r]& 0\\ 
&0\ar[d]&0\ar[d]&0\ar[d]&\\
0\ar[r]& \frac{\ka(D_1'')}{\fJ_1(\ka(D_1))}\ar[d]^{\de_1}\ar[r]& 
\frac{\ka(D_2''D_1'')}{\fJ_1(\ka(D_1))}\ar[d]^{D_1\circ\fJ_1^{-1}}\ar[r]& 
\frac{\ka(D_2''D_1'')}{\ka(D_1'')}\ar[d]^{D_1''}\ar[r]& 0\\ 
0\ar[r]& \fc(D_1')\ar[d]\ar[r]^{\fI_2}&
\frac{\fJ_2^{-1}(\ka(D_2''))}{\Im(\fI_2D_1')}\ar[d]\ar[r]^{\fJ_2}& \ka(D_2'')\ar[d]\ar[r]& 0\\
0\ar[r]& \frac{X_2'}{\fI_2^{-1}(\Im(D_1'))}\ar[d]\ar[r]^{\fI_2}&
\frac{\fJ_2^{-1}(\ka(D_2''))}{\fJ_2^{-1}(\ka(D_2''))\cap\Im(D_1)}\ar[d]\ar[r]^{\fJ_2}& 
\frac{\ka(D_2'')}{\ka(D_2'')\cap\Im(D_1'')}\ar[d]\ar[r]& 0\\ 
&0\ar[d]&0\ar[d]&0\ar[d]&\\
0\ar[r]& \frac{\ka(D_2D_1)}{\ka(D_1)+\ka(D_2D_1)\cap\Im(\fI_1)}\ar[d]^{D_1}\ar[r]^{\fJ_1}&
\frac{\ka(D_2''D_1'')}{\fJ_1(\ka(D_1))} \ar[d]^{D_1\circ\fJ_1^{-1}}\ar[r]&
\frac{\ka(D_2''D_1'')}{\fJ_1(\ka(D_2D_1))} \ar[d]^{D_1\circ\fJ_1^{-1}}\ar[r]& 0\\
0\ar[r]& \frac{\ka(D_2)}{\ka(D_2)\cap\Im(\fI_2D_1')}\ar[d]\ar[r]&
\frac{\fJ_2^{-1}(\ka(D_2''))}{\Im(\fI_2D_1')} \ar[d]\ar[r]&
\frac{\fJ_2^{-1}(\ka(D_2''))}{\ka(D_2)+\Im(\fI_2D_1')} \ar[d]\ar[r]& 0\\
0\ar[r]& \frac{\ka(D_2)}{\ka(D_2)\cap\Im(D_1)}\ar[d]\ar[r]&
\frac{\fJ_2^{-1}(\ka(D_2''))}{\fJ_2^{-1}(\ka(D_2''))\cap\Im(D_1)} \ar[d]\ar[r]&
\frac{\fJ_2^{-1}(\ka(D_2''))}{\ka(D_2)+\fJ_2^{-1}(\ka(D_2''))\cap\Im(D_1)} \ar[d]\ar[r]& 0\\
&0&0&0&}
\end{gather*}
\rput{45}(2.5,18.3){$v'\ltimes$}\rput{45}(11.2,18.2){$u_2\ltimes$}
\rput{45}(3.1,16.6){$y_1'\ltimes$}\rput{45}(11.7,16.6){$x_2''\ltimes$}
\rput{45}(2.6,14.8){$w'\ltimes$}\rput{45}(11.5,14.7){$v_2\ltimes$}
\rput{45}(2.8,11.8){$v_1\ltimes$}\rput{45}(11.6,11.8){$u''\ltimes$}
\rput{45}(3.1,10.1){$y_1'\ltimes$}\rput{45}(11.7,10.1){$x_2''\ltimes$}
\rput{45}(2.7,8.1){$w_1\ltimes$}\rput{45}(11.1,8.2){$v''\ltimes$}
\rput{45}(1.9,5.2){$\mu\ltimes$}\rput{45}(11.2,5.1){$v_{12}\ltimes$}
\rput{45}(2.6,1.7){$v\ltimes$}\rput{45}(10.3,1.7){$\eta\ltimes$}
\caption{Commutative diagrams of exact sequences used in the proof of Proposition~\ref{CompET_prp2}}
\label{CompDiags_fig3}
\end{figure}

\noindent
In addition to the choices of $y_1'$ and $x_2''$ specified 
in~\eref{comp1_e} and~\eref{comp2_e},  we take 
\BE{xydfn2_e}\begin{aligned}
x_1''&=\la(\fJ_1)u_1\w_{\ka(D_1'')} v_1\,,&\qquad 
\la(\fJ_2)y_1''&=v''\w_{\fc(D_2')}w''\,,\\
x_2'&=\la(D_1')u'\w_{\ka(D_2')}v'\,,&\qquad
y_2'&=\la(\de_2)v_2\w_{\fc(D_2')}w_2\,.
\end{aligned}\EE
By \eref{muetanew_e}, the last two equations in~\eref{comp1_e}, 
the first equation in~\eref{comp2_e}, \eref{comp3_e}, and Corollary~\ref{ses_crl}, 
\BE{comp4_e}\begin{split}
\la(\fI_2)w_1\w_{\fc(D_1)}y_1''=v\w_{\fc(D_1)}w\,, &\quad
\la(\fI_2)x_2'\w_{\la(D_2)}u_2=(-1)^{\lr{w'}\lr{u_2}}\la(D_1)u\w_{\ka(D_2)}v\,,\\
\la(D_2')w'\w_{\fc(D_2'D_1')}y_2'&=(-1)^{\lr{v}\lr{v_{12}}}
\la(\de_{12})v_{12}\w_{\fc(D_2'D_1')}w_{12}\,,\\
x_1''\w_{\ka(D_2''D_1'')}u''&=(-1)^{\lr{u''}\lr{w_1}}
\la(\fJ_1)u_{12}\w_{\ka(D_2''D_1'')}v_{12}\,;
\end{split}\EE
the third identity above also uses
$$\la(D_2')=\la(\fI_3^{-1}\!\circ\!D_2)\circ\la(\fI_2)\,, \quad
\la(\de_2)=\la(\fI_3^{-1}\!\circ\!D_2)\circ\la(\fJ_2)^{-1}\,,\quad
\la(\de_{12})=\la(\fI_3^{-1}\!\circ\!D_2)\circ\la(D_1\!\circ\!\fJ_1^{-1})\,.$$
By \eref{cUDDdfn_e}, the second equality in the first identity in~\eref{comp1_e},
the first equality in the last identity in~\eref{comp2_e},
the first and last equations in~\eref{xydfn2_e},  
\eref{cUDDdfn_e2}, and the first two equations in~\eref{comp4_e},
the image of 
$$x_1'\!\otimes\!y_1'^*\otimes x_1''\!\otimes\!(\la(\fJ_2)y_1'')^*\otimes
x_2'\!\otimes\!y_2'^* \otimes x_2''\!\otimes\!(\la(\fJ_3)y_2'')^* \in
\la(D_1')\otimes\la(D_1'')\otimes\la(D_2')\otimes\la(D_2'')$$
under $\wt{C}_{D_1,D_2}\circ\Psi_{\ft_1}\!\otimes\!\Psi_{\ft_2}$ is 
$$(-1)^{\ep_{\tnR}+\lr{w'}\lr{u_2}}
\big((\la(\fI_1)x_1'\!\w_{\ka(D_1)}\!u_1)\!\w_{\ka(D_2D_1)}u\big)\otimes
\big(\la(D_2)w\!\w_{\fc(D_2D_1)}(\la(\fI_3)w_2\!\w_{\fc(D_2)}\!y_2'')\big)^*\,.$$
By \eref{cUDDdfn_e2}, the first equality in the first identity in~\eref{comp1_e},
the second equality in the last identity in~\eref{comp2_e},
the second and third equations in~\eref{xydfn2_e},  
\eref{cUDDdfn_e}, and the last two equations in~\eref{comp4_e},
the image of this element under 
$\Psi_{\ft_{12}}\circ \wt\cC_{D_1',D_2'}\!\otimes\!\wt\cC_{D_1'',D_2''}
 \circ \id\!\otimes\!R\!\otimes\!\id$ is
\begin{equation*}\begin{split}
&(-1)^{\ep_{\tnL}+\lr{v}\lr{v_{12}}+\lr{u''}\lr{w_1}}
\big(\la(\fI_1)(x_1'\!\w_{\ka(D_2'D_1')}u')\w_{\ka(D_2D_1)}u_{12}\big)\\
&\hspace{2in}
\otimes \big(\la(\fI_3)w_{12}\w_{\fc(D_2D_1)}(\la(D_2)\wt{w}''
\w_{\frac{X_3}{\Im(\fI_3)+\Im(D_2D_1)}}y_2'')\big)^*\,.
\end{split}\end{equation*}
By \eref{value2_e} and \eref{sign2_e}-\eref{sign3_e},
these two elements of $\la(D_2D_1)$ are the same, which establishes the~claim.
\end{proof}

\subsection{Stabilizations of Fredholm operators}
\label{Reg_subs}

\noindent
We now describe stabilizations of Fredholm operators which are 
used to topologize determinant line bundles in the next section.
In this subsection, we use them to deduce the Exact Squares property
on page~\pageref{ESQs_prop} for the isomorphisms~\eref{ETisom_e} 
from the Naturality~II, Normalization~II, and Compositions~II properties
via Lemmas~\ref{ses_lmm2} and~\ref{StabTrip_lmm}.\\

\noindent
For any Banach vector space $X$, $N\!\in\!\Z^{\ge0}$, 
and homomorphism $\Th\!:\R^N\!\lra\!Y$, let 
$$\io_{X;N}\!:X\!\lra\!X\!\oplus\!\R^N, \qquad
D_{\Th}\!:X\!\oplus\!\R^N\lra Y, \quad\hbox{and}\quad
\hat\cI_{\Th;D}\!: \la(D)\lra \la(D_{\Th})$$
be as in Section~\ref{outline_subs}.
Since $D=D_{\Th}\!\circ\!\io_{X;N}$ and the projection $\pi_2\!:X\!\oplus\!\R^N\!\lra\!\R^N$
identifies $\fc(\io_{X;N})$ with~$\R^N$, 
the corresponding exact triple \eref{compEX_e} gives rise to the isomorphism
\BE{cIThD_e1}\cI_{\Th;D}\!: \la(D_{\Th})\lra \la(D), \qquad
\cI_{\Th;D}(\si)=\ti\cC_{\io_{X;N},D_{\Th}}\big(1\!\otimes\!
(\OmN^*\!\circ\!\la(\pi_2))\otimes \si\big),\EE
where $\OmN$ is the standard volume tensor on~$\R^N$ as before.

\begin{lmm}\label{cIdual_lmm}
For any collection of exact triple isomorphisms~$\Psi_{\ft}$ as in~\eref{ETisom_e} satisfying 
the Normalization~III and Compositions~II properties and $N\!\in\!\Z^{\ge0}$,
there exists $A_N\!\in\!\R^*$ with the following property.
For any Banach vector spaces~$X$ and~$Y$,
homomorphism $\Th\!:\R^N\!\lra\!Y$, and $D\!\in\!\cF(X,Y)$,
the isomorphisms~\eref{cIThD_e2} and~\eref{cIThD_e1}
induced by the isomorphisms~$\Psi_{\ft}$  satisfy
\BE{cIdualcomp_e}\cI_{\Th;D}\circ\hat\cI_{\Th;D}=(-1)^{(\ind\,D)N}A_N\,\id_{\la(D)}.\EE
\end{lmm}

\begin{proof} Let $A_N\!\in\!\R^*$ be such that
$$\ti\cC_{\io_{0;N},j_N}\!:\la(\io_{0;N})\!\otimes\!\la(j_N)\lra
\la\big(\id_0\!=\!j_N\!\circ\!\io_{0;N}\big), \quad
\ti\cC_{\io_{0;N},j_N}\big((1\!\otimes\!\OmN^*)\!\otimes\!(\OmN\!\otimes\!1^*)\big)
=A_N\,1\!\otimes\!1^*.$$
By the Compositions~II and Normalization~III properties applied to the diagram
$$\xymatrix{0\ar[r]& X\ar[d]^{\id}\ar[r]^{\id}& 
X\ar[d]^{\io_{X;N}}\ar[r]& 0\ar[d]^{\io_{0;N}}\ar[r]&0\\
0\ar[r]& X\ar[d]^D\ar[r]^{\io_{X;N}}& X\!\oplus\!\R^N\ar[d]^{D_{\Th}}\ar[r]^{\pi_2}&
\R^N\ar[d]^{j_N}\ar[r]&0\\
0\ar[r]& Y\ar[r]^{\id}& Y\ar[r]& 0\ar[r]&0\,,}$$
the diagram 
$$\xymatrix{\la(D)\ar[d]|{(-1)^{(\ind\,D)N}A_N} 
 \ar[rr]^{\hat\cI_{\Th;D}} &&\la(D_{\Th})\ar[d]|{\cI_{\Th;D}}\\
\la(D) \ar[rr]^{\id} && \la(D)}$$
commutes.
This gives~\eref{cIdualcomp_e}.
\end{proof}

\noindent
For a short exact sequence
$$0\lra \R^{N'}\stackrel{\fI}\lra \R^N \stackrel{\fJ}\lra \R^{N''}\lra0$$
of vector spaces, we denote by $\fc_{\fI,\fJ}$ the exact triple
$$\xymatrix{0\ar[r]&0\ar[r]\ar[d]&0\ar[r]\ar[d]&0\ar[r]\ar[d]&0\\
0\ar[r]&\R^{N'}\ar[r]^{\fI}& \R^N\ar[r]^{\fJ}& \R^{N''}\ar[r]&0}$$
of Fredholm operators.
Given a collection of exact triple isomorphisms as in~\eref{ETisom_e},
define $A_{\fI,\fJ}\!\in\!\R^*$ by
\BE{Asdfn_e} 
\Psi_{\fc_{\fI,\fJ}}\big(1\otimes\!\Om_{N'}^*\!
\otimes 1\!\otimes\!\Om_{N''}^*\big)=A_{\fI,\fJ}^{-1}(-1)^{N'N''}\,
\big(1\otimes\OmN^*\big).\EE
In the case of the collection of exact triple isomorphisms given by~\eref{cUDDdfn_e}, 
$$\Om_{N'}\w_{\R^N}\Om_{N''}=A_{\fI,\fJ}\,\OmN\,,$$
where $\w_{\R^N}$ is the isomorphism provided by Lemma~\ref{ses_lmm} for the $(\fI,\fJ)$
short exact sequence above.\\

\noindent
For an exact triple~$\fs$ of vector-space homomorphisms of the form 
\BE{StbETdfn_e}\begin{split}
\xymatrix{0\ar[r]& \R^{N'}\ar[d]^{\Th'}\ar[r]^{\fI}& \R^N\ar[d]^{\Th}\ar[r]^{\fJ}& 
\R^{N''}\ar[d]^{\Th''}\ar[r]&0\\
0\ar[r]& Y'\ar[r]^{\fI_Y}& Y\ar[r]^{\fJ_Y}& Y''\ar[r]&0\,,}
\end{split}\EE
let $A_{\fs}\!=\!A_{\fI,\fJ}$.
If $\ft$ is an exact triple of Fredholm operators as in~\eref{cTdiag_e}, 
we denote by $\ft_{\fs}$ the exact triple 
$$\xymatrix{0\ar[r]& X'\!\oplus\!\R^{N'} \ar[rr]^{\fI_X\oplus\fI} \ar[d]|{D_{\Th'}'}&&
X\!\oplus\!\R^N \ar[rr]^{\fJ_X\oplus\fJ}\ar[d]|{D_{\Th}} && 
X''\!\oplus\!\R^{N''} \ar[r]\ar[d]|{D_{\Th''}''} & 0\\
0\ar[r]& Y' \ar[rr]^{\fI_Y}&& Y \ar[rr]^{\fJ_Y}&& Y'\ar[r]& 0}$$
of Fredholm operators.

\begin{lmm}\label{StabTrip_lmm}
Let $\{\Psi_{\ft}\}_{\ft}$ be a family of exact triple isomorphisms as in~\eref{ETisom_e} satisfying 
the Naturality~II and Compositions~II properties.
For every exact triple~$\ft$ of Fredholm operators as in~\eref{cTdiag_e}
and for every  exact triple~$\fs$ of homomorphisms as in~\eref{StbETdfn_e},
the diagram
\begin{equation*}\begin{split}
\xymatrix{\la(D_{\Th'}')\otimes\la(D_{\Th''}'') 
\ar[rrrrr]^{\Psi_{\ft_{\fs}}}
\ar[d]|{\cI_{\Th';D'}\otimes\cI_{\Th'';D''}} &&&&&
\la(D_{\Th}) \ar[d]^{\cI_{\Th;D}} \\
\la(D')\otimes\la(D'')
\ar[rrrrr]^>>>>>>>>>>>>>>>>>>>>>>>>>>{(-1)^{(\ind\,D')N''}A_{\fs}\Psi_{\ft}} 
&&&&& \la(D)}
\end{split}\end{equation*}
of isomorphisms induced by the isomorphisms~\eref{ETisom_e} commutes.
\end{lmm}

\begin{proof}
By our assumptions, the diagram
$$\xymatrix{ 0\ar[r]& X' \ar[rr]^{\fI_X}\ar[d]|{\io_{X';N'}} &&
X \ar[rr]^{\fJ_X}\ar[d]|{\io_{X;N}} && X'' \ar[r]\ar[d]|{\io_{X'';N''}} & 0\\
0\ar[r]& X'\!\oplus\!\R^{N'} \ar[rr]^{\fI_X\oplus\fI} \ar[d]|{D_{\Th'}'}&&
X\!\oplus\!\R^N \ar[rr]^{\fJ_X\oplus\fJ}\ar[d]|{D_{\Th}} && 
X''\!\oplus\!\R^{N''} \ar[r]\ar[d]|{D_{\Th''}''} & 0\\
0\ar[r]& Y' \ar[rr]^{\fI_Y}&& Y \ar[rr]^{\fJ_Y}&& Y'\ar[r]& 0}$$
commutes.
By Naturality~II applied to the exact triple~$\ft_{\tnT}$ in the top half
of this diagram and the exact triple~$\fc_{\fI,\fJ}$ 
and by~\eref{Asdfn_e},
$$\Psi_{\ft_{\tnT}}\big(1\otimes\!(\Om_{N'}^*\!\circ\!\la(\pi_2'))\!
\otimes 1\!\otimes\!(\Om_{N''}^*\!\circ\!\la(\pi_2'')) \big)
=A_{\fs}^{-1}(-1)^{N'N''}\,1\otimes(\OmN^*\!\circ\!\la(\pi_2)),$$
where $\pi_2^{\star}\!:X^{\star}\!\oplus\!\R^{N^{\star}}\lra\!\R^{N^{\star}}$
is the projection on the second component and $\star=',''$ or blank.
Thus, the claim follows from the Compositions~II property applied to the above
diagram, along with~\eref{cIThD_e1}.
\end{proof}

\begin{crl}[Exact Squares]\label{ExSQ_crl}
A family of exact triple isomorphisms~$\Psi_{\ft}$ as in~\eref{ETisom_e} satisfying 
the Naturality~II, Normalization~II, and Compositions~II properties
also satisfies the Exact Squares property on page~\pageref{ESQs_prop}.
\end{crl}

\begin{proof} 
We augment the domains in~\eref{SQexact_e1} by a commuting square of homomorphisms between
finite-dimensional vector spaces, 
obtaining a version of the commutative diagram~\eref{SQexact_e1} with surjective
Fredholm operators. 
The conclusion of this corollary holds for such a diagram by 
the Normalization~II property and Lemma~\ref{ses_lmm2}.
The diagrams~\eref{SQexact_e2} corresponding to the original
and new diagrams~\eref{SQexact_e1} are related by Lemma~\ref{StabTrip_lmm}.
This gives rise to a cube of commuting diagrams; see Figure~\ref{CubeDiag_fig}.
We put the new version of~\eref{SQexact_e2} on the back face
and the diagrams arising from Lemma~\ref{StabTrip_lmm} on the top, right, bottom, 
and left faces.
This forces some coefficients~$A_{\star}$ on each edge of the front face
in order to make the last four diagrams commute.
The resulting coefficient distribution on the edges of the front face may be different
from the coefficients distribution (all $A_{\star}\!=\!1$) on the original version 
of~\eref{SQexact_e2}.
However, by Lemma~\ref{ses_lmm2}, the two coefficient distributions are equivalent at least if  
the original diagram consists of surjective Fredholm operators.
Since the coefficients depend only on the supplementary commuting square of 
homomorphisms between finite-dimensional vector spaces
and on the parities of the indices of the Fredholm operators (not the dimensions of 
their kernels or cokernels), it follows that the two coefficient distributions are
equivalent in all cases; this establishes Corollary~\ref{ExSQ_crl}.\\

\begin{sidewaysfigure}
\thisfloatpagestyle{empty}
\begin{gather*}
\xymatrix{
&\la(\ti{D}_{\tnT\tnL})\!\otimes\!\la(\ti{D}_{\tnB\tnL})
                    \!\otimes\!\la(\ti{D}_{\tnT\tnR})\!\otimes\!\la(\ti{D}_{\tnB\tnR}) 
 \ar@{--}[dd]|{~~\ti\Psi_{\tnL}\otimes\ti\Psi_{\tnR}}
\ar[rr]^>>>>>>>>>>>>>>>>>>>>>>{\ti\Psi_{\tnT}\otimes\ti\Psi_{\tnB}\,\circ\,\id \otimes R\otimes \id} 
 \ar[ddl]|{\cI_{\tnT\tnL}\otimes\cI_{\tnB\tnL}\otimes\cI_{\tnT\tnR}\otimes\cI_{\tnB\tnR}}&&
\la(\ti{D}_{\tnT\tnM})\!\otimes\!\la(\ti{D}_{\tnB\tnM})
\ar[ddl]|{~~\cI_{\tnT\tnM}\otimes\cI_{\tnB\tnM}} \ar[dddd]^{\ti\Psi_{\tnM}}\\ 
\\
\la(D_{\tnT\tnL})\!\otimes\!\la(D_{\tnB\tnL})\!\otimes\!\la(D_{\tnT\tnR})\!\otimes\!\la(D_{\tnB\tnR})
 \ar[dddd]|{~~A_{\tnL}\Psi_{\tnL}\otimes\Psi_{\tnR}} 
\ar[rr]^>>>>>>>>>>>>>>>>>>>>{A_{\tnT}
\Psi_{\tnT}\otimes\Psi_{\tnB}\,\circ\,\id \otimes R\otimes \id} 
&\ar@{-->}[dd] 
&\la(D_{\tnT\tnM})\!\otimes\!\la(D_{\tnB\tnM})
\ar[dddd]|>>>>>>>>>>>>>>>>{A_{\tnR}\Psi_{\tnM}}\\ 
\\
&\la(\ti{D}_{\tnC\tnL})\!\otimes\!\la(\ti{D}_{\tnC\tnR})
\ar@{--}[r]^{\ti\Psi_{\tnC}} 
\ar@{-->}[ddl]|{~~\cI_{\tnC\tnL}\otimes\cI_{\tnC\tnR}} 
& \ar@{-->}[r]& \la(\ti{D}_{\tnC\tnM}) \ar[ddl]|{\cI_{\tnC\tnM}}\\
\\
\la(D_{\tnC\tnL})\!\otimes\!\la(D_{\tnC\tnR})\ar[rr]^{A_{\tnB}\Psi_{\tnC}}
&& \la(D_{\tnC\tnM})}
\end{gather*}
\caption{The cube of commutative diagrams used in the proof of Corollary~\ref{ExSQ_crl},
where $\ti{D}_{\star*}\!=\!(D_{\star*})_{\Th_{\star*}}$,
$\cI_{\star*}\!=\!\cI_{\Th_{\star*};D_{\star*}}$, and
$\ti\Psi_{\star}$ are the isomorphisms~\eref{ETisom_e}
corresponding to the \sf{t}op, \sf{c}enter, and \sf{b}ottom rows and
\sf{l}eft, \sf{m}iddle, and \sf{r}ight columns of 
the regularized version of the diagram~\eref{SQexact_e1} described in the proof}
\label{CubeDiag_fig}
\end{sidewaysfigure}

\noindent
We  denote the range of the operator $D_{\star*}$ by~$Y_{\star*}$.
Let
$$\Th_{\tnT\tnL}\!:\R^{N_{\tnT\tnL}}\lra Y_{\tnT\tnL}, \quad
\ti\Th_{\tnT\tnR}\!:\R^{N_{\tnT\tnR}}\lra Y_{\tnT\tnM},\quad
\ti\Th_{\tnB\tnL}\!:\R^{N_{\tnB\tnL}}\lra Y_{\tnC\tnL}, \quad
\ti\Th_{\tnB\tnR}\!:\R^{N_{\tnB\tnR}}\lra Y_{\tnC\tnM},$$
be homomorphisms such that
\BE{regulcond_e}\fc\big((D_{\tnT\tnL})_{\Th_{\tnT\tnL}}\big),
\fc\big((D_{\tnT\tnR})_{\fJ_{\tnT}\circ\ti\Th_{\tnT\tnR}}\big),
\fc\big((D_{\tnB\tnL})_{\fJ_{\tnL}\circ\ti\Th_{\tnB\tnL}}\big),
\fc\big((D_{\tnB\tnR})_{\fJ_{\tnR}\circ\fJ_{\tnC}\circ\ti\Th_{\tnB\tnR}}\big)=0.\EE
Let 
\begin{gather*}
N_{\tnT\tnM}=N_{\tnT\tnL}+N_{\tnT\tnR}, \quad
N_{\tnC\tnL}=N_{\tnT\tnL}+N_{\tnB\tnL}, \quad
N_{\tnC\tnR}=N_{\tnT\tnR}+N_{\tnB\tnR}, \quad
N_{\tnB\tnM}=N_{\tnB\tnL}+N_{\tnB\tnR}, \\
N_{\tnC\tnM}=N_{\tnC\tnL}+N_{\tnC\tnR}=N_{\tnT\tnM}+N_{\tnB\tnM}.
\end{gather*}
We define $\Th_{\star*}\!:\R^{N_{\star*}}\!\lra\!Y_{\star*}$ 
for $(\star,*)\in\{\tnT,\tnC,\tnB\}\!\times\!\{\tnL,\tnM,\tnR\}-\{(\tnT,\tnL)\}$ by 
\begin{gather*}
\Th_{\tnT\tnR}=\fJ_{\tnT}\circ\ti\Th_{\tnT\tnR}, \quad
\Th_{\tnB\tnL}=\fJ_{\tnL}\circ\ti\Th_{\tnB\tnL}, \quad
\Th_{\tnB\tnR}=\fJ_{\tnR}\circ\fJ_{\tnC}\circ\ti\Th_{\tnB\tnR}
=\fJ_{\tnB}\circ\fJ_{\tnM}\circ\ti\Th_{\tnB\tnR},\\
\Th_{\tnT\tnM}(x_{\tnT\tnL},x_{\tnT\tnR})=
\fI_{\tnT}\big(\Th_{\tnT\tnL}(x_{\tnT\tnL})\big)+\ti\Th_{\tnT\tnR}(x_{\tnT\tnR}),\quad
\Th_{\tnC\tnL}(x_{\tnT\tnL},x_{\tnB\tnL})=
\fI_{\tnL}\big(\Th_{\tnT\tnL}(x_{\tnT\tnL})\big)+\ti\Th_{\tnB\tnL}(x_{\tnB\tnL}),\\
\Th_{\tnC\tnR}(x_{\tnT\tnR},x_{\tnB\tnR})=
\fI_{\tnR}\big(\Th_{\tnT\tnR}(x_{\tnT\tnR})\big)+\fJ_{\tnC}(\ti\Th_{\tnB\tnR}(x_{\tnB\tnR})),\\
\quad \Th_{\tnB\tnM}(x_{\tnB\tnL},x_{\tnB\tnR})=
\fI_{\tnB}\big(\Th_{\tnB\tnL}(x_{\tnB\tnL})\big)+\fJ_{\tnM}(\ti\Th_{\tnB\tnR}(x_{\tnB\tnR})),\\
\Th_{\tnC\tnM}(x_{\tnT\tnL},x_{\tnT\tnR},x_{\tnB\tnL},x_{\tnB\tnR})
=\fI_{\tnM}(\Th_{\tnT\tnM}(x_{\tnT\tnL},x_{\tnT\tnR}))+
\fI_{\tnC}(\ti\Th_{\tnB\tnL}(x_{\tnB\tnL}))+\ti\Th_{\tnB\tnR}(x_{\tnB\tnR})
\end{gather*}
for all $x_{\star*}\!\in\!\R^{N_{\star*}}$ with 
$(\star,*)\!\in\!\{\tnT,\tnB\}\!\times\!\{\tnL,\tnR\}$.
For any $N'\!\le\!N$, we denote by $\fI\!:\R^{N'}\!\lra\!\R^N$
and $\fJ\!:\R^N\!\lra\!\R^{N'}$ the inclusion as $\R^{N'}\!\oplus\!0^{N-N'}$
and the projection onto the last $N'$ coordinates, respectively.
We also define
\begin{alignat*}{2}
\fI'\!:\R^{N_{\tnC\tnL}}&\lra \R^{N_{\tnC\tnM}}, &\qquad
&\fI'(x_{\tnT\tnL},x_{\tnB\tnL})=(x_{\tnT\tnL},0,x_{\tnB\tnL},0),\\
\fJ'\!:\R^{N_{\tnC\tnM}}&\lra \R^{N_{\tnC\tnR}}, &\qquad
&\fI'(x_{\tnT\tnL},x_{\tnT\tnR},x_{\tnB\tnL},x_{\tnB\tnR})=(x_{\tnT\tnR},x_{\tnB\tnR}),
\end{alignat*}
for all $x_{\star*}\!\in\!\R^{N_{\star*}}$.
In particular, the diagram in Figure~\ref{CubeDiag_fig2} commutes and its 6 rows and 6 columns
are exact.\\

\begin{sidewaysfigure}
\thisfloatpagestyle{empty}
\begin{gather*}
\xymatrix{&&&0\ar[dd]&&0\ar[dd]&&0\ar[dd]&&\\ 
&&0\ar[dd]&&0\ar[dd]&&0\ar[dd]&&&\\
&0\ar@{-->}'[r][rr]&& 
\R^{N_{\tnT\tnL}}\ar@{-->}'[d][dd]^{\fI}\ar@{--}[r]^{\fI}\ar@{.>}[dl]|{\Th_{\tnT\tnL}}&
\ar@{-->}[r]& 
\R^{N_{\tnT\tnM}}\ar@{-->}'[d][dd]^{\fI}\ar@{--}[r]^{\fJ} \ar@{.>}[dl]|{\Th_{\tnT\tnM}}&
\ar@{-->}[r]& 
\R^{N_{\tnT\tnR}}\ar@{-->}'[d][dd]^{\fI}\ar@{-->}[rr] \ar@{.>}[dl]|{\Th_{\tnT\tnR}}&&0\\
0\ar[rr]&& Y_{\tnT\tnL}\ar[dd]^>>>>>>>>>>>>>>>>{\fI_{\tnL}} \ar[rr]^>>>>>>>>>{\fI_{\tnT}}&& 
Y_{\tnT\tnM}\ar[dd]^>>>>>>>>>>>>>>>>{\fI_{\tnM}}\ar[rr]^>>>>>>>>>{\fJ_{\tnT}} && 
Y_{\tnT\tnR}\ar[dd]^>>>>>>>>>>>>>>>>{\fI_{\tnR}}\ar[rr]&&0\\
&0\ar@{-->}'[r][rr]&& 
\R^{N_{\tnC\tnL}}\ar@{-->}'[d][dd]^{\fJ}\ar@{--}[r]^{\fI'} \ar@{.>}[dl]|{\Th_{\tnC\tnL}}&
\ar@{-->}[r]& 
\R^{N_{\tnC\tnM}}\ar@{-->}'[d][dd]^{\fI}\ar@{--}[r]^{\fJ'} \ar@{.>}[dl]|{\Th_{\tnC\tnM}}&
\ar@{-->}[r]&
\R^{N_{\tnC\tnR}}\ar@{-->}'[d][dd]^{\fI}\ar@{-->}[rr] \ar@{.>}[dl]|{\Th_{\tnC\tnR}}&&0\\
0\ar[rr]&&Y_{\tnC\tnL}\ar[dd]^>>>>>>>>>>>>>>>>{\fJ_{\tnL}}\ar[rr]^>>>>>>>>>{\fI_{\tnC}}&& 
Y_{\tnC\tnM}\ar[dd]^>>>>>>>>>>>>>>>>{\fJ_{\tnM}}\ar[rr]^>>>>>>>>>{\fJ_{\tnC}}&& 
Y_{\tnC\tnR}\ar[dd]^>>>>>>>>>>>>>>>>{\fJ_{\tnR}}\ar[rr]&&0\\
&0\ar@{-->}'[r][rr]&& 
\R^{N_{\tnB\tnL}}\ar@{-->}'[d][dd]\ar@{--}[r]^{\fI}\ar@{.>}[dl]|{\Th_{\tnB\tnL}}&
\ar@{-->}[r]& 
\R^{N_{\tnB\tnM}}\ar@{-->}'[d][dd]\ar@{--}[r]^{\fJ}\ar@{.>}[dl]|{\Th_{\tnB\tnM}}&
\ar@{-->}[r]& 
\R^{N_{\tnB\tnR}}\ar@{-->}'[d][dd]\ar@{-->}[rr]\ar@{.>}[dl]|{\Th_{\tnB\tnR}}&&0\\
0\ar[rr]&& Y_{\tnB\tnL}\ar[dd]\ar[rr]^>>>>>>>>>{\fI_{\tnB}}&& 
Y_{\tnB\tnM}\ar[dd]\ar[rr]^>>>>>>>>>{\fJ_{\tnB}}&& Y_{\tnB\tnR}\ar[dd]\ar[rr]&&0\\
&&&0&&0&&0&&\\ &&0&&0&&0&&&}
\end{gather*}
\caption{The panel of commutative diagrams, with exact rows and columns,
used in the proof of Corollary~\ref{ExSQ_crl} to regularize the square
grid~\eref{SQexact_e1}}
\label{CubeDiag_fig2}
\end{sidewaysfigure}

\noindent
By the commutativity and exactness properties of the diagram in Figure~\ref{CubeDiag_fig2},
the diagram~\eref{SQexact_e1} with $D_{\star*}$ replaced~by
$\ti{D}_{\star*}\!\equiv\!(D_{\star*})_{\Th_{\star*}}$,
$\fI_{\tnC}\!:X_{\tnC\tnL}\!\lra\!X_{\tnC\tnM}$ 
and $\fJ_{\tnC}\!:X_{\tnC\tnM}\!\lra\!X_{\tnC\tnR}$
replaced~by
$$\fI_{\tnC}\!\oplus\!\fI':X_{\tnC\tnL}\!\oplus\!\R^{N_{\tnC\tnL}}\lra
X_{\tnC\tnM}\!\oplus\!\R^{N_{\tnC\tnM}}
\qquad\hbox{and}\qquad
\fJ_{\tnC}\!\oplus\!\fJ':X_{\tnC\tnM}\!\oplus\!\R^{N_{\tnC\tnM}}\lra
X_{\tnC\tnR}\!\oplus\!\R^{N_{\tnC\tnR}},$$
respectively, and the remaining homomorphisms $\fI_{\star}$ and $\fJ_{\star}$
on $X_{*\circ}$ by $\fI_{\star}\!\oplus\!\fI$ and $\fJ_{\star}\!\oplus\!\fJ$
on $X_{*\circ}\!\oplus\!\R^{N_{*\circ}}$, respectively,
still commutes and its 3 rows and 3 columns are still exact.
Thus, by \eref{regulcond_e}, the Normalization~II property, and Lemma~\ref{ses_lmm2}, 
the diagram on the back face of the cube in Figure~\ref{CubeDiag_fig} commutes.\\

\noindent
Let $\cI_{\star*}\!=\!\cI_{\Th_{\star*};D_{\star*}}$ be the isomorphisms defined
by~\eref{cIThD_e1}.
By the commutativity of the 3 pairs of exact rows and 3 pairs of exact columns
in Figure~\ref{CubeDiag_fig2} and Lemma~\ref{StabTrip_lmm},
the diagrams on the top, right, bottom, and left faces in Figure~\ref{CubeDiag_fig}
commute for some $A_{\tnT},A_{\tnR},A_{\tnB},A_{\tnL}\!\in\!\R^*$ determined
by $N_{\tnT\tnL},N_{\tnT\tnR},N_{\tnB\tnL},N_{\tnB\tnR}$
and the indices of the Fredholm operators~$D_{\star*}$. 
By Lemma~\ref{ses_lmm2},
\BE{signsum_e} A_{\tnT}A_{\tnR}\!=\!A_{\tnB}A_{\tnL}\EE
if $\fc(D_{\star*})\!=\!\{0\}$ for $(\star,*)\!\in\!\{\tnT,\tnB\}\!\times\!\{\tnL,\tnR\}$.
Thus, \eref{signsum_e} always holds, which establishes Corollary~\ref{ExSQ_crl} in all cases.
\end{proof}

\begin{rmk}\label{ExSQ_rmk}
For the collection of exact triple isomorphisms~\eref{ETisom_e} given by~\eref{cUDDdfn_e},
$A_{\star}\!=\!(-1)^{\ep_{\star}}$ with
\begin{gather*}
\ep_{\tnT}=N_{\tnT\tnR}N_{\tnB\tnL}+(\ind\,D_{\tnC\tnL})N_{\tnT\tnR}+
(\ind\,D_{\tnT\tnR})N_{\tnB\tnL}+(\ind\,D_{\tnB\tnL})N_{\tnB\tnR},\quad
\ep_{\tnR}=(\ind\,D_{\tnT\tnM})N_{\tnB\tnM}, \\
\ep_{\tnB}=N_{\tnT\tnR}N_{\tnB\tnL}+(\ind\,D_{\tnC\tnL})N_{\tnC\tnR}, \qquad
\ep_{\tnL}=(\ind\,D_{\tnT\tnL})N_{\tnB\tnL}+(\ind\,D_{\tnT\tnR})N_{\tnB\tnR}.
\end{gather*}
In this case, it can be checked directly that 
$\ep_{\tnT}\!+\!\ep_{\tnR}\!+\!\ep_{\tnB}\!+\!\ep_{\tnL}\!\in\!2\Z$.
\end{rmk}

\begin{rmk}\label{cIdual_rmk}
For the collection of exact triple isomorphisms~\eref{ETisom_e} given by~\eref{cUDDdfn_e},
$A_N\!=\!1$ in Lemma~\ref{cIdual_lmm}.
Thus, $A_N\!=\!A_{-N,N}$ in the case of the collection of isomorphisms~\eref{ETisom_e}
of Theorem~\ref{classify_thm} on page~\pageref{classify_thm} corresponding to
the collection~$(A_{i,c})_{i,c}$.
\end{rmk}

\section{Topology}
\label{mainthmpf_sec}

\noindent
It remains to topologize each set~${\det}_{X,Y}$ as a line bundle over~$\cF(X,Y)$
with good properties.
By Proposition~\ref{overlap_prp}, the Normalization~I property on page~\pageref{NormalI_prop} 
and the isomorphisms~\eref{cIThD_e2} determine a topology on~${\det}_{X,Y}$
if the collection of exact triple isomorphisms~$\Psi_{\ft}$ as in~\eref{ETisom_e}
satisfies the Normalization~II,III and Compositions~I,II properties. 
By Corollary~\ref{ETcont_crl}, all isomorphisms~$\Psi_{\ft}$ are continuous 
with respect to these topologies if this collection also satisfies 
the Naturality~II property.
By Corollary~\ref{compdual_crl},
a collection of dualization isomorphisms as in~\eref{wtcD_e}
satisfying the Normalization~IV$^*$ and Dual Exact Triples properties
consists of continuous isomorphisms with respect to 
these topologies.

\subsection{Continuity of overlap and exact triple maps}
\label{cont_subs}

\noindent
For Banach vector spaces $X,Y,X',Y',X'',Y''$, let
$$\cT^*(X,Y;X',Y';X'',Y'')\subset \cT(X,Y;X',Y';X'',Y'')$$
denote the subspace of short exact sequences as in~\eref{cTdiag_e} 
with surjective Fredholm operators $D,D',D''$.

\begin{lmm}\label{surET_lmm}
Let $X,Y,X',Y',X'',Y''$ be Banach vector spaces.
A family of exact triple isomorphisms~$\Psi_{\ft}$ as in~\eref{ETisom_e} satisfying 
the Normalization~II property induces a continuous bundle map
$$\Psi\!: \pi_{\tnL}^*{\det}_{X',Y'}\otimes \pi_{\tnR}^*{\det}_{X'',Y''}
\lra \pi_{\tnC}^*{\det}_{X,Y}$$
over $\cT^*(X,Y;X',Y';X'',Y'')$ with respect to the topologies determined by
the Normalization~I property.
\end{lmm}

\begin{proof}
We abbreviate $\cT^*(X,Y;X',Y';X'',Y'')$ as $\cT^*$.
Let $\ft_0\!\in\!\cT^*$ be as in~\eref{cTdiag_e},
with all seven homomorphisms carrying subscript~$0$, and
$T\!:Y\!\lra\!X$, $T'\!:Y'\!\lra\!X'$, and $T''\!:Y''\!\lra\!X''$ 
be right inverses for $D_0$, $D_0'$, and~$D_0''$, respectively.
For each $\ft$ as in~\eref{cTdiag_e} and $\star\!=\!',''$ or blank,  let
$$\Phi_{D_0^{\star};\ft}\!:\,\ka(D^{\star})\lra \ka(D_0^{\star}), \qquad
\Phi_{D_0^{\star};\ft}(x)=x\!+\!T^{\star}\big(\{D^{\star}\!-\!D^{\star}_0\}(x)\big),$$
be as in~\eref{Phidfn_e}; this map is an isomorphism if $\ft$ is sufficiently close to~$\ft_0$.
We need to show that the~map
\BE{ETcontmap_e}
\Psi_{\ft_0;\ft}\!:
\la(\ka(D_0'))\otimes \la(\ka(D_0'')) \lra \la(\ka(D_0))\EE
described by
$$\Psi_{\ft}\big(\la(\Phi_{D_0';\ft}^{-1})x'\!\otimes\!1^*\otimes
\la(\Phi_{D_0'';\ft}^{-1})x''\!\otimes\!1^*\big)
=\big(\la(\Phi_{D_0;\ft}^{-1})\Psi_{\ft_0;\ft}(x'\!\otimes\!x'')\big)\otimes1^*$$
depends continuously on $\ft\!\in\!\cT^*$ near~$\ft_0$.\\

\noindent
Let $x_1',\ldots,x_k'$ be a basis for $\ka(D_0')$ and $x_1,\ldots,x_{\ell}\!\in\!\ka(D_0)$
be such that $\fJ_{X;0}(x_1),\ldots,\fJ_{X;0}(x_{\ell})$ is a basis for
$\ka(D_0'')$.
If $\ft$ as in~\eref{cTdiag_e} is sufficiently close to~$\ft_0$, then
$\Phi_{D_0';\ft}^{-1}(x_1'),\ldots,\Phi_{D_0';\ft}^{-1}(x_k')$ is a basis
for $\ka(D')$ and 
$$\fJ_X\big(\Phi_{D_0;\ft}^{-1}(x_1)\big),\ldots,\fJ_X\big(\Phi_{D_0;\ft}^{-1}(x_{\ell})\big)
\in X''$$
is a basis for $\ka(D'')$.
In particular,
\begin{equation*}\begin{split}
&\fJ_X\big(\Phi_{D_0;\ft}^{-1}(x_1)\big)\!\w\!\ldots\!\w\!\fJ_X\big(\Phi_{D_0;\ft}^{-1}(x_{\ell})\big)
=f(\ft)\,\Phi_{D_0'';\ft}^{-1}\big(\fJ_{X;0}(x_1)\big)\!\w\!\ldots\!\w\!
\Phi_{D_0'';\ft}^{-1}\big(\fJ_{X;0}(x_{\ell})\big),\\
&\fI_X\big(\Phi_{D_0';\ft}^{-1}(x_1')\big)\!\w\!\ldots
\!\w\!\fI_X\big(\Phi_{D_0';\ft}^{-1}(x_k')\big) \w
\Phi_{D_0;\ft}^{-1}(x_1)\!\w\!\ldots\!\w\!\Phi_{D_0;\ft}^{-1}(x_{\ell})\\
&\hspace{.4in}
=g(\ft)\,\Phi_{D_0;\ft}^{-1}\big(\fI_{X;0}(x_1')\big)\!\w\!\ldots\!\w\!
\Phi_{D_0;\ft}^{-1}\big(\fI_{X;0}(x_k')\big)\w
\Phi_{D_0;\ft}^{-1}(x_1)\!\w\!\ldots\!\w\!\Phi_{D_0;\ft}^{-1}(x_{\ell}) 
\end{split}\end{equation*}
for some $\R^+$-valued continuous functions $f$ and~$g$. 
By the Normalization~II property, the homomorphism~\eref{ETcontmap_e} is then given~by
\begin{equation*}\begin{split}
\Psi_{\ft_0;\ft}\big((x_1'\!\w\!\ldots\!\w\!x_k')\!\otimes\!
(\fJ_{X;0}(x_1)\!\w\!\ldots\!\w\!\fJ_{X;0}(x_{\ell}))\big)
=\frac{g(\ft)}{f(\ft)}\,
\fI_{X;0}(x_1')\!\w\!\ldots\!\w\!\fI_{X;0}(x_k') \w x_1\!\w\!\ldots\!\w\!x_{\ell}
\end{split}\end{equation*}
and thus depends continuously on~$\ft$.
\end{proof}

\begin{crl}\label{surET_crl}
Let $(\Psi_{\ft})_{\ft}$ be a collection of exact triple isomorphisms~as 
in~\eref{ETisom_e} satisfying the Normalization~II,III and Compositions~II properties. 
For any Banach vector spaces~$X$ and~$Y$ and homomorphism $\Th\!:\R^N\!\lra\!Y$,
the induced map
$$\cI_{\Th}\!: \io_{\Th}^*{\det}_{X\oplus\R^N,Y}\lra{\det}_{X,Y}$$
as in~\eref{cIThD_e1} is continuous over $\cF^*(X,Y)$ 
with respect to the topologies determined by the Normalization~I property.
\end{crl}

\begin{proof}
By Lemma~\ref{cIdual_lmm}, $\cI_{\Th;D}\!=\!(-1)^{(\ind\,D)N}A_N\hat\cI_{\Th;D}^{-1}$,
with~$\hat\cI_{\Th;D}$ given by~\eref{cIThD_e2}.
By Lemma~\ref{surET_lmm}, the family of isomorphisms~$\hat\cI_{\Th;D}$
induce a continuous bundle map
$$\hat\cI_{\Th}\!:  {\det}_{X,Y}\lra\io_{\Th}^*{\det}_{X\oplus\R^N,Y}$$
over $\cF^*(X,Y)$.
This implies the claim.
\end{proof}

\noindent
Let $X$ and $Y$ be as above.
The subsets
$$U_{X;\Th}\equiv\big\{D\!\in\!\cF(X,Y)\!:\,\fc(D_{\Th})=0\big\}$$
form an open cover of $\cF(X,Y)$ as $\Th$ ranges over all homomorphisms $\R^N\!\lra\!Y$
and $N$ ranges over all nonnegative integers.
We topologize ${\det}_{X,Y}|_{U_{X;\Th}}$ by requiring the bundle isomorphism
$$\io_{\Th}^*{\det}_{X\oplus \R^N,Y}\lra {\det}_{X,Y}, \qquad
\si\lra \cI_{\Th;D}(\si)~~\forall~\si\!\in\!\la(D_{\Th}),~D\!\in\!\cF(X,Y),$$
to be a homeomorphism over $U_{X;\Th}$ with respect to the topology on the domain 
induced by the topology on ${\det}_{X\oplus \R^N,Y}|_{\cF^*(X\oplus \R^N,Y)}$ described
at the beginning of this section. 
We next show that this topology is well-defined.

\begin{prp}[Continuity of transition maps]\label{overlap_prp}
Let $(\Psi_{\ft})_{\ft}$ be a collection of exact triple isomorphisms~as 
in~\eref{ETisom_e} satisfying the Normalization~II,III and Compositions~I,II properties. 
For any Banach vector spaces~$X$ and~$Y$ and homomorphisms \hbox{$\Th_i\!:\R^{N_i}\!\lra\!Y$}
with $i\!=\!1,2$,
the induced overlap~map
\begin{gather*}
\cI_{\Th_2;D}^{-1}\!\circ\!\cI_{\Th_1;D}\!:
\io_{\Th_1}^*{\det}_{X\oplus\R^{N_1},Y}\lra \io_{\Th_2}^*{\det}_{X\oplus \R^{N_2},Y}
\end{gather*}
is continuous over $U_{X;\Th_1}\!\cap\!U_{X;\Th_2}$
with respect to the topologies determined by the Normalization~I property.
\end{prp}

\begin{proof} With $N\!\equiv\!N_1\!+\!N_2$, let
\begin{equation*}\begin{split}
&0\lra \R^{N_1}\stackrel{\io_1}{\lra}  \R^N\!=\!\R^{N_1}\!\oplus\!\R^{N_2}
\stackrel{\pi_{\tnR;N_2}}{\lra}\R^{N_2}\lra0\qquad\hbox{and}\\
&0\lra \R^{N_2}\stackrel{\io_2}{\lra}  \R^N\!=\!\R^{N_1}\!\oplus\!\R^{N_2}
\stackrel{\pi_{\tnL;N_1}}{\lra}\R^{N_1}\lra0
\end{split}\end{equation*}
be the natural exact sequences of vector spaces and 
$$\io_{k;X}=\id_X\!\oplus\!\io_k\!:\,X\!\oplus\!\R^{N_k}\lra X\!\oplus\!\R^N$$
for $k\!=\!1,2$.
Define
$$\Th\!:\,\R^N\lra Y \qquad\hbox{by}\quad \Th(u_1,u_2)=\Th_1(u_1)+\Th_2(u_2).$$
Thus, the diagram 
$$\xymatrix{& X\!\oplus\!\R^{N_1}\ar[rd]|{\io_{1;X}}\ar@/^/[rrd]|{D_{\Th_1}}&&\\
X\ar[ru]|{\io_{X;N_1}}\ar[rd]|{\io_{X;N_2}} \ar[rr]^{\io_{X;N}}&& 
X\!\oplus\!\R^N\ar[r]^{D_{\Th}}& Y\\
& X\!\oplus\!\R^{N_2}\ar[ru]|{\io_{2;X}}\ar@/_/[rru]|{D_{\Th_2}}&&}$$
of Fredholm operators commutes.\\

\noindent
By the Commutativity~I property, the diagram 
\BE{overlapBigDiag_e}\begin{split}
\xymatrix{&\la(D_{\Th_1})\ar@{..>}[rrrd]|<<<<<<<<<<<<<<{A_1}&&& \\
&\la(\io_{X;N_1})\otimes\la(\io_{1;X})\otimes\la(D_{\Th})
\ar[rrr]^<<<<<<<<<<<<<<{\id\otimes\wt\cC_{\io_{1;X},D_{\Th}}}
\ar[d]|{\wt\cC_{\io_{X;N_1},\io_{1;X}}\otimes\id}&&&
 \la(\io_{X;N_1})\otimes\la(D_{\Th_1})\ar[d]|{\wt\cC_{\io_{X;N_1},D_{\Th_1}}}\\
\la(D_{\Th}) \ar@{..>}[r]^<<<<<<<<<A \ar@{..>}@/^2pc/[uur]|{\cI_{\Th_2;D_{\Th_1}}}
\ar@{..>}@/_2pc/[ddr]|{\cI_{\Th_1;D_{\Th_2}}\circ\wt\cI_{R;D}}&
\la(\io_{X;N})\otimes\la(D_{\Th})\ar[rrr]^{\wt\cC_{\io_{X;N},D_{\Th}}}&&& \la(D)\\
& \la(\io_{X;N_2})\otimes\la(\io_{2;X})\otimes\la(D_{\Th})
\ar[rrr]^<<<<<<<<<<<<<<{\id\otimes\wt\cC_{\io_{2;X},D_{\Th}}}
\ar[u]|{\wt\cC_{\io_{X;N_2},\io_{2;X}}\otimes\id}&&&
 \la(\io_{X;N_2})\otimes\la(D_{\Th_2})\ar[u]|{\wt\cC_{\io_{X;N_2},D_{\Th_2}}}\\
&\la(D_{\Th_2})\ar@{..>}[rrru]|<<<<<<<<<<<<<<{A_2}&&& }
\end{split}\EE
also commutes (excluding the dotted arrows).
We define the isomorphisms $A,A_1,A_2$ in this diagram~by
\BE{Aprop_e}A(\si)=1\otimes(\OmN^*\!\circ\!\la(\pi_2))\otimes\si, \qquad
A_k(\si_k)=1\otimes(\Om_{N_k}^*\!\circ\!\la(\pi_2))\otimes\si_k,\EE
where $\pi_2\!:\fc(\io_{X;N_k})\!\lra\!\R^{N_k}$ is the isomorphism induced by 
the projection $X\!\oplus\!\R^{N_k}\lra\R^{N_k}$.
By~\eref{cIThD_e1},
\BE{overlap_e2b}
\cI_{\Th;D}=\wt\cC_{\io_{X;N},D_{\Th}}\circ A\,,\quad
\cI_{\Th_k;D}=\wt\cC_{\io_{X;N_k},D_{\Th_k}}\circ A_k\,,\quad k=1,2.\EE
Let $R\!:X\!\oplus\!\R^N\!\lra\!X\!\oplus\!\R^N$ be the isomorphism given by
$$R(x,u_1,u_2)=(x,u_2,u_1) \qquad 
\forall\,(x,u_1,u_2)\in X\!\oplus\!\R^{N_1}\!\oplus\!\R^{N_2}$$
and
\begin{equation*}\begin{split}
\wt\cI_{R;D}\!\equiv\!\wt\cI_{R,\id_Y;D_{\Th}}\!: 
\la(D_{\Th})&\lra \la(D_{\Th}\!\circ\!R^{-1}) =\la\big(\cI_{R,\id_Y}(D_{\Th})\big),\\
\wt\cI_{R;N_1}\!\equiv\!
\wt\cI_{\id_{X\oplus\R^{N_2}},R^{-1};\io_{X\oplus\R^{N_2};N_1}}\!:
\la(\io_{X\oplus\R^{N_2};N_1})&\lra 
\la\big(R^{-1}\!\circ\!\io_{X\oplus\R^{N_2};N_1}\big)=\la(\io_{2;X})
\end{split}\end{equation*}
be the corresponding isomorphisms as in~\eref{wtphipsi_e}.\\

\noindent
Let $C_1,C_2\!\in\!\R^*$ be such that
\begin{equation*}\begin{split}
\wt\cC_{\io_{0;N_1},\io_1}\!\big(1\!\otimes\!\Om_{N_1}^*\otimes
1\!\otimes\!(\Om_{N_2}^*\!\circ\!\la(\pi_{\tnR;N_2}))\big)
&=C_11\otimes\OmN^*,\\
\wt\cC_{\io_{0;N_2},\io_2}\!\big(1\!\otimes\!\Om_{N_2}^*\otimes
1\!\otimes\!(\Om_{N_1}^*\!\circ\!\la(\pi_{\tnL;N_1}))\big)
&=C_21\otimes\OmN^*,
\end{split}\end{equation*}
where $\pi_{\tnR;N_2}\!:\fc(\io_1)\!\lra\!\R^{N_2}$ and
$\pi_{\tnL;N_1}\!:\fc(\io_2)\!\lra\!\R^{N_1}$ are
are the isomorphisms induced by the projections~$\pi_{\tnR;N_2}$ and~$\pi_{\tnL;N_1}$.
By the Compositions~II and Normalization~III properties applied to the diagram
$$\xymatrix{0\ar[r]& X\ar[d]^{\id}\ar[r]^{\id}& 
X\ar[d]^{\io_{X;N_1}}\ar[r]& 0\ar[d]^{\io_{0;N_1}}\ar[r]&0\\
0\ar[r]& X\ar[d]^{\id}\ar[r]^>>>>>{\io_{X;N_1}}& 
X\!\oplus\!\R^{N_1}\ar[d]^{\io_{1;X}}\ar[r]^{\pi_2}& \R^{N_1}\ar[d]^{\io_1}\ar[r]&0\\
0\ar[r]& X\ar[r]^>>>>>>{\io_{X;N}}& X\!\oplus\!\R^N\ar[r]^{\pi_2}& \R^N\ar[r]&0\,,}$$
the diagram
$$\xymatrix{\la(\io_{0;N_1})\!\otimes\!\la(\io_1)\ar[d]^{\wt\cC_{\io_{0;N_1},\io_1}} 
 \ar[rrrr]^{\ti\cI_{\pi_2,\pi_2;\io_{X;N_1}}^{~-1}\otimes\ti\cI_{\pi_2,\pi_2;\io_{1;X}}^{~-1}} 
&&&&\la(\io_{X;N_1})\!\otimes\!\la(\io_{1;X})
\ar[d]_{\wt\cC_{\io_{X;N_1},\io_{1:X}}}\\
\la(\io_{0;N}) \ar[rrrr]^{\ti\cI_{\pi_2,\pi_2;\io_{X;N}}^{~-1}} &&&& \la(\io_{X;N})}$$
commutes.
Thus,
\BE{overlap_e4}
\wt\cC_{\io_{X;N_1},\io_{1;X}}
\big(1\!\otimes\!(\Om_{N_1}^*\!\circ\!\la(\pi_2)) \otimes
1\!\otimes\!(\Om_{N_2}^*\!\circ\!\la(\pi_{\tnR;N_2;X}))\big)
=C_1 1\otimes(\OmN^*\!\circ\!\la(\pi_2))\,,\EE
where $\pi_{\tnR;N_2;X}\!:\fc(\io_{1;X})\!\lra\!\R^{N_2}$ is the isomorphism induced by
the projection $X\!\oplus\!\R^N\!\lra\!\R^{N_2}$ onto the last $N_2$ Euclidean coordinates. 
Similarly,
\BE{overlap_e4b}
\wt\cC_{\io_{X;N_2},\io_{2;X}}
\big(1\!\otimes\!(\Om_{N_2}^*\!\circ\!\la(\pi_2)) \otimes
1\!\otimes\!(\Om_{N_1}^*\!\circ\!\la(\pi_{\tnL;N_1;X}))\big)
=C_2\,1\otimes(\OmN^*\!\circ\!\la(\pi_2)),\EE
where $\pi_{\tnL;N_1;X}\!:\fc(\io_{2;X})\!\lra\!\R^{N_1}$ is the isomorphism induced by
the projection $X\!\oplus\!\R^N\!\lra\!\R^{N_1}$ onto the first $N_1$ Euclidean coordinates.\\

\noindent
Since $(D_{\Th_1})_{\Th_2}\!=\!D_{\Th}$ and $\io_{X\oplus\R^{N_1};N_2}\!=\!\io_{1;X}$, 
\begin{equation*}\begin{split}
&\big\{\id\!\otimes\!\wt\cC_{\io_{1;X;},D_{\Th}}\big\}^{-1}
\big(A_1(\cI_{\Th_2;D_{\Th_1}}(\si))\big)\\
&\qquad=\big\{\id\!\otimes\!\wt\cC_{\io_{1;X;},D_{\Th}}\big\}^{-1}
\big(1\!\otimes\!(\Om_{N_1}^*\!\circ\!\la(\pi_2))\otimes
\wt\cC_{\io_{1;X;},D_{\Th}}
(1\!\otimes\!(\Om_{N_2}^*\!\circ\!\la(\pi_{\tnR;N_2})) \otimes\si)\big)\\
&\qquad=1\!\otimes\!(\Om_{N_1}^*\!\circ\!\la(\pi_2)) \otimes
1\!\otimes\!(\Om_{N_2}^*\!\circ\!\la(\pi_{\tnR;N_2}))  \otimes\si.
\end{split}\end{equation*}
Combining this with \eref{overlap_e2b}, 
the commutativity of the upper rectangle in~\eref{overlapBigDiag_e},
and~\eref{overlap_e4}, we obtain
\BE{overlap_e1}
\cI_{\Th;D}^{-1}\circ\cI_{\Th_1;D}=C_1\,\cI_{\Th_2;D_{\Th_1}}^{-1}\,.\EE

\vspace{.15in}

\noindent
On the other hand, $(D_{\Th_2})_{\Th_1}\!=\!D_{\Th}\!\circ\!R^{-1}$ and
$\io_{X\oplus\R^{N_2};N_1}\!=\!R\!\circ\!\io_{2;X}$.
By the Composition~I property applied to  
$D_{\Th}\!\circ\!R^{-1}\!\circ\!\io_{X\oplus\R^{N_2};N_1}$
and the Normalization~III property, the~diagram
$$\xymatrix{\la(\io_{X\oplus\R^{N_2};N_1})\otimes\la(D_{\Th})
\ar[rrrr]^<<<<<<<<<<<<<<<<<<<<<<{\wt\cI_{R;N_1}\otimes\id}
\ar[d]|{~~\id\otimes\wt\cI_{R;D}}
&&&& \la(\io_{2;X})\otimes\la(D_{\Th})
\ar[d]|{\wt\cC_{\io_{2;X},D_{\Th}}}\\
\la(\io_{X\oplus\R^{N_2};N_1})\otimes\la(D_{\Th}\!\circ\!R^{-1})
\ar[rrrr]^<<<<<<<<<<<<<<<<<<<<<{\wt\cC_{\io_{X\oplus\R^{N_2};N_1},D_{\Th}\circ R^{-1}}}&&&&
\la(D_{\Th_2})}$$
commutes.
Since
$$\wt\cI_{R;N_1}\big(1\!\otimes\!(\Om_{N_1}^*\!\circ\!\la(\pi_{\tnR;N_1}))\big)
=1\otimes(\Om_{N_1}^*\!\circ\!\la(\pi_{\tnL;N_1}))\,,$$
the last commutative diagram gives
\begin{equation*}\begin{split}
&\big\{\id\!\otimes\!\wt\cC_{\io_{2;X;},D_{\Th}}\big\}^{-1}
\big(A_2\big(\cI_{\Th_1;D_{\Th_2}}(\wt\cI_{R;D}(\si))\big)\big)
=1\!\otimes\!(\Om_{N_2}^*\!\circ\!\la(\pi_2)) \otimes
1\!\otimes\!(\Om_{N_1}^*\!\circ\!\la(\pi_{\tnL;N_1}))  \otimes\si\,.
\end{split}\end{equation*}
Combining this with \eref{overlap_e2b}, the commutativity of the lower rectangle 
in~\eref{overlapBigDiag_e}, and~\eref{overlap_e4b}, 
we obtain
\BE{overlap_e2}
\cI_{\Th;D}^{-1}\circ\cI_{\Th_2;D}
=C_2\,\wt\cI_{R;D}^{-1}\circ\cI_{\Th_1;D_{\Th_2}}^{-1}\,.\EE
From~\eref{overlap_e1} and~\eref{overlap_e2}, we conclude that 
$$\cI_{\Th_2;D}^{-1}\circ\cI_{\Th_1;D} =
(C_1/C_2)\,\cI_{\Th_1;D_{\Th_2}}\circ\wt\cI_{R;D}\circ \cI_{\Th_2;D_{\Th_1}}^{-1}.$$
The outer maps on the right-hand side above are continuous over 
$U_{X;\Th_1}\!\cap\!U_{X;\Th_2}$ by Corollary~\ref{surET_crl},
while the middle map is continuous over $U_{X;\Th_1}\!\cap\!U_{X;\Th_2}$ by 
Lemma~\ref{surET_lmm}.
\end{proof}

\begin{crl}[Continuity of~\eref{ETisom_e}]\label{ETcont_crl}
Let $(\Psi_{\ft})_{\ft}$ be a collection of exact triple isomorphisms~as 
in~\eref{ETisom_e} satisfying the Naturality~II, Normalization~II,III, 
and Compositions~I,II properties.
For any Banach vector spaces $X,Y,X',Y',X'',Y''$, the bundle map
$$\Psi\!:  \pi_{\tnL}^*{\det}_{X',Y'}\otimes \pi_{\tnR}^*{\det}_{X'',Y''}
\lra \pi_{\tnC}^*{\det}_{X,Y}$$
over $\cT(X,Y;X',Y';X'',Y'')$ is continuous with respect to the topologies
provided by Proposition~\ref{overlap_prp}.
\end{crl}

\begin{proof} We abbreviate $\cT(X,Y;X',Y';X'',Y'')$ as $\cT$.
Let $\ft_0\!\in\!\cT$ be as in~\eref{cTdiag_e},
with all seven homomorphisms carrying subscript~$0$, and
$$\Th'\!:\R^{N'}\lra Y' \qquad\hbox{and}\qquad  \ti\Th''\!:\R^{N''}\lra Y$$ 
be homomorphisms such that $D_0'\!\in\!U_{X';\Th'}$ and
$D_0''\!\in\!U_{X'';\fJ_{Y;0}\circ\ti\Th''}$.
Let $N\!=\!N'\!+\!N''$, $\fI\!:\R^{N'}\!\lra\!\R^N$
be the inclusion as $\R^{N'}\!\times\!0^{N''}$, and
$\fJ\!:\R^N\lra\R^{N''}$ be the projection onto the last $N''$ coordinates,
and $A_{i,j}\!\in\!\R^*$ be as in~\eref{Asdfn_e}.
For each $\ft\!\in\!\cT$ as in~\eref{cTdiag_e}, define 
\begin{alignat*}{3}
\Th_{\ft}\!:\R^N&\lra X, &\qquad 
\Th_{\ft}(x',x'')&=\fI_Y\big(\Th'(x')\big)+\ti\Th''(x'')
&\quad &\forall\,(x',x'')\in\R^{N'}\!\oplus\!\R^{N''}\,,\\
\Th_{\ft}''\!:\R^{N''}&\lra X'', &\qquad 
\Th_{\ft}''(x'')&=\fJ_Y\big(\ti\Th''(x'')\big)
&\quad &\forall\,x''\in\!\R^{N''}\,.
\end{alignat*}
Thus, the diagram~$\fs(\ft)$ given~by
$$\xymatrix{ 
0\ar[r]& \R^{N'} \ar[r]^{\fI} \ar[d]^{\Th'}& \R^N \ar[r]^{\fJ}\ar[d]^{\Th_{\ft}} 
& \R^{N''} \ar[r]\ar[d]|{\Th_{\ft}''} & 0\\
0\ar[r]& Y' \ar[r]^{\fI_Y}& Y \ar[r]^{\fJ_Y}& Y'\ar[r]& 0}$$
commutes for every exact triple $\ft$ as in~\eref{cTdiag_e}, and we obtain
an embedding
$$\cT\lra 
\cT(X\!\oplus\!\R^N,Y;X'\!\oplus\!\R^{N'},Y';X''\!\oplus\!\R^{N''},Y''),
\qquad \ft\lra\ft_{\fs(\ft)}.$$
By Lemma~\ref{StabTrip_lmm}, the diagram
\begin{equation*}\begin{split}
\xymatrix{\la(D_{\Th'}')\otimes\la(D_{\Th_{\ft}''}'') 
\ar[rrrrr]^{\Psi_{\ft_{\fs(\ft)}}}
\ar[d]|{\cI_{\Th';D'}\otimes\cI_{\Th_{\ft}'';D''}} &&&&&
\la(D_{\Th_{\ft}})  \ar[d]^{\cI_{\Th_{\ft};D}} \\
\la(D')\otimes \la(D'') 
\ar[rrrrr]^>>>>>>>>>>>>>>>>>>>>>>>>>>{(-1)^{(\ind\,D')N''}A_{i,j}\Psi_{\ft}}&&&&& \la(D)}
\end{split}\end{equation*}
commutes.
By the definition of the topologies on the determinant line bundles, the vertical arrows 
in the above diagram induce continuous line-bundle isomorphisms over the open subset 
of~$\cT$ consisting of the exact triples~$\ft$ as in~\eref{cTdiag_e}
so that $D'\!\in\!U_{X';\Th'}$ and $D''\!\in\!U_{X'';\Th_{\ft}''}$.
By Lemma~\ref{surET_lmm}, the top arrow  
induces a continuous line-bundle isomorphism over the same open subset.
Thus, the bottom arrow in this diagram induces a continuous 
line-bundle isomorphism as~well.
\end{proof}

\begin{rmk}\label{cont_rmk}
For the collection of exact triple isomorphisms~\eref{ETisom_e} given by~\eref{cUDDdfn_e},
$C_1\!=\!(-1)^{N_1N_2}$ and $C_2\!=\!1$ in the proof of Proposition~\ref{overlap_prp},
 as can be seen from~\eref{cUDDdfn_e2}.
\end{rmk}

\subsection{Continuity of dualization isomorphisms}
\label{DualProp_subs}

\noindent
We begin by verifying the Normalization~I$'$ property on page~\pageref{NormalIpr_prop};
see Lemma~\ref{NormalIdual_lmm}.
This allows us to confirm the continuity of~\eref{cDdfn_e0} over~$\cF^*(X,Y)$;
see Lemma~\ref{surDual_lmm}.
The continuity of~\eref{wtcD_e} over $\cF(X,Y)$
then follows from the Dual Exact Triples property
on page~\pageref{compdual_prop}; see the proof of Corollary~\ref{compdual_crl}.
For each $D\!\in\!\cF(X,Y)$, let
$$q_D\!:Y\lra\fc(D), \qquad q_D(y)= y+\Im\,D,$$
be the projection map as before.

\begin{lmm}[Normalization~I$'$]\label{NormalIdual_lmm}
Let $(\Psi_{\ft})_{\ft}$ be a collection of exact triple isomorphisms~as 
in~\eref{ETisom_e} satisfying the Naturality~II, Normalization~II,III, 
and Compositions~I,II properties.
For any Banach vector spaces~$X$ and~$Y$, $D_0\!\in\!\cF'(X,Y)$, 
and right inverse $S\!:\fc(D_0)\!\lra\!Y$ for~$q_{D_0}$,
there exists a neighborhood~$U_{D_0,S}$ of~$D_0$ in~$\cF'(X,Y)$
so that the bundle isomorphism~\eref{cIdual_e2} is well-defined and continuous
with respect to the topology provided by Proposition~\ref{overlap_prp}.
\end{lmm}

\begin{proof}
By the Open Mapping Theorem for Banach vector spaces,
$$U_{D_0,S}\equiv\big\{D\!\in\!\cF'(X,Y)\!:\,Y=\Im\,D\oplus\Im\,S\big\}$$
is an open neighborhood of $D_0$ in $\cF'(X,Y)$.
Let
$$\pi_X,\pi_S\!: Y=\Im\,D_0\oplus\Im\,S\lra \Im\,D_0,\,\Im\,S$$
be the projection maps and $D_0^{-1}\!:\Im\,D_0\!\lra\!X$ be the inverse of the isomorphism
$$D_0\!:X\lra\Im\,D_0\,,\qquad x\lra D_0x.$$
For each $D\!\in\!U_{D_0,S}$, the~map
$$\psi_{D_0;D}\!: Y\lra Y, \qquad 
\psi_{D_0;D}(y)=D\big(D_0^{-1}(\pi_X(y))\!\big)+\pi_S(y),$$
is an isomorphism  so that 
$$D\!=\!\psi_{D_0;D}\!\circ\!D_0\!\circ\!\id_X^{-1} 
\qquad\hbox{and}\qquad \psi_{D_0;D}(y)\!-\!S(q_{D_0}(y))\in\Im\,D~~\forall\,y\!\in\!Y.$$
By the last property, 
$$\wt\cI_{D_0,S;D}\!=\!\wt\cI_{\id_X,\psi_{D_0;D}}\!: \la(D_0)\lra \la(D).$$
Since $\psi_{D_0;D}$ depends continuously on~$D$,
the claim follows from the continuity of~\eref{ETisom_e} provided by 
Corollary~\ref{ETcont_crl}, along with the Normalization~III property.
\end{proof}

\begin{lmm}\label{surDual_lmm}
Let $(\Psi_{\ft})_{\ft}$ be a collection of exact triple isomorphisms~as 
in~\eref{ETisom_e} satisfying the Naturality~II, Normalization~II,III, 
and Compositions~I,II properties.
For any Banach vector spaces~$X$ and~$Y$,
the family of maps~$\wt\cD_D$ given by~\eref{cDdfn_e0} induces a continuous bundle map
$$\wt\cD\!: {\det}_{X,Y} \lra \cD^*{\det}_{Y^*,X^*} $$
over $\cF^*(X,Y)$ with respect to the topologies provided by Proposition~\ref{overlap_prp}.
\end{lmm}

\begin{proof} 
Let $D_0\!\in\!\cF^*(X,Y)$, $T\!:Y\!\lra\!X$ be a right inverse for~$D_0$, and
$$\pi_T\!: X=\ka(D_0)\oplus\Im\,T\lra\ka(D_0),\qquad x\lra x-TD_0x~~\forall\,x\!\in\!X,$$
be the projection map. 
Thus, the homomorphism
$$S\!:\fc(D_0^*)\lra X^*, \qquad \al+\Im\,D_0^*\lra \al|_{\ka(D_0)}\circ\pi_T\,,$$
is a right inverse for~$q_{D_0^*}$.
By the Normalization~I property on page~\pageref{NormalI_prop} and Lemma~\ref{NormalIdual_lmm},
it is sufficient to show that the~map
$$\wt\cI_{D_0^*,S;D^*}^{-1}\!\circ\!\wt\cD_D\!\circ\!\wt\cI_{D_0,T;D}\!:
\la(D_0)\lra \la(D)\lra\la(D^*)\lra \la(D_0^*)$$
depends continuously on $D\!\in\!U_{D_0,S}$.
By \eref{Phidfn_e}, \eref{cDdfn_e0}, and \eref{cIdual_e}, this map is given~by
$$x\!\otimes\!1^*\lra 1 \otimes\cP\big(\la(\cD_{D_0})x\big),$$
which establishes the claim.
\end{proof}

\begin{crl}[Continuity of~\eref{wtcD_e}]\label{compdual_crl}
Let $(\Psi_{\ft})_{\ft}$ be a collection of exact triple isomorphisms~as 
in~\eref{ETisom_e} satisfying the Naturality~II, Normalization~II,III, 
and Compositions~I,II properties.
If a collection of bundle maps
$$\wt\cD\!: {\det}_{X,Y} \lra \cD^*{\det}_{Y^*,X^*} $$
over $\cF(X,Y)$ satisfies the Normalization~IV$^*$ and Dual Exact Triples
properties, 
then the maps in this collection are continuous with respect to the topologies 
provided by Proposition~\ref{overlap_prp}.
\end{crl}

\begin{proof} Let $X,Y$ be Banach vector spaces.
Let $D\!\in\!\cF(X,Y)$ and $\Th\!:\R^N\!\lra\!Y$ be a homomorphism so that 
$D\!\in\!U_{X;\Th}$.
By the Dual Exact Triples property for the commutative diagram
$$\xymatrix{0\ar[r]& X\ar[r]^>>>>>{\io_{X;N}}\ar[d]^D& X\oplus\R^N\ar[r]\ar[d]^{D_{\Th}}& 
\R^N\ar[r]\ar[d]^j&0\\
0\ar[r]& Y\ar[r]^{\id_Y}& Y\ar[r]& 0\ar[r]&0\,, }$$
the diagram 
$$\xymatrix{\la(D)\otimes\la(j)
 \ar[d]|{~~\wt\cD_j\otimes\wt\cD_D\circ R}\ar[rr]^>>>>>>>>>>>>>{\Psi_{\ft}}&& 
\la\big(D_{\Th}\big)\ar[d]^{\wt\cD_{D_{\Th}}}\\
\la(j^*)\otimes\la(D^*)\ar[rr]^>>>>>>>>>{\Psi_{\ft^*}}&& \la\big(D_{\Th}^*\big)}$$
commutes.
By Corollary~\ref{ETcont_crl}, the horizontal arrows in this diagram induce 
bundle maps that are continuous
with respect to the topologies provided by Proposition~\ref{overlap_prp}.
The right vertical arrow induces a continuous bundle map over~$U_{X;\Th}$ 
by Lemma~\ref{surDual_lmm}.
The isomorphisms~$R$ and~$\wt\cD_j$ on the left-hand side of 
this diagram do not depend on~$D$.
Thus, the isomorphisms~$\wt\cD_D$ also induce continuous bundle maps
over~$U_{X;\Th}$.
\end{proof}

\subsection{Orientations along paths}
\label{CrossNums_subs}

\noindent
Let $X,Y$ be Banach vector spaces.
We denote by 
$$\cF^{\star}(X,Y)\subset\cF^*(X,Y)\subset\cF(X,Y)$$
the subspace of isomorphisms between $X$ and~$Y$. 
If $D\!\in\!\cF^{\star}(X,Y)$ is an isomorphism, 
the element $1\!\otimes\!1^*$ of~$\la(D)$ determines an orientation on this line,
which we will call the \sf{canonical orientation} of~$\la(D)$.
Below we determine whether the extension of this orientation over a generic path
in~$\cF(X,Y)$ ending in~$\cF^{\star}(X,Y)$ restricts to the canonical orientation
over the endpoint as well; see Proposition~\ref{CrossNums_prp}.\\

\noindent
Let $D_t\!\in\!\cF(X,Y)$ with $t\!\in\!(-\de,\de)$ be a continuous path 
so that $D_t\!\in\!\cF^{\star}(X,Y)$ for $t\!\neq\!0$ and $\fd(\fc(D_0))\!=\!1$.
By the continuity of the index, this implies that $\fd(\ka(D_0))\!=\!1$ as well.
Choose 
$$x_0\in\ka(D)\!-\!\{0\}, \qquad  y_0\in Y\!-\!\Im\,D_0,$$ 
and a closed linear subspace $\dot{X}\!\subset\!X$ such that the~operator
$$\dot{D}_0\!:\dot{X}\lra\Im\,D_0, \qquad \dot{D}_0(x)=D_0x,$$
is an isomorphism.
Shrinking~$\de$ if necessary, we can assume~that 
$$Y=\Im\,D_t\!\oplus\!\R y_0 \qquad\forall\,t\!\in\!(-\de,\de).$$ 
Let $B\!:(-\de,\de)\!\lra\!\dot{X}$ and $f_t\!:(-\de,\de)\!\lra\!\R$ be continuous maps 
so~that
$$D_tx_0=D_t(B(t)\!)\!+\!f(t)y_0.$$
By our assumptions, $B(0)\!=\!0$ and $f^{-1}(0)\!=\!\{0\}$.
We will call $0\!\in\!(-\de,\de)$ a \sf{transverse degeneration} of
the path $(D_t)_{t\in(-\de,\de)}$ if $f$ changes sign at $t\!=\!0$.
This notion is independent of the choices of~$x_0,y_0,\dot{X}$.\\

\noindent
Continuing with the setup above, define
$$\Th\!:\R\lra Y, \qquad \Th(s)=sy_0.$$
Thus, $\ka((D_t)_{\Th})$ is generated by the element $(x_0\!-\!B(t),-f(t)\!)$
of $X\!\oplus\!\R$.
Since the operators~$(D_t)_{\Th}$ are surjective, the~map 
$$\ti\cI\!:(-\de,\de)\!\times\!\R\lra
\bigsqcup_{t\in(-\de,\de)}\hspace{-.15in}\la\big((D_t)_{\Th}), \qquad
\ti\cI(t,s)=\big(x_0\!-\!B(t),-f(t)\!\big)s\in\la\big((D_t)_{\Th}),$$
is a continuous line bundle isomorphism over $(-\de,\de)$ by the Normalization~I property
on page~\pageref{NormalI_prop}.
By the Normalization~II property on page~\pageref{NormalI_prop}, 
the isomorphisms~\eref{cIThD_e2} are given~by
\begin{gather*}
\hat\cI_{\Th;D_t}\!: \la(D_t)\lra \la\big((D_t)_{\Th}),\\
\hat\cI_{\Th;D_t}\big(f(t)1\!\otimes\!1^*\big)=(B(t)\!-\!x_0,f(t)\!)\!\otimes\!1^*=\ti\cI(t,-1)
\quad\hbox{if}~t\neq0.
\end{gather*}
Since the isomorphisms $\hat\cI_{\Th;D_t}$ extend over $t\!=\!0$ by 
the Exact Triples property on page~\pageref{ETisom_e},
it follows that the canonical orientations of~$\la(D_t)$ with $t\!\neq\!0$
do not extend across $t\!=\!0$ if $0\!\in\!(-\de,\de)$ is a transverse degeneration of
the path $(D_t)_{t\in(-\de,\de)}$.
This establishes the following.

\begin{prp}[Wall Crossing]\label{CrossNums_prp}
Suppose $X,Y$ are Banach vector spaces and $(D_t)_{t\in[0,1]}$ is a continuous path
in~$\cF(X,Y)$ so that $D_t$ is an isomorphism except for finitely many values of~$t$ in~$(0,1)$.
If all degenerations of the path $(D_t)_{t\in[0,1]}$ are transverse, then
the canonical orientations of~$\la(D_0)$ and~$\la(D_1)$ extend continuously to 
orientations of~$\la(D_t)$ over~$[0,1]$ if and only if
the number of the degenerations is even.
\end{prp}

\vspace{1in}

\noindent
{\it Department of Mathematics, SUNY Stony Brook, NY 11794-3651\\
azinger@math.stonybrook.edu}


\clearpage

\begin{thebibliography}{99}

\bibitem{BF} J.-M.~Bismut and D.~Freed, 
{\it The analysis of elliptic families~I: metrics and connections on determinant bundles},
Comm.~Math.~Phys.~106 (1986), no.~1, 159--176. 

\bibitem{BGS1} J.-M.~Bismut, H.~Gillet,  and C.~Soul\'e,  
{\it Analytic torsion and holomorphic determinant bundles.~I.~Bott-Chern forms 
and analytic torsion}, Comm.~Math.~Phys.~115 (1988), no.~1, 49--78.

\bibitem{BGS2} J.-M.~Bismut, H.~Gillet,  and C.~Soul\'e,  
{\it Analytic torsion and holomorphic determinant bundles.~II.~Direct 
images and Bott-Chern forms}, Comm.~Math.~Phys.~115 (1988), no.~1, 79--126.

\bibitem{BGS3} J.-M.~Bismut, H.~Gillet,  and C.~Soul\'e,  
{\it Analytic torsion and holomorphic determinant bundles.~III.~Quillen 
metrics on holomorphic determinants}, Comm.~Math.~Phys.~115 (1988), no.~2, 301--351.

\bibitem{Faltings} G.~Faltings, 
{\it Lectures on the Arithmetic Riemann-Roch Theorem}, 
Annals of Mathematics Studies~127, Princeton University Press,~1992.

\bibitem{GS1} H.~Gillet and C.~Soul\'e, 
{\it Classes caract\'eristiques en th\'eorie d'Arakelov},
C.~R.~Acad.~Sci.~Paris S\'er.~I, Math.~301 (1985), no.~9, 439--442. 

\bibitem{GS2} H.~Gillet and C.~Soul\'e, 
{\it Direct images of Hermitian holomorphic bundles},
Bull.~AMS, 15 (1986), no.~2, 209--212.

\bibitem{GS3} H.~Gillet and C.~Soul\'e, 
{\it An arithmetic Riemann-Roch theorem},
Invent.~Math.~110 (1992), no.~3, 473--543.

\bibitem{EES} T.~Ekholm, J.~Entyre, and M.~Sullivan
{\it Orientations in Legendrian Contact Homology and Exact 
Lagrangian Immersions}, Inter.~J.~Math.~16 (2005), no.~5, 453-532.

\bibitem{Huang} Y.-Z. Huang,
{\it Two-Dimensional Conformal Geometry and Vertex Operator Algebras},
Progress in Mathematics~148, Birkh\"auser 1995.

\bibitem{KM} F.~Knudsen and D.~Mumford, 
{\it The projectivity of the moduli space of stable curves, I: 
Preliminaries on ``det" and ``Div"}, Math.~Scand.~39 (1976), no.~1, 19--55.  

\bibitem{KrMr} P.~Kronheimer and T.~Mrowka,
{\it Monopoles and Three-Manifolds}, 
New Mathematical Monographs~10, Cambridge University Press, 2007.

\bibitem{MS} D.~McDuff and D.~Salamon, 
{\it $J$-Holomorphic Curves and Symplectic Topology}, 2nd Ed.,
AMS Colloquium Publications 52, 2012.

\bibitem{MW} D.~McDuff and K.~Wehrheim, 
{\it Smooth Kuranishi structures with trivial isotropy}, 
arXiv:1208.1340v3.

\bibitem{Quillen} D.~Quillen, 
{\it Determinants of Cauchy-Riemann operators on Riemann surfaces} (Russian), 
Funktsional.~Anal.~i Prilozhen.~19 (1985), no.~1, 37--41.

\bibitem{Salamon} D.~Salamon, 
{\it Notes on the universal determinant bundle}, preprint, February 2013.

\bibitem{Seidel} P.~Seidel, 
{\it Fukaya Categories and Picard-Lefschetz Theory},
ETH Lecture Notes~8, EMS 2008.


\bibitem{SouleBook} C.~Soul\'e,  
{\it Lectures on Arakelov geometry}, Cambridge Studies in Advanced Mathematics~33, 
Cambridge University Press,~1992.


\end{thebibliography}
\end{document}